\documentclass[
  oneside,            
  openany,            
  abstract=true,      
  titlepage=false,
  numbers=noenddot,
  DIV=12,             
  11pt,
]{scrreprt}
\usepackage{lmodern}          
\usepackage{mathtools}   

\usepackage[main=english,french]{babel}
\usepackage[utf8]{inputenc}
\usepackage[T1]{fontenc}

\usepackage[french=guillemets, english=british]{csquotes}

\usepackage{amsmath,amssymb}
\usepackage{amsfonts}
\usepackage{mathtools}
\usepackage{amssymb}
\usepackage{mathabx}
\usepackage{MnSymbol}
\usepackage{xargs}
\usepackage{enumitem}
\usepackage{tikz}
\usetikzlibrary{matrix,arrows,calc,external,positioning, svg.path, patterns, patterns.meta}
\usepackage{calc}
\usepackage[mode=multiuser,status=final,lang=english]{fixme} 
\fxsetup{theme=colorsig}

\usepackage{amsthm} 
\usepackage{thmtools}

\usepackage{needspace}
\usepackage{graphicx}

\usepackage{listings}
\usepackage{upquote}
\lstdefinelanguage{lean4}{
  morekeywords={theorem,lemma,def,structure,instance,variable,open,namespace,end,
    fun,let,have,show,by,intro,exact,refine,Prop,Type,Set},
  sensitive=true,
  morecomment=[l]{--},
  morecomment=[s]{/-}{-/},
  literate=
    {∀}{{$\forall$}}1 {∃}{{$\exists$}}1 {→}{{$\rightarrow$}}1 {↦}{{$\mapsto$}}1
    {≤}{{$\leq$}}1 {≥}{{$\geq$}}1 {≠}{{$\neq$}}1 {∧}{{$\wedge$}}1 {∨}{{$\vee$}}1
    {¬}{{$\neg$}}1 {∘}{{$\circ$}}1 {∈}{{$\in$}}1 {⊆}{{$\subseteq$}}1
    {∪}{{$\cup$}}1 {∩}{{$\cap$}}1 {⊔}{{$\sqcup$}}1 {×}{{$\times$}}1
    {ℕ}{{$\mathbb{N}$}}1 {ℚ}{{$\mathbb{Q}$}}1 {ℝ}{{$\mathbb{R}$}}1
    {α}{{$\alpha$}}1 {β}{{$\beta$}}1 {γ}{{$\gamma$}}1 {δ}{{$\delta$}}1
    {λ}{{$\lambda$}}1 {σ}{{$\sigma$}}1 {τ}{{$\tau$}}1 {ι}{{$\iota$}}1
    {ω}{{$\omega$}}1 {Σ}{{$\Sigma$}}1 {Π}{{$\Pi$}}1
    {⟨}{{$\langle$}}1 {⟩}{{$\rangle$}}1 {↑}{{$\uparrow$}}1 {≡}{{$\equiv$}}1,
}
\usepackage[original]{imakeidx}
\makeindex 
\makeindex[name=IdxNotion, title=Index]
\makeindex[name=IdxFig, title=List of Figures]
\makeindex[name=IdxSymb, title=Symbols]

\usepackage{hyperref}

\usepackage{zref-clever}
\zcsetup{cap=true}
\newcommand{\cref}[1]{\zcref{#1}}
\newcommand{\Cref}[1]{\zcref[S]{#1}}

\NewDocumentCommand{\newzctheorem}{mO{#1}m}{
  \newtheorem{#1}[sharedtheoremcounter]{#3}
    \AddToHook{env/#1/begin}{%
      \zcsetup{countertype={sharedtheoremcounter=#2}}}
}

\theoremstyle{plain} 
\newtheorem{maintheorem}{Main Theorem}
\newzctheorem{theorem}{Theorem}
\newzctheorem{proposition}{Proposition}
\newzctheorem{corollary}{Corollary}
\newzctheorem{lemma}{Lemma}
\newzctheorem{conjecture}{Conjecture}
\newzctheorem{fact}{Fact}
\newtheorem*{claim}{Claim}

\theoremstyle{definition}
\newzctheorem{definition}{Definition}	

\theoremstyle{remark}
\newzctheorem{remark}{Remark}
\newzctheorem{remarks}{Remarks}
\newzctheorem{example}{Example}

\newzctheorem{question} {Question}

\zcRefTypeSetup{maintheorem}{
  Name-sg = Main Theorem ,
  name-sg = main theorem ,
  Name-pl = Main Theorems ,
  name-pl = main theorems ,
}

\zcRefTypeSetup{fact}{
  Name-sg = Fact ,
  name-sg = fact ,
  Name-pl = Facts ,
  name-pl = facts ,
}

\zcRefTypeSetup{conjecture}{
  Name-sg = Conjecture ,
  name-sg = conjecture ,
  Name-pl = Conjectures ,
  name-pl = conjectures ,
}

\zcRefTypeSetup{question}{
  Name-sg = Question ,
  name-sg = question ,
  Name-pl = Questions ,
  name-pl = questions ,
}

\zcRefTypeSetup{remarks}{
  Name-sg = Remark ,
  name-sg = remark ,
  Name-pl = Remarks ,
  name-pl = remarks ,
}

\zcRefTypeSetup{item}{
    refbounds = {(,,,)} , 
    Name-sg=,
    name-sg=,
    Name-pl=,
    name-pl=,
    Name-sg-ab=,
    name-sg-ab=,
    Name-pl-ab=,
    name-pl-ab=
}

\zcRefTypeSetup{item}{
}

\newlist{descThm}{description}{1}
\setlist[descThm]{font=\normalfont}
\newcommand{\N}{\mathbb{N}}

\newcommand{\cN} {\N^\N} 
\newcommand{\cantor} {2^\N}

\newcommand{\sub}{\subseteq}

\newcommand{\Dee}[1]{\ensuremath{\mathbf{\Delta}^0_#1}}
\newcommand{\Pee}[1]{\ensuremath{\mathbf{\Pi}^0_#1}}

\newcommand{\rao} {\rightarrow}

\newcommand{\Lglra} {\Longleftrightarrow}

\newcommand{\wqo}{\textsc{wqo}}
\newcommand{\Wqo}{\textsc{Wqo}}

\newcommand{\bqo}{\textsc{bqo}}

\newcommand{\conc}{\mathord{{}^{\smallfrown}}} 

\DeclareMathOperator{\im}{\mathrm{im}} 
\DeclareMathOperator{\id}{\mathrm{id}} 
\DeclareMathOperator{\dom}{\mathrm{dom}}

\DeclareMathOperator{\base}{\mathord{\bigcup}}

\DeclareMathOperator{\CB}{CB}

\newcommand{\Q}{\mathbb{Q}}

\newcommand{\nsegm} {\not\sqsubseteq}
\newcommand{\segm} {\sqsubseteq}

\newcommand{\segs} {\sqsubset}

\newcommand{\shift}[1]{{}_{*}#1}

\newcommand{\iw}[1]{ (#1)^{\infty}}

\newcommand{\ie} {\text{\emph{i.e. }}}

\newcommand{\1}{{\mathbf{1}}}
\newcommand{\sC}{\mathsf{Scat}}

\newcommand{\gl}{\bigoplus}
\newcommand{\glbin}{\oplus}
\DeclareMathOperator{\pgl}{ptgl}
\newcommand{\fctwedge}{\textstyle{\bigvee}}

\newcommand{\cocenter}{co\-center}
\DeclareMathOperator{\tp}{tp}

\newcommand{\generator}[1]{\mathsf{G}_{#1}}
\newcommand{\centered}[1]{\mathsf{C}_{#1}}
\newcommand{\omegacentered}[1]{\mathsf{W}_{#1}}

\newcommand{\Minimalfct}[1]{\mathsf{k}_{#1}}
\newcommand{\Maximalfct}[1]{\mathsf{l}_{#1}}

\newcommand{\restr}[1]{\mathord{\upharpoonright}_{#1}} 
\newcommand{\corestr}[1]{\mathord{\downharpoonright}^{#1}} 

\newcommand{\ABSd}[2][]{
 \ifthenelse{\equal{#1}{undefined}}{\overline{\mathsf{Abs}}(#2)}
 {\overline{\mathsf{Abs}}_{#1}(#2)}} 

\newcommand{\ABS}[2][]{
 \ifthenelse{\equal{#1}{undefined}}{\mathsf{Abs}(#2)}
 {\mathsf{Abs}_{#1}(#2)}}

\newcommand{\FinGl}[1]{\textstyle\gl_{\mathsf{Fin}}(#1)}

\newcommandx{\set}[2][2 = undefined]{
  \ifthenelse{\equal{#2}{undefined}}{\{ #1 \}}{
    \{ #1 \mid #2 \}
  }
}

\newcommand{\Ellentuckspace}{[\N]^\N}

\newcommand{\functionsonbaire}{\mathcal{F}(\cN)}
\newcommand{\FG}{\mathrm{FG}}

\DeclareMathOperator{\C}{C}

\newcommand{\nbhd}[2]{N_{#1\restr{#2}}}

\newcommand{\ray}[2]{#1^{(#2)}}

\newcommand{\partit}{\mathsf{P}}
\newcommand{\partitn}[1]{\mathsf{P}^{\neq #1}}

\newenvironment{subproof}[1][\proofname]{%
  \begin{proof}[#1]%
}{%
  \end{proof}%
}

\FXRegisterAuthor{y}{ay}{Y}
\FXRegisterAuthor{r}{ar}{R}

\begin{document}

\title{A well-quasi-order\\ for continuous functions}

\author{Rapha\"el Carroy and Yann Pequignot}

%
%




\maketitle

\begin{abstract}
We prove that continuous reducibility is a well-quasi-order on the class of continuous functions between separable metrizable spaces with analytic zero-dimensional domain.
To achieve this, we define \emph{scattered} functions, which generalize scattered spaces, and describe exhaustively scattered continuous functions between zero-dimen\-sional separable metrizable spaces up to continuous equivalence.

%
%

\end{abstract}


\tableofcontents





\chapter{Introduction}

\begin{definition}\label{MainDef}
Given topological spaces $X,X', Y, Y'$,
we say that a function $f:X\rao Y$ \emph{continuously reduces}\index[IdxNotion]{continuous reduction!between functions}
to a function $g:X'\rao Y'$ if there are two continuous functions $\sigma:X\rao X'$
and $\tau:\im(g\circ\sigma)\rao\im(f)$ such that $f=\tau\circ g\circ\sigma$. In this case, we write $f\leq g$ and say that the pair $(\sigma,\tau)$ \emph{continuously reduces}, or simply \emph{reduces}, $f$ to $g$.
\end{definition}

\begin{center}
\begin{tikzpicture}[line width=1pt]

\def \Side{2}
\def \p{0.2}
\draw[->, dotted] (\p,0) -- ({\Side-\p},0) node[midway,below] {$\sigma$};

\draw[->, dotted] ({\Side-\p},\Side)-- (\p,\Side)  node[midway,above] {$\tau$};
\draw[->, thick] (0,\p) -- (0,{\Side-\p}) node[midway,right] {$f$};
\draw[->, thick] (\Side,\p) -- (\Side,{\Side-\p})  node[midway,left] {$g$};
\draw (0.5*\Side,0.5*\Side) node    {$\leq$};

\end{tikzpicture}
\end{center}

Continuous reducibility is a transitive and reflexive relation, namely a \emph{quasi-order}\index[IdxNotion]{quasi-order}.
As usual with quasi-orders, there is an induced equivalence relation: we write $f\equiv g$ when both $f\leq g$ and $g\leq f$ hold,
and we say that $f$ and $g$ are \emph{continuously equivalent}\index[IdxNotion]{continuous reducibility!induced equivalence relation}.
We also denote by $f<g$\index[IdxSymb]{0@{$f\leq g$, $f<g$, $f\equiv g$}} the relation $f\leq g$ and $f\not\geq g$, and we say that $f$ and $g$ are \emph{incomparable} \index[IdxNotion]{incomparable}
when both $f\not\leq g$ and $g\not\leq f$.

An \emph{antichain}\index[IdxNotion]{antichain} is a set of pairwise incomparable elements, and an \emph{infinite strictly descending chain} is a sequence $(f_n)_{n\in\N}$
satisfying $f_{n+1}<f_n$ for all $n\in\N$. A quasi-order is called a \emph{well-quasi-order}\index[IdxNotion]{well-quasi-order (\wqo{})}, or \wqo{}, if it has no infinite antichains and no infinite descending chains.

A topological space is \emph{Polish}\index[IdxNotion]{space!Polish} if it is separable and completely metrizable, and \emph{zero-dimensional}\index[IdxNotion]{space!zero-dimensional} if it has a basis consisting of \emph{clopen sets},
that is sets that are both open and closed. A topological space is \emph{analytic}\index[IdxNotion]{space!analytic} if it is a continuous image of a Polish space.

Our main result is the following theorem.

\begin{maintheorem}\label{MainTheorem1}
Continuous reducibility is a well-quasi-order on the class of continuous functions from an analytic zero-dimensional space to a separable metrizable space.
\end{maintheorem}

The analyticity hypothesis in \cref{MainTheorem1} is only used to show that all such continuous functions with uncountable image are equivalent to the identity function on the Cantor space. In fact, we prove

\begin{maintheorem}\label{MainTheorem2}
Continuous reducibility is a well-quasi-order on the class of continuous functions from a separable metrizable zero-dimensional space to a countable metrizable space.
\end{maintheorem}

The above theorem is proved following a dichotomy which is reminiscent of the case of spaces.
Recall that a space is \emph{scattered}\index[IdxNotion]{space!scattered}\index[IdxNotion]{scattered!space} if any of its non-empty subsets contains an isolated point. We suggest to make the following definition:
a function $f$ between topological spaces is \emph{scattered}\index[IdxNotion]{scattered!function}\index[IdxNotion]{function!scattered} if any non-empty subset of its domain contains a {non-empty} open set on which $f$ is constant.
Every continuous function with a scattered image is scattered, but scattered continuous functions may have a non scattered image (for example, any bijection $:\N\to\Q$) or domain (for example, a constant function on $\Q$).

While the space $\Q$ of rationals is universal for countable metrizable spaces, any non scattered metrizable space contains a copy of $\Q$.
We show analogous results for functions with $\id_\Q$, the identity on $\Q$, in the role of $\Q$, thereby establishing that, up to continuous equivalence,
$\id_\Q$ is the only non scattered continuous function from a metric space to a countable metric space. 

With this in mind, the main challenge in proving our \cref{MainTheorem2} concerns the scattered functions.
The result at the heart of our main theorem can then be stated as follows.

\begin{maintheorem}\label{MainTheorem3}
Continuous reducibility is a well-quasi-order on the class of scattered continuous functions from a zero-dimensional separable metrizable space to a metrizable space.
\end{maintheorem}

In particular, we therefore provide a definitive positive answer to the question asked in \cite[Question 5.5]{carroy2013quasi}.

\section{The context}
Our Main Theorem ties together two concepts, continuous reduction and well-quasi-orders.
The notion of reduction originally comes from computability theory; given sets $X$ and $Y$ we say that a relation $R\subseteq X^n$ for some $n\in\N$ \emph{reduces}\index[IdxNotion]{continuous reduction!between general relations} to
$S\subseteq Y^n$ if there is a map $\sigma:X\to Y$ satisfying $(x_1,\ldots,x_n)\in R$ if and only if $(\sigma(x_1),\ldots\sigma(x_n))\in S$.
Continuous and Borel reducibilities between definable (equivalence) relations on Polish and analytic spaces have been explored in a rich series of works,
see \cite{MR1791302,MottoRosSurvey,MR3837073} for surveys on that matter, and the references contained therein\index[IdxNotion]{continuous reduction!between equivalence relations}.

Well-quasi-orders (\wqo{}s) are a natural generalization of well-orders in the context of partial quasi-orders. They too have a rich history and body of works,
see \cite{kruskal1972theory} for an early account of \wqo{} theory, and \cite{WQObook} for a recent book on \wqo{} theory.
Proving that a rich class of structures is a \wqo{} often requires an extensive analysis, as illustrated for instance by Robertson and Seymour's twenty papers on
the graph minor, culminating in their \wqo{} theorem in \cite{Robertson2004325}\index[IdxNotion]{well-quasi-order (\wqo{})}.
Our article presents similarities with another celebrated \wqo{} result: Laver's theorem \cite{laverfraisse} proving Fraïssé's conjecture that embeddability is \wqo{} on countable linear orders.
Much like Laver's work on linear orders, our result on continuous functions mostly consists of unfolding the properties of scattered objects.
Our result thus comes back to the origins of the concept of \enquote{scatteredness}, first studied by Cantor for spaces.
See \cite{MR0107806,MR0107849,MR0334150} for more on scattered spaces\index[IdxNotion]{space!scattered}\index[IdxNotion]{scattered!space}.

\Wqo{} results for continuous reductions are fairly rare, the first one is the Wadge-Martin-Monk result that continuous reducibility\index[IdxNotion]{continuous reduction!Wadge} is \wqo{} on Borel subsets
of the Baire space $\cN$ equipped with the product of the discrete topology (see \cite{wadge,cabal2} for the early works, and \cite{articolo_raphael,CMRSWadge} for recent accounts).
Continuous reduction on subsets, also known as Wadge theory, can be seen as continuous reduction between characteristic functions.
In this sense our \cref{MainTheorem1} is in the lineage of all these results.

Wadge theory has also sparked a lot of research, among which \cite{vEMSbqo} in the late 1980's tackles for the first time more than characteristic functions.
A consequence of that work -- not phrased this way at the time -- is that continuous reducibility is \wqo{} on Borel functions with finite image.
In \cite{vEMSbqo} is also considered the stronger notion of \emph{topological embeddability}\index[IdxNotion]{topological embeddability!between functions},
which is the quasi-order on functions obtained by requiring the reducing pair $(\sigma,\tau)$ in \cref{MainDef}
to be \emph{topological embeddings}, that is injective continuous functions whose inverse is also continuous.
They prove that $\Dee{2}$-measurable functions with finite image are \wqo{} under topological embeddability, and recently Hertling gave a description of this class of functions \cite{MR4133716}.

A number of finite basis results can be proven for topological embeddability between functions, see \cite{soleckidecomposing,pawlikowski2012decomposing,debs2012,carroymiller2}.
However,  in \cite{embeddabilityCPZ} with Zolt\'an Vidny\'anszky we manage to prove that topological embeddability between continuous functions defined on the Baire space,
as a quasi-order, reduces all analytic quasi-orders, and therefore is as far away from a \wqo{} as it can possibly be.

Starting in the 1990's and the work of Weihrauch, the computable version of continuous reducibility between functions has become subject of interest in computable analysis, see \cite{her-wei}.
In computable analysis, continuous reducibility between functions is known as \emph{topological strong Weihrauch reducibility}\index[IdxNotion]{topological strong Weihrauch reducibility}. We still choose the name continuous reducibility, not only
to underline the filiation with descriptive set-theoretical studies of definable reducibilities, but also because the objects studied in computable analysis are not functions but multi-functions, which changes radically the nature of the question.
Interestingly, Dzhafarov recently proved in \cite{MR4028100} that (computable) strong Weihrauch reducibility on multi-functions has a lattice structure that contains a copy of all
countable lattices, making it also very far from being a \wqo{}. See \cite{MR4300761} for a recent survey on Weihrauch redubility in computable analysis.

\section{The strategy}

The proof of \cref{MainTheorem1,MainTheorem2,MainTheorem3} takes most of the paper. We now briefly summarize our strategy, explaining the organization of the paper along the way.
First off, we adopt a strategy similar to Laver's in \cite{laverfraisse} and prove that continuous reducibility actually enjoys the stronger property introduced by Nash-Williams~\cite{nash1965well}: it is a better-quasi-order.
Here we briefly give a definition and use some results that are to be recalled later, but we will rely entirely on the existing literature on better-quasi-orders,
see for instance \cite{simpsonbqo,marcone1994foundations,yann2017towardsbetter,wellbetterinbetween}\index[IdxNotion]{better-quasi-order (\bqo{})}.

We denote by $\Ellentuckspace$ the \emph{Ellentuck space} of all infinite subsets of natural numbers.
We identify each element of the Ellentuck space with its increasing enumeration, making $\Ellentuckspace$ (homeomorphic to) a closed subset of $\cN$,
thus endowing it with a Polish topology. Given $Z\in\Ellentuckspace$ we denote by $\shift{Z}$ the \emph{shift} of $Z$, that is $Z\setminus\set{\min{Z}}$.
Given a quasi-order $(Q,\leq_Q)$, a \emph{$Q$-multisequence} is a locally constant map $\varphi:\Ellentuckspace\to Q$, we say that it is \emph{bad} if
for all $Z\in\Ellentuckspace$ we have $\varphi(Z)\not\leq_Q\varphi(\shift{Z})$. We say that $Q$ is a \emph{better-quasi-order}, or \bqo{}, if there are no bad $Q$-multisequences.
As the name suggests, every \bqo{} is in particular a \wqo{}. We aim for the following theorem, which is thus a strengthening of \cref{MainTheorem1,MainTheorem2}.

\begin{theorem}\label{Introthm:BQO}
Continuous reducibility is a better-quasi-order on the class of continuous functions $f:X\to Y$ from a zero-dimensional separable metrizable space $X$ to a metrizable space $Y$ such that either $X$ is analytic or $Y$ is countable. 
\end{theorem}

We start by characterizing scattered functions\index[IdxNotion]{scattered!function}\index[IdxNotion]{function!scattered} in a way that parallels the characterization of scattered spaces using isolated points (see \cite{SierpMazCompDen,Frechet1927,MR0107849}).

The set of points at which a function $f$ is locally constant is an open set of its domain, so the restriction of $f$ to the set of points on which $f$ is not locally constant defines a derivative
 that we call the \emph{Cantor--Bendixson derivative}\index[IdxNotion]{Cantor--Bendixson!derivative} of $f$ (see \cref{sectionCB} and \cite[Section 34.D]{kechris} for more on derivations).
The fixed point for this derivative is called the \emph{perfect kernel} of $f$, and the minimal ordinal index of the fixed point is the \emph{Cantor--Bendixson rank}\index[IdxNotion]{Cantor--Bendixson!rank}.

The restriction of $f$ to its perfect kernel is nowhere locally constant, this is why we also call the perfect kernel the \emph{nowhere locally constant part} of $f$.

\begin{theorem}\label{scatterediffemptykernel}
Suppose that $X$ is metrizable, $Y$ Hausdorff, and $f:X\to Y$ continuous.
Then $f$ is scattered if and only if it has empty perfect kernel if and only if $\id_\Q$ does not embed topologically in $f$.
\end{theorem}

We will see in \cref{scatterediffemptykernel_general} that, much like for Fréchet's result on scattered spaces \cite{Frechet1927},\index[IdxNotion]{space!scattered}\index[IdxNotion]{scattered!space}
the extra hypotheses on $X$, $Y$ and $f$ in \cref{scatterediffemptykernel} are only needed to prove that the third condition implies the first two.

We proceed with proving that there are, up to continuous equivalence, only two non-scattered continuous functions at stake in \cref{Introthm:BQO}:
$\id_\Q$ and $\id_{\cN}$ (\cref{FirststepforBQOthm}).
This allows us to focus next on scattered functions. More precisely, \cref{Introthm:BQO} follows from the following strengthening of \cref{MainTheorem3} (see \cref{FirststepforBQOthm}).

\begin{theorem}\label{Introthm:BQOonScat}
Continuous reducibility is a \bqo{} on the class of scattered continuous functions from a zero-dimensional separable metrizable space to a metrizable space.
\end{theorem}

We will see that the class of functions from \cref{Introthm:BQOonScat} can be represented, up to homeomorphism, as the class $\sC$ of scattered continuous functions $f:A\to B$ where $A,B$ are subsets of the Baire space $\cN$. 
For $\alpha\in\omega_1$, we denote by $\sC_\alpha$ the functions in $\sC$ that have Cantor--Bendixson rank exactly $\alpha$.
The proof of \cref{Introthm:BQOonScat} splits in two parts:
the first part consists of proving that if $\sC_\alpha$ is \bqo{} for all $\alpha<\omega_1$, then \cref{Introthm:BQOonScat} holds.
This is done in \cref{sectionPointedGluing} through a generalization to $\sC$ of the main theorem from \cite{carroy2013quasi}: the General Structure \cref{JSLgeneralstructure}.

The second part is the central novelty of this paper, it consists of proving by induction that continuous reducibility is \bqo{} on $\sC_\alpha$ for all $\alpha<\omega_1$;
it will be the goal of \cref{sectionCentered,PreciseStructureFinal,sectionDoubleSucc}.

In \cite{carroy2013quasi} the main results were proven for the class of continuous functions between Polish zero-dimensional spaces with countable image.
We dedicate \cref{subsectionGluing,sectionPointedGluing} to generalize and strengthen these results to more general classes of spaces and functions\footnote{In an effort to provide a self-contained account of the result, all proofs are included even when they are essentially the same as in \cite{carroy2013quasi}. Similarities are duly signaled.}.
For now let us just give a brief outline.

Our strategy for proving that continuous reducibility is \bqo{} on $\sC_\alpha$ is to define operations that allow to describe
every continuous equivalence class of functions in $\sC_\alpha$ from functions with smaller rank.
To do so, we consider an operation on continuous function that we call \emph{gluing}. It is the natural operation of sum on functions;
it roughly consists of the disjoint union on both domains and co-domains of a sequence of functions.
We use the term gluing to distinguish this operation from the disjoint union performed on domains only, that will also be used in this paper.

We say that a set $\mathcal{F}$ of functions is \emph{finitely generated} if there exists a finite set $\generator{}$ of functions such that each element of $\mathcal{F}$
is continuously equivalent to a finite gluing of elements of $\generator{}$.
It is easy to see that continuous reducibility is always \bqo{} on a finitely generated set (cf. \cref{SecondstepforBQOthm}). 
The main result of the present article can therefore be formulated as follows.

\begin{theorem}\label{Introthm:LevelsAreFinitelyGenerated}
For all $\alpha\in\omega_1$, the set $\sC_\alpha$ is finitely generated.
\end{theorem}

To prove \cref{Introthm:LevelsAreFinitelyGenerated} we will explicitly produce a finite generating set of functions for all $\alpha<\omega_1$ (\cref{subsectionGenerators});
this set is made of required functions obtained inductively from generators of smaller ranks using three operations:
the aforementioned gluing (\cref{subsectionGluing}), the Pointed Gluing (\cref{sectionPointedGluing}),
and the Wedge operation (\cref{subsectionWedge}).

To prove \cref{Introthm:LevelsAreFinitelyGenerated}, we proceed in two main steps.
For the first one, we prove new upper and lower bound results for the Pointed Gluing operation (\cref{Pgluingasupperbound,Pgluingaslowerbound});
these allow us to identify minimum functions in $\sC_{\leq \alpha}$ and maximum functions in $\sC_{\leq \alpha}$ (thus generalizing yet another result of \cite{carroy2013quasi}) in \cref{Maxfunctions,Minfunctions}.
We use this machinery to prove the (aforementioned) General Structure~\cref{JSLgeneralstructure}. That result takes care of $\sC_\lambda$ for all limit $\lambda$ simultaneously.

The second step is the main new feature, we prove finite generation of $\sC_{\lambda+n+1}$ for all limit $\lambda$ and $n\in\N$, uniformly on $\lambda$, by induction on $n$.
A first new ingredient is the identification of specific functions that we call \emph{centered} (\cref{sectionCentered}), we use the induction hypothesis to prove that:
in $\sC_{\lambda+n+1}$ all functions are locally centered (\cref{LocalCenterednessFromBQO}) and there are only finitely many centered functions (\cref{finitenessofcenteredfunctions}).
Using this first new ingredient and the Wedge operation (which is the second new ingredient) we treat separately the base case $\sC_{\lambda+1}$
(\cref{simplefunctionslambda+1damuddafuckaz,FGatsuccessoroflimit}) and the general one (\cref{FGatdoublesuccessors}).

 All the results of this article have been formalized in the Lean~4 proof assistant; the formalization shows in particular that the full strength of \bqo{} theory is not needed, a weakening called $2$-\bqo{} (\cite[Definition 4.1]{wellbetterinbetween} and \cite{yann2017towardsbetter}) suffices.
We describe it in \cref{chapterFormalization}, together with the formal statements of \cref{MainTheorem1,MainTheorem2,MainTheorem3}.

We conclude the article by discussing the optimality of our strategy, proving that the generators that we identify cannot be omitted (\cref{OptimalityOfGenerators}), as well as
future directions for research following two main angles: continuous reducibility on more generally definable functions on (general) Polish spaces (\cref{DirectionPolish}),
and functions with uncountable image on more general spaces (\cref{DirectionPerfect}). We give all along \cref{sectionConclusion} a healthy list of questions and conjectures.

\section*{Acknowledgments}
We want to thank the Paris ``Groupe de Travail en Théorie Descriptive des Ensembles'' for their valuable advice and their patience during numerous hours of seminars,
presenting various aspects of this work at various stages of development. We also want to thank Andrew Marks for inviting the first author to UCLA for a month in 2018,
giving us a unique opportunity to come up with a first version of the strategy for this proof.
We also want to thank Luca Motto Ros for suggesting a very early stage of the Vertical \cref{VerticalTheorem}.


\chapter{Some known and less known preliminaries}\label{sectionPrelim}

Given a function $f:X\to Y$, we use the following classical notations. We denote by $\dom f=X$ the domain of $f$ and refer to $Y$ as the codomain of $f$.
We write $\im f =\set{f(x)}[x\in X]$ for the image, or the range, of $f$. 
For $A\subseteq X$, $f\restr{A}:A\to Y$\index[IdxSymb]{f@{$f\restr{B}$}} stands for the restriction of $f$ to $A$, and $\id_{X}:X\to X$\index[IdxSymb]{id@{$\id_X$}} denotes the identity function on $X$.
For $B\subseteq Y$, $f^{-1}(B)=\set{x\in X}[f(x)\in B]$ denotes the preimage of $B$ by $f$, and we call \emph{co-restriction}\index[IdxNotion]{co-restriction} of $f$ to $B$ the function $f\restr{f^{-1}(B)}$, denoted by $f\corestr{B}$\index[IdxSymb]{f@{$f\corestr{B}$}}.
To improve readability, we often use the multiplicative notation $fg$ for the composition $f\circ g$ of $g$ with $f$.

When a pair $(\sigma,\tau)$ of continuous functions witnesses that a function $f$ reduces continuously to a functions $g$, we write $f\leq g$ and we say that $(\sigma,\tau)$ \emph{continuously reduces}, or simply \emph{reduces}, $f$ to $g$.
As usual with quasi-orders, there is an induced equivalence relation: we write $f\equiv g$ when both $f\leq g$ and $g\leq f$ hold,
and we say that $f$ and $g$ are \emph{continuously equivalent}.
We also denote by $f<g$ the relation $f\leq g$ and $f\not\geq g$, and we say that $f$ and $g$ are \emph{incomparable} when both $f\not\leq g$ and $g\not\leq f$. \index[IdxSymb]{0@{$f\leq g$, $f<g$, $f\equiv g$}}
An \emph{antichain} is a set of pairwise incomparable elements, and an \emph{infinite strictly descending chain} is a sequence $(f_n)_{n\in\N}$
satisfying $f_{n+1}<f_n$ for all $n\in\N$.


It follows from the definition of continuous reducibility that any function $f:X\to Y$ is continuously equivalent to the surjective function $f\corestr{\im f}:X\to \im f$,
where $\im f$ is endowed with the subspace topology\footnote{However, restricting our attention to surjective functions would not simplify proofs nor notations.
Indeed some functions like the centered functions from \cref{sectionCentered} are equivalent to their corestriction to arbitrary small subset of their image.}. 

\smallskip

About the existence of a reduction $(\sigma,\tau)$ between functions $f$ and $g$, it is useful to notice that to any given (potential) $\sigma$ can only correspond to a single $\tau$.
More precisely, $f\leq g$ if and only if there exists a continuous $\sigma:X\to X'$ such that the relation  
\[
T=\set{g\sigma(x), f(x))}[x\in X]
\]
is the graph of a continuous function $\tau: \im (g\sigma) \to \im(f)$.

While $\sigma$ fully determines $\tau$, a given candidate $\tau$ can potentially correspond to various $\sigma$.
Indeed, $f\leq g$ if and only if there exists a subset $B$ of $Y'$ and a continuous surjection $\tau:B\to \im(f)$ such that the relation 
\[
\set{(x,y)\in X\times X'}[g(y)\in B \text{ and } \tau g(y) =f(x)]
\]
admits a continuous uniformization (for more on uniformizations, see \cite[Section 18]{kechris}).

\medskip

We next show that continuous reducibility between functions respects measurability and generalizes both (coarse) Wadge reducibility and embeddability.
\index[IdxNotion]{continuous reduction!Wadge reducibility}\index[IdxNotion]{topological embeddability!between spaces}

Recall that for two topological spaces $X$ and $Y$, $A\subseteq X$ \emph{Wadge reduces to} $B\subseteq Y$ if there is a continuous function $\sigma:X\to Y$
such that $\sigma^{-1}(B)=A$, denoted by $A\leq^{X,Y}_WB$\index[IdxSymb]{2@{$\leq^{X,Y}_W$}}.
Given $A\subseteq X$ the notation $\1_{A,X}:X\to \{0,1\}$\index[IdxSymb]{one@{$\1_{A,X}$}} stands for the characteristic function of $A$ in $X$.

A continuous function $\sigma:X\to Y$ is a topological embedding (or simply an embedding) if and only if there is a continuous function $\tau:\im(\sigma)\to X$
such that $\tau\sigma=\id_X$, we also say that $(\sigma,\tau)$ \emph{witnesses} the embedding from $X$ to $Y$.

We sometimes write \enquote{reduces} and \enquote{embeds} instead of \enquote{continuously reduces} and \enquote{topologically embeds}.

Finally, a \emph{pointclass}\index[IdxNotion]{pointclass} is a functional class $\mathbf{\Gamma}$ 
mapping any topological space $X$ to a subset of the power set of $X$ with the property
that given any continuous function $\sigma:X\to Y$ if $A\in\mathbf{\Gamma}(Y)$ then $\sigma^{-1}(A)\in\mathbf{\Gamma}(X)$.
Observe that $\mathbf{\Gamma}(X)$ is closed under preimages by continuous functions from $X$ to $X$, for all spaces $X$.
A function $f:X\to Y$ is said \emph{$\mathbf{\Gamma}$-measurable} if the preimage of an open set of $Y$ by $f$ is in $\mathbf{\Gamma}(X)$.

\begin{proposition}\label{LinkswithWadgeEmbedMeas}
Let $X,X',Y$ and $Y'$ be topological spaces.
\begin{enumerate}
\item $X$ embeds in $Y$ if and only if $\id_X\leq\id_Y$.
\item If $A\subseteq X$ and $B\subseteq Y$ then $\1_{A,X}\leq\1_{B,Y}$ if and only if $A\leq^{X,Y}_WB$ or $A\leq^{X,Y}_WY\setminus B$.
\item Given a pointclass $\mathbf{\Gamma}$, if $f:X\to Y$ reduces to $g:X'\to Y'$ and $g$ is $\mathbf{\Gamma}$-measurable, then so is $f$.
\end{enumerate}
\end{proposition}

\begin{proof}
A space $X$ embeds in $Y$ if and only if a pair of maps $(\sigma,\tau)$ witnesses the embedding from $X$ to $Y$ if and only if $(\sigma,\tau)$ reduces $\id_X$ to $\id_Y$.

There are only two (continuous) functions $\set{0,1}\to\set{0,1}$, the identity and the function swapping $0$ and $1$.
So, given $A\subseteq X$ and $B\subseteq Y$, $(\sigma,\tau)$ reduces $\1_{A,X}$ to $\1_{B,Y}$ if and only if $A\leq^{X,Y}_WB$ (in case $\tau$ is the identity)
or $A\leq^{X,Y}_WY\setminus B$ (in case $\tau$ swaps $0$ and $1$).

Fix a continuous reduction $(\sigma,\tau)$ from $f$ to $g$, and assume that $U\subseteq Y$ is open. Note that we have $f^{-1}(U)=(g\sigma)^{-1}(\tau^{-1}(U))=\sigma^{-1}(g^{-1}(\tau^{-1}(U)))$. By continuity of $\tau$, $\tau^{-1}(U)$ is open in $\im (g\sigma)$ so there exists an open set $V$ in $Y'$ with $\tau^{-1}(U)=V\cap \im (g\sigma)$. Since $g$ is $\mathbf{\Gamma}$-measurable and $\sigma$ is continuous, it follows that $g\sigma:X\to Y'$ is $\mathbf{\Gamma}$-measurable. Therefore, we get that $f^{-1}(U)=(g\sigma)^{-1}(\tau^{-1}(U))=(g\sigma)^{-1}(V)\in \mathbf{\Gamma}(X)$, as desired.
\end{proof}

We use the following basic facts all along the paper without reference.

\begin{proposition}
Let $f$ be a function between topological spaces.
\begin{enumerate}
\item If $X\subseteq\dom f$ then $f\restr{X}\leq f$,
\item if $f$ is continuous and $\dom f$ embeds in $X$ then $f\leq\id_X$, and
\item if $f$ is continuous and $\im f$ embeds in $Y$ then $f\leq\id_Y$.
\end{enumerate}
\end{proposition}

\begin{proof}
The desired continuous reductions are:
$(\id_X, \id_{f(X)})$; $(\sigma, f\tau)$ with $(\sigma,\tau)$ witnessing the embedding from $\dom f$ into $X$;
and finally $(\sigma f, \tau)$ with $(\sigma,\tau)$ witnessing the embedding from $\im f$ into $Y$.
\end{proof}

Using the universal properties of $\Q$, $\cantor$ and $\cN$ \cite[(7.12),(7.8)]{kechris}, we have the following corollary.
\begin{corollary}\label{ConsequencesofUniversality}
Let $f:X\to Y$ be continuous.
\begin{enumerate}
\item If $\im f$ is countable and metrizable, then $f\leq \id_\Q$.
\item If $X$ is zero-dimensional separable metrizable, then $f\leq \id_{\cN}$ and $f\leq \id_{\cantor}$. 
\end{enumerate}
In particular, $\id_{2^\N}\equiv \id_{\cN}$.
\end{corollary}

Finally, recall from the introduction that when a reduction $(\sigma,\tau)$ from $f$ to $g$ consists of topological embeddings, we say that $f$ \emph{embeds topologically}\footnote{Observe that this actually coincides with   the following definition used in the articles cited in the introduction: $f$ embeds in $g$ if there is a pair $(\sigma, \tau)$ of embeddings such that $\tau f=g\sigma$. See \cite[Section 3]{embeddabilityCPZ} for more details.}
 in $g$, we denote it by $f\segm g$\index[IdxSymb]{0@{$f\segm g$}}.\index[IdxNotion]{topological embeddability!between functions}
Topological embeddability and continuous reducibility coincide on topological embeddings.

\begin{proposition}\label{ContRedonEmbed}
Let $f$ and $g$ be functions between topological spaces.
Assume that $(\sigma,\tau)$ reduces $f$ to $g$.

If $f$ is injective, then $\sigma$, $g\restr{\im \sigma}$ and $\tau$ are injective too.

If moreover $f$ is an embedding, then so is $g\restr{\im \sigma}$, and actually $(\sigma,\tau)$ embeds $f$ in $g$.
\end{proposition}

\begin{proof}
To see the first point, note that if any of those maps is not injective, nor is $f$ since $f=\tau g \sigma$ holds.

If now $f$ is an embedding, by definition there is an embedding $h$ such that $hf=\id_{\dom f}$, but since $hf= h\tau g\sigma=\id_{\dom f}$, we know that $\sigma$ is an embedding too.
Similarly $fh=\tau g\sigma h=\id_{\im f }$ and $\tau$ is an embedding. Finally we get $g\sigma h\tau=\id_{\im \sigma}$, which concludes.
\end{proof}

We say that a function $f:X\to Y$ is \emph{homeomorphic} to a function $f':X'\to Y'$ if $f\leq f'$ is witnessed by $(\sigma,\tau)$ where $\sigma:X\to X'$ and $\tau:\im f'\to \im f$ are both homeomorphisms.\index[IdxNotion]{homeomorphism between functions}
Note that in this special case, not only does $f$ embed in $f'$ and reduces to $f'$, but also the pair $(\sigma^{-1},\tau^{-1})$ is a homeomorphism from $f'$ to $f$. Let us denote by $\functionsonbaire$ the class of functions $f:A\to B$ where $A,B$ are subsets of the Baire space $\cN$.\index[IdxSymb]{F@{$\functionsonbaire$}}

\begin{proposition}\label{RepresentationforFunctions}
Any function between $0$-dimensional separable metrizable spaces is homeomorphic to a function in $\functionsonbaire$. 
So is any function from a $0$-dimensional separable metrizable space to a countable metrizable space. 
\end{proposition}

\begin{proof}
Because any $0$-dimen\-sional separable metrizable space is homeomorphic to a subset of $\cN$ and every countable metrizable space is homeomorphic to a subset of $\Q$, hence $0$-dimensional.
\end{proof}

\section{Non-scattered functions}

Recall that a topological space $X$ is \emph{scattered} if any of its non-empty subset has an isolated point.\index[IdxNotion]{space!scattered}\index[IdxNotion]{scattered!space}
For a metrizable space, $X$ is scattered if and only if $\Q$ does not embed into $X$ (see for instance \cite[Théorème 2]{MR0107849}, or \cite[Notes on (7.12)]{kechris}).
We show an analogous characterization for continuous functions.
\begin{definition}
We say that a function $f$ is \emph{scattered}\index[IdxNotion]{space!scattered}\index[IdxNotion]{scattered!space} if for every non-empty subset $A\subseteq \dom f$ the restriction $f\restr{A}$ is constant on some non-empty open subset of $A$.
\end{definition}
 Note that $f$ is not scattered if and only if it admits a non-empty restriction $f\restr{A}$ which is \emph{nowhere locally constant},
 that is such that $f\restr{U}$ is not constant for every non-empty open subset $U$ of $A$. 
Remark that  the identity on the space of rationals $\id_\Q$ is not scattered. For continuous functions between metrizable spaces, this function is actually the only obstruction to being scattered.

\begin{theorem}\label{prop:nlc_implies_nonscattered}
Let $f:X\to Y$ be a continuous function from a metrizable space to a Hausdorff space.
If $f$ is not scattered, then $\id_\Q\segm f$.  
\end{theorem}

\begin{proof}
If $f$ is not scattered, then it admits a non-empty restriction $f\restr{A}$ that is nowhere locally constant and $f\restr{A}\segm f$.
Hence, it suffices to show that if $f$ is nowhere locally constant, then $\id_\Q\segm  f$.

Let $d$ be a compatible metric for $X$.
We make the following observation: if $x\in X$, $\epsilon>0$ and $f(B(x,\epsilon))\subseteq U$ for some open set $U$ of $Y$,
then there exists $x'\in B(x,\epsilon)$, $\epsilon'<\epsilon$ and open sets $U_0,U_1\subseteq U$ such that $B(x,\epsilon'),B(x',\epsilon')\subseteq B(x,\epsilon)$,
$U_0\cap U_1=\emptyset$, $f(B(x,\epsilon'))\subseteq U_0$ and $f(B(x',\epsilon'))\subseteq U_1$.
To see this, note that since $f$ is nowhere locally constant there exists $x'\in B(x,\epsilon)$ such that $f(x)\neq f(x')$.
Since $Y$ is Hausdorff, there exists disjoint neighborhoods $U_0,U_1\subseteq U$ separating $f(x)$ from $f(x')$.
By continuity of $f$, we can find a small enough $\epsilon'<\epsilon$ such that $f(B(x,\epsilon'))\subseteq U_0$ and $f(B(x',\epsilon'))\subseteq U_1$ and $B(x,\epsilon'), B(x',\epsilon')\subseteq B(x,\epsilon)$.

We use this to define Cantor schemes\footnote{See \cite[Definition 6.1]{kechris} for a definition of \emph{Cantor scheme}.}
$(C_s)_{s\in 2^{<\N}}$ consisting of open balls $B(x_s, \rho_s)$ of $X$ and $(U_s)_{s\in 2^{<\N}}$ consisting of open sets of $Y$ such that for all $s\in 2^{<\N}$:
\begin{enumerate}
\item $f(C_{s\conc(i)})\subseteq U_{s\conc(i)}$ for $i=0,1$,
\item $\rho_s\leq 2^{-|s|}$.
\end{enumerate}
Choose $x_\emptyset\in X$, let $\rho_\emptyset=1$, set $C_\emptyset=B(x,1)$ and $U_\emptyset=Y$. Assume we have defined $C_s=B(x_s, \rho_s)$.
The preliminary observation allows us to find $x'$ and $\rho\leq 2^{-(|s|+1)}$ such that $B(x_s,\rho)$ and $B(x',\rho)$ are included in $C_s$ and whose images by $f$ are separated by open sets $U_0,U_1\subseteq U_s$.
We set $x_{s\conc(0)}=x_s$, $x_{s\conc(1)}=x'$, $\rho_{s\conc(0)}=\rho_{s\conc(1)}=\rho$ and $U_{s\conc (i)}=U_i$ for $i=0,1$. 

Note that by construction, we have $\{x_{s}\}= \bigcap_{n\in\N}C_{s\conc (0)^n}$ for every $s\in 2^{<\N}$.
Identifying $\Q$ with the subset of $\cantor$ of sequences that are eventually constant equal to $0$, our Cantor scheme induces a function $\sigma:\Q\to X$ which is an embedding by \cite[(7.6)]{kechris}. 
Moreover, $f\sigma$ is an embedding too. 
To see this, note that $f\sigma(\Q\cap N_s)=f\sigma(\Q)\cap U_s$ for all $s\in 2^{<\N}$.
Indeed if $x\in \Q\cap N_s$ then by construction $\sigma(x)\in C_s$ and so $f\sigma(x)\in U_s$.
Conversely, if $x\notin \Q\cap N_s$ then $x\in N_t$ for some $t\in 2^{<\N}$ incompatible with $s$ and so $f\sigma(x)\in U_t$ but then $f\sigma(x)\notin U_s$ since $U_s\cap U_t=\emptyset$.

Since $f\sigma$ is an embedding, its inverse $\tau:\im (f\sigma)\to \Q$, $f\sigma(x)\mapsto x$ is an embedding too and clearly $\id_\Q=\tau f \sigma$. Therefore $\id_\Q\segm f$, as desired.
\end{proof}

%

Continuity\footnote{Metrizability as well, see \cite[Footnote of Théorème 2]{MR0107849}.} is needed in \cref{prop:nlc_implies_nonscattered}
\begin{example}[A non-scattered function that does not reduce $\id_\Q$]\label{counterexamplenonscat}
Let $f$ be the characteristic function of a dense and co-dense subset of $\Q$.
Then $f$ is not scattered because it is nowhere locally constant, as its image on every basic open set is exactly $2$.
Being discontinuous and having finite image, it cannot continuously reduce $\id_\Q$.
\end{example}

For a finite basis result for all functions from $\Q$ to an analytic metric space with respect to closed topological embeddability, see \cite[Theorem 5.5]{carroymiller2}.

\medskip

Our next goal is to understand the continuous functions with uncountable range up to continuous reducibility.
For a continuous function $f$ with a Polish domain and separable metric image, the proof of \cref{prop:nlc_implies_nonscattered} can be adapted to show that if $f$ is not scattered,
then the constructed Cantor schemes are actually witnessing $\id_{\cantor}\segm f$.

We next show that we can extend this phenomenon to analytic domains and Borel functions,
for which we first establish equivalent conditions for the embeddability of $\id_{\cantor}$ into a function.
Recall that a function is \emph{Baire-measurable} if preimages of open sets have the Baire property\index[IdxNotion]{Baire-measurable},
and that $\Pee{2}$ sets are countable intersections of open sets.

\begin{proposition}\label{EquivEmbedCantor}
Let $f$ be a function between topological spaces. The following are equivalent:
\begin{enumerate}
\item There is an embedding $\sigma:\cantor\to\dom(f)$ such that $f\sigma$ is an embedding.
\item There is a Baire-measurable $\sigma:\cantor\to\dom(f)$ such that $f\sigma$ is an embedding.
\item There is $K\subseteq \dom(f)$ homeomorphic to $\cantor$ such that $f\restr{K}$ is both injective and Baire-measurable.
\end{enumerate}
\end{proposition}

\begin{proof}
Any embedding -- being continuous -- is Baire-measurable, so the first item implies the second one.

If $f\sigma$ is injective, then so is $\sigma$. Since $\sigma$ is moreover Baire-measurable on $\cantor$ 
then by \cite[(8.38)]{kechris} it must be continuous on a dense $\Pee{2}$ subset $G$ of $\cantor$.
In particular, $G$ is a perfect Polish space and so there exists an embedding $\sigma':\cantor \to G$ \cite[(6.2)]{kechris}.
Now $\sigma\sigma'$ is both injective and continuous, so $K=\im \sigma\sigma'$ is homeomorphic to $\cantor$ by compactness. Moreover, $f\restr{K}=(f\sigma) \sigma^{-1}$ is an embedding as a composition of embeddings. Hence $f\restr{K}$ is injective and Baire measurable, which shows that the second item implies the third one.

Finally, suppose there is $K\subseteq \dom(f)$ homeomorphic to $\cantor$ such that $f\restr{K}$ is both injective and Baire-measurable. By
\cite[(8.38)]{kechris} again, $f\restr{K}$ is continuous on a dense $\Pee{2}$ subset of $K$. So there exists an embedding $\sigma:\cantor\to K$
such that $f\sigma$ is both injective and continuous, hence an embedding by compactness
, and so the first item holds.
\end{proof}

We say that a class of functions has the \emph{Perfect Function Property}\index[IdxNotion]{Perfect Function Property} if every function in it either has countable image or $\id_{\cantor}$ embeds in $f$.
Observe that a class of spaces has the perfect set property if and only if the class of their identity functions has the perfect function property.

The following proposition has various proofs. It was first proven in \cite[Proposition 2.1]{carroy2013quasi} using Silver's Theorem and
 a game-theoretic proof was recently found by Lutz and Siskind \cite{lutz20231}.
 These proofs both use the third item from \cref{EquivEmbedCantor}; here we propose another proof, suggested to us by Zolt\'an Vidny\'anszky.

A subset $A$ of a topological space $X$ is \emph{Borel} if it is in the $\sigma$-algebra generated by the open sets,
and a function is \emph{Borel} if it is Borel-measurable. Note that zero-dimensionality is not used in the following result.

\begin{proposition}\label{uncountablerange}
The class of Borel functions from analytic to separable metrizable spaces has the Perfect Function Property.
\end{proposition}

\begin{proof}
Let $f:X\to Y$ be Borel with uncountable image. As $X$ is analytic we can take a continuous surjection $\phi:\cN\to X$.
Since $\im(f\phi)$ is analytic and its completion is an uncountable Polish space, the Perfect Set Theorem for analytic sets \cite[(29.1)]{kechris} ensures that there is an embedding $\tau:\cantor\to\im(f\phi)$.
The relation $xRx'$ if and only if $\tau(x)=f\phi(x')$ is analytic in $\cantor\times\cN$, so by the Jankov--von Neumann uniformization Theorem (\cite[(29.9)]{kechris})
it has a Baire-measurable uniformizing function $\sigma:\cantor\to\cN$, so $f\phi\sigma=\tau$ is an embedding, and this concludes by \cref{EquivEmbedCantor}.
\end{proof}

\section{Scattered functions}\label{sectionCB}

We generalize the Cantor--Bendixson analysis of spaces to functions, which was already initiated in~\cite{carroy2013quasi}. 

\index[IdxNotion]{Cantor--Bendixson!analysis}
Given a function $f$ and $x\in\dom f$, we say that $x$ is \textit{$f$-isolated} if $f^{-1}(\{f(x)\})$ is a neighbourhood of $x$, \ie $f$ is locally constant at $x$.
Note that a function $f$ is scattered if and only if every non-empty subset $A$ of $\dom f$ contains an $f\restr{A}$-isolated point.
Given $A\subseteq\dom f$, we let $I(f,A)$ stand for the set of all $f\restr{A}$-isolated points. It is an open subset of $A$.
We now define inductively the following closed sets of $\dom f$:\index[IdxNotion]{Cantor--Bendixson!derivative}
\begin{align*}
\CB_0(f)&=\dom f,\\
\CB_{\alpha+1}(f)&=\CB_\alpha(f)\backslash I(f,\CB_\alpha(f)),\\ 
\CB_\lambda(f)&=\bigcap_{\alpha<\lambda}\CB_\alpha(f)\mbox{ for }\lambda\mbox{ limit.}
\end{align*}
\index[IdxSymb]{C@{$\CB_\alpha(f)$}}\index[IdxSymb]{C@{$\CB(f)$}}
We set $\CB(f)=\min\set{\alpha}[\CB_\alpha(f)=\CB_{\alpha+1}(f)]$, and
we call it the \textit{Cantor--Bendixson rank} of $f$, or \emph{$\CB$-rank} for short.\index[IdxNotion]{Cantor--Bendixson!rank}
Recall from the introduction that $\ker_{\CB}{f}:=\CB_{\CB(f)}(f)$ is called the perfect kernel of $f$, and $f$ is nowhere locally constant on it.
In particular, functions with an empty perfect kernel play a central role in this paper as are exactly the scattered functions. 
Note that this equivalence stated in introduction as part of \cref{scatterediffemptykernel} actually holds for functions between topological spaces. \index[IdxNotion]{scattered!function}\index[IdxNotion]{function!scattered}

\begin{proposition}\label{scatterediffemptykernel_general}
Let $f$ and $g$ be two functions between topological spaces. Then $f$ is scattered if and only if $f$ has empty perfect kernel. 
\end{proposition}

\begin{proof}
If $f$ has non-empty perfect kernel $A$, then by definition of the Cantor--Bendixson derivative $f\restr{A}$ is nowhere locally constant, so $f$ is not scattered.
Conversely, if $f$ has empty perfect kernel, then for some ordinal $\alpha$ we have $\dom(f)=\bigcup_{\beta<\alpha}I(f,\CB_\beta(f))$,
so given any non-empty $A\subseteq\dom(f)$, setting $\gamma=\min\set{\beta<\alpha}[A\cap I(f,\CB_\beta(f))\neq\emptyset]$,
then $A\subseteq \CB_\gamma(f)$ and $f\restr{A}$ is locally constant at any $x\in A\cap I(f,\CB_\gamma(f))$, which shows that $f$ is scattered.
\end{proof}
The Cantor--Bendixson rank is defined using a closed derivative, so any function with a second countable domain has a countable $\CB$-rank\footnote{See \cite[Theorem 6.9]{kechris}}.
Since a locally constant function on a second countable space has countable range, it follows from \cref{scatterediffemptykernel_general} that any scattered function with second countable domain has countable image. 
In particular, every scattered continuous functions from a zero-dimensional separable metrizable space to a metrizable space, has countable range and therefore is homeomorphic to a function in $\functionsonbaire$ by \cref{RepresentationforFunctions}. With this in mind, we define the class of functions which is the main focus of this article as follows: 
{
\begin{definition}
The class $\sC$\index[IdxSymb]{Scat@{$\sC$}} denotes the class of functions $f\in \functionsonbaire$ which are continuous and scattered. 
\end{definition}
 }
We are now in a position to show that \cref{Introthm:BQO} reduces to \cref{Introthm:BQOonScat}.
Recall that both \wqo{} and \bqo{} are closed under finite unions, so if $Q$ and $Q'$ are \wqo{} (resp. \bqo{}) then so is $Q\cup Q'$.\index[IdxNotion]{better-quasi-order (\bqo{})}
Recall as well that quasi-orders whose induced equivalence relation has finitely many classes are \bqo{} (see \cite[Example 3.8 and Proposition 3.11]{wellbetterinbetween}).
As announced we have:

\begin{theorem}\label{FirststepforBQOthm} 
For every function $f$ that satisifies the hypothesis of \cref{Introthm:BQO}, one of the following holds: $f\equiv \id_\Q$, $f\equiv \id_{\cN}$ or $f$ is scattered.
In particular, \cref{Introthm:BQOonScat} implies \cref{Introthm:BQO}.
\end{theorem}

\begin{proof}
Let $f:X\to Y$ be a continuous function with $X$ zero-dimensional separable metrizable and $Y$ metrizable. If $f$ is scattered then being continuous with second countable domain it has countable range. So in light of \cref{RepresentationforFunctions}, there exists $\hat{f}\in\sC$ such that $f$ is homeomorphic to $\hat{f}$ and in particular $f\equiv \hat{f}$.

So suppose that $f$ is non scattered, so $\id_\Q\segm f$ by \cref{prop:nlc_implies_nonscattered}. If $\im f$ is furthemore countable, then actually $f\equiv \id_\Q$ by \cref{ConsequencesofUniversality}. Otherwise if $X$ is furthermore analytic, then $f$ has the Perfect Function Property by \cref{uncountablerange} and so $\id_{\cN}\segm f$. In this case, since $X$ is zero-dimen\-sional separable metrizable, then in fact $f\equiv \id_{\cN}$ \cref{ConsequencesofUniversality}.

Therefore, up to continuous equivalence, the class of functions described in \cref{Introthm:BQO} can be written as the union of $\sC$ together with the two equivalence classes of $\id_\Q$ or $\id_{\cN}$. Hence if $\sC$ is \bqo{} then the whole class is \bqo{}.
\end{proof}

%
%
%
Here are few basic facts about our closed derivative.

\begin{fact}\label{CBbasics0}
Let $f$ be a function between topological spaces.
\begin{enumerate}
\item For every open set $U\subseteq\dom(f)$ and every ordinal $\alpha$ we have $\CB_\alpha(f\restr{U})=\CB_\alpha(f)\cap U$. \label{CBbasicsfromJSL2}
\item If $f$ scattered and $\CB(f)$ is limit, then for all $x\in\dom(f)$ there is an open $U\ni x$ with $\CB(f\restr{U})<\CB(f)$.
\item If $f$ is scattered and has nonempty compact domain, then $\CB(f)$ is a successor ordinal $\alpha+1$ and $f(\CB_\alpha(f))$ is finite.
\end{enumerate}
\end{fact}

\begin{proof}
\begin{enumerate}
\item {By induction on $\alpha$, observing that} for $U$ open and $x\in U$ we have $f\restr{U}$ is locally constant at $x$ if and only if $f$ is locally constant at $x$. 
\item As $\ker_{\CB}{f}=\bigcap_{\alpha<\CB(f)}\CB_\alpha(f)=\emptyset$, for all $x\in \dom f$ there exists $\alpha<\CB(f)$
such that $x$ belongs to the open set $U=\dom f \setminus \CB_\alpha(f)$. By the first item, it follows that $\CB(f\restr{U})\leq \alpha$. 
\item If $f$ had a limit $\CB$-rank $\lambda$, then the decreasing sequence of closed sets $(\CB_\alpha(f))_{\alpha<\lambda}$ would have empty intersection,
a contradiction with the compactness of $\dom f$.
So $\CB(f)=\alpha+1$ and $\CB_{\alpha}(f)$ is compact -- as a closed set of $\dom f$. Since $\CB_{\alpha+1}(f)=\emptyset$,
the sets $f^{-1}(\set{y})\cap \CB_{\alpha}(f)$ for $y\in \im(f\restr{\CB_\alpha(f)})$, form a partition of $\CB_\alpha(f)$ in nonempty open sets,
so $f$ must have must have finite range on $\CB_\alpha(f)$.\qedhere
\end{enumerate}
\end{proof}

The following basic results on the interplay between continuous reducibility and the $\CB$-derivative on functions is proven in \cite[Proposition 2.2 and Corollary 2.3]{carroy2013quasi} for Polish spaces,
but the same proofs work for arbitrary functions between topological spaces.

\begin{proposition}\label{CBbasicsfromJSL}
Let $f$ and $g$ be two functions between topological spaces. 
\begin{enumerate}
\item If $f\leq g$ and $g$ is scattered, then $f$ is scattered too. \label{CBbasicsfromJSL0}
\item If $(\sigma,\tau)$ continuously reduces $f$ to $g$, then for all $\alpha$ we have $\sigma(\CB_\alpha(f))\subseteq\CB_\alpha(g)$. \label{CBbasicsfromJSL1}
\end{enumerate}
\end{proposition}
\begin{proof}
Let $(\sigma,\tau)$ continuously reduce $f$ to $g$ and $A\subseteq \dom f$. We show that for all $x\in A$, if $g\restr{\sigma(A)}$ is locally constant at $\sigma(x)$, then $f\restr{A}$ is locally constant at $x$. If $g\restr{\sigma(A)}$ is locally constant at $\sigma(x)$, then there exists $U\ni \sigma(x)$ open in $\dom g$ such that $g(U\cap \sigma(A))$ is a singleton. By continuity of $\sigma$ there exists an open set $V\ni x$ of $\dom f$ such that $\sigma(V)\subseteq U$.
Hence $g\sigma(V\cap A)$ is a singleton, and therefore $f=\tau g \sigma$ is constant on $V\cap A$. This directly implies the first item and by induction it yields the second item.
\end{proof}

We now have all the elements to prove \cref{scatterediffemptykernel} from the introduction:
a continuous function between metrizable spaces is scattered if and only if it has non-empty perfect kernel if and only if $\id_\Q$ does not embed in it.

\begin{proof}[Proof of \cref{scatterediffemptykernel}]
After \cref{scatterediffemptykernel_general}, note that $\id_\Q$ is not scattered and use \cref{prop:nlc_implies_nonscattered} and \cref{CBbasicsfromJSL}~\cref{CBbasicsfromJSL0}.
\end{proof}

Here is useful relation on the $\CB$-rank of a scattered function and the ranks of its restrictions to open sets. 

\begin{corollary}\label{CBrankofclopenunion}
Let $f$ be a scattered function and $(A_i)_{i\in I}$ be an open covering of $\dom(f)$ for some set $I$.
Then $\CB(f)=\sup_{i\in I}\CB(f\restr{A_i})$.
\end{corollary}

\begin{proof}
Set $\alpha=\sup_{i\in I}\CB(f\restr{A_i})$. Using \cref{CBbasics0}~\cref{CBbasicsfromJSL2}, for all $i\in I$, we have $\CB_\alpha(f\restr{A_i})=\emptyset$ and 
 $\CB(f\restr{A_i})=\CB(f)\cap A_i$, and so: 
\[\CB_\alpha(f)\subseteq\bigcup_{i\in I}\CB_\alpha(f)\cap A_i=\bigcup_{i\in I}\CB_\alpha(f\restr{A_i})=\emptyset,\]
and so $\CB(f)\leq\alpha$. Now if $\beta<\alpha$, then $\beta<\CB(f\restr{A_i})$ for some ${i\in I}$ and
\[\CB_\beta(f)\supseteq\CB_\beta(f)\cap A_i=\CB_\beta(f\restr{A_i})\neq\emptyset.\]
Since $\CB_{\CB(f)}(f)=\emptyset$, it follows that $\beta<\CB(f)$, hence $\alpha=\CB(f)$.
\end{proof}

Given a scattered function $f$, we define as follows the \emph{$\CB$-degree}\index[IdxNotion]{Cantor--Bendixson!degree} of $f$ that we denote by $N_f$. When $\CB(f)$ is limit or null, set $N_f=0$.
When $\CB(f)=\alpha+1$ for some $\alpha<\omega_1$, we let $N_f$\index[IdxSymb]{N@{$N_f$}} be the cardinality of $f(\CB_\alpha(f))$. 
We also call \emph{$\CB$-type} of $f$ the pair $\tp(f)=(\CB(f),N_f)$. \index[IdxNotion]{Cantor--Bendixson!type}\index[IdxSymb]{tp@{$\tp(f)$}}

We denote by $\leq_{lex}$\index[IdxSymb]{2@{$\leq_{lex}$}} the lexicographical order on sequences of ordinals. Here is an immediate application of \cref{CBbasicsfromJSL} (see \cite[Fact, p. 642]{carroy2013quasi}).

\begin{corollary}\label{CBbasicsfromJSL-tp}
For any scattered functions $f$ and $g$, $f\leq g$ implies $\tp(f)\leq_{lex}\tp(g)$.
\end{corollary}
\begin{proof}
Assume that $f\leq g$ as witnessed by $(\sigma,\tau)$. By \cref{CBbasicsfromJSL}~\cref{CBbasicsfromJSL1}, we have $\sigma(\CB_{\CB(g)(f)})\subseteq \ker_{\CB} g=\emptyset$ so $\CB(f)\leq \CB(g)$. If $\CB(f)<\CB(g)$, then $\tp(f)\leq_{lex}\tp(g)$. So assume that $\CB(f)=\CB(g)=\alpha$. If $\alpha$ is limit, then by definition $N_f=N_g=0$. Otherwise, $\alpha=\beta+1$ and $\sigma(\CB_\beta(f))\subseteq \CB_\beta(g)$ by \cref{CBbasicsfromJSL}~\cref{CBbasicsfromJSL1}. 
Hence for every $x\in \CB_\beta(f)$ we have $f(x)=\tau ( g\sigma(x))$ with $g\sigma(x)\in g(\CB_\beta(g))$,
which shows that $\tau$ maps surjectively a subset of $g(\CB_\beta(g))$ onto $f(\CB_\beta(f))$, so $N_g \geq N_f$ as desired.
\end{proof}

\section{Disjoint union}

\begin{definition}
Given an index set $I$ and a sequence $(f_i)_{i\in I}$ of functions between topological spaces,\index[IdxNotion]{operation!disjoint union}
we say that a function $f$ is the \emph{disjoint union} of the maps $f_i$ and we write $f=\bigsqcup_{i\in I} f_i$\index[IdxSymb]{u@{$\bigsqcup_if_i$}}
if $f=\bigcup_{i\in I}f_i$ and $(\dom f_i)_{i\in I}$ is a clopen partition of $\dom(f)$.
\end{definition}

Note that the disjoint union over the empty set is the empty function.
The previous definition makes sense for any set $I$, but in this paper $I$ will always be countable, so in general we assume that $I\subseteq \N$.
When the context is clear, we write $(f_i)_i$ instead of $(f_i)_{i\in I}$.

\smallskip

We say that an operation on a class $\mathcal{C}$ of objects -- such as sets, spaces, functions -- \emph{preserves} a property $\mathcal{F}$ if\index[IdxNotion]{operation!preserving a property}
the operation gives a result in $\mathcal{F}$ whenever the operands are in $\mathcal{F}$.
For instance, the following basic but useful criterion for continuity, similar to \cite[2.2.4 and 2.2.6]{engelking}, describes a sufficient condition for the disjoint union operation to preserve continuity.

\smallskip

We say that a family $(A_i)_{i\in I}$ of subsets of a topological space $X$ is a \emph{relative clopen partition}\index[IdxNotion]{relative clopen partition}\index[IdxNotion]{partition!relative clopen} in $X$
when $(A_i)_{i\in I}$ forms a clopen partition of the subspace $\bigcup_{i\in I} A_i$, or equivalently
if the sets $A_i$ are pairwise disjoint and relatively open in $\bigcup_iA_i$.
Note that a pairwise disjoint family $(A_i)_{i\in I}$ is a relative clopen partition exactly when the labelling function $c:\bigcup_i A_i \to I$,
$c(a)=i$ if and only if $a\in A_i$, is continuous when $I$ is endowed with the discrete topology.

\begin{lemma}\label{lem:ContUnion}
Let $X$ and $Y$ be topological spaces and $(f_i:A_i\to B_i)_{i\in I}$ be sequence of continuous functions with $A_i\subseteq X$ and $B_i\subseteq Y$ for all $i\in I$.
If $(A_i)_i$ is a relative clopen partition, then $f=\bigsqcup_{i\in I} f_i$ is continuous.
\end{lemma}
\begin{proof}
Let $Z=\bigcup_{i\in I} A_i$ and $U$ be open in $Y$. Note that
\[
f^{-1}(U)=\bigcup_{i\in I}\{ x\in A_i \mid f(x)\in U\}= \bigcup_{i\in I}f_i^{-1}(U),
\] 
so it is enough to show that $f_i^{-1}(U)$ is open in $Z$ for each $i\in I$. Since $f_i$ is continuous, there exists $U_i$ open in $X$ such that $f_i^{-1}(U)=A_i\cap U_i$. Moreover as $(A_i)_i$ is a relative clopen partition, for each $i\in I$ there exists $V_i$ open in $X$ such that $A_i=Z\cap V_i$. Therefore $f_i^{-1}(U)=Z\cap V_i\cap U_i$ is open in $Z$, as desired.
\end{proof}

%
%
A recurring strategy in the proof of our main theorem is to realize an arbitrary
scattered function as a disjoint union of functions of a few especially well-behaved
kinds: first \emph{simple} functions, which is later strengthened by \emph{centered}
functions. The difficulty is that these classes
are \emph{not} closed under restriction --- not even under restriction to a clopen
subset of the domain: cutting a simple, respectively centered, function along an
arbitrary clopen partition typically destroys simplicity, respectively centeredness.
Yet, we first establish a \emph{local} property --- namely that the function is locally simple,
respectively locally centered --- and then pass from the local property to a genuine
disjoint union. 

This passage is made possible for functions with zero-dimensional separable metrizable by selecting an ultrametric, which is always possible. 
In a \emph{separable ultrametric space}, the balls form a countable,
ranked, family of clopen sets, which are pairwise nested or disjoint. This structure  lets us select whole balls on which the function is
well-behaved, rather than cut it, which the classes do not tolerate, and assemble
the desired decomposition. \Cref{ultrametricdecomposition} makes this passage precise, and it
is the tool underlying both the decomposition into simple functions
(\cref{JSLdecompositionlemma}) and, later, the decomposition into centered functions.

\begin{definition}\label{def:locallyinF}
Let $(X,d)$ be a separable ultrametric space and $\mathcal F$ a class of functions.
A function $f:X\to Y$ is \emph{locally in $\mathcal F$}, or \emph{locally $\mathcal F$}
for short, if every point of $X$ has a ball on which the restriction of $f$ lies in
$\mathcal F$: for all $x\in X$ there is $n\in\N$ with $f\restr{B(x,2^{-n})}\in\mathcal F$.
\end{definition}

Any locally $\mathcal{F}$ function is a disjoint union of functions in $\mathcal{F}$: 

\begin{proposition}\label{ultrametricdecomposition}
Let $(X,d)$ be a separable ultrametric space, $\mathcal F$ a class of functions. If 
$f:X\to Y$ is locally $\mathcal F$, then $X$ admits a partition into countably many balls
$(V_m)_m$ with $f\restr{V_m}\in\mathcal F$ for all $m$. In particular,
$f=\bigsqcup_m f\restr{V_m}$ is a disjoint union of functions in $\mathcal F$.
\end{proposition}

\begin{proof}
Fix a countable dense set $\set{x_m}[m\in\N]$ in $X$. Since $f$ is locally
$\mathcal F$, for every $m$ the restriction $f\restr{B(x_m,2^{-n})}\in\mathcal F$ for some $n\in\N$ and we let $n_m$ be the smallest $n$ with this property (i.e. the largest ball). The sets $V_m=B(x_m,2^{-n_m})$ form a clopen partition of $X$ and the restriction of $f$ to any of these sets belongs to $\mathcal{F}$. To see that elements of $\{V_m\mid m\in \N\}$ are pairwise disjoint, recall that any two ultrametric balls are either nested or disjoint and that any point of a ball is a centre. If $V_m\subseteq V_{m'}$, then they must be equal since $V_{m'}=B(x_m, 2^{-n_{m'}})$ and so $n_{m'}\geq n_m$ by minimality.
To see they cover $X$, let $x\in X$ and let $k\in \N$ be such that $f\restr{B(x,2^{-k})}\in\mathcal F$. By density, we have $x_m\in B(x,2^{-k})$ for some $m$, but then we obtain similarly $x\in B(x,2^{-k})\subseteq B(x_m,2^{-n_m})=V_m$. 
\end{proof}

In the sequel, we only use this result for functions whose domain $X$ is a subset of $\cN$. In this case, $X$ inherits the standard ultrametric for $\cN$
\[
d(x, y) = 2^{-n -1} \text{if $x\neq y$ and $n$ is the least number with $x_n\neq y_n$}.
\] 
In this case, ultrametric balls $B(x,2^{-n})$ enjoy a concrete description that we now recall. 
Given sequences $s\in\N^{<\N}$ and $x\in\N^{\leq\N}$ we write $s\sqsubseteq x$\index[IdxSymb]{sq@{$s\sqsubseteq x$}} when $s$ is a \emph{prefix}, or an \emph{initial segment}, of $x$.
The sets $N_{s}=\set{x\in\cN}[s\sqsubset x]$ for all $s\in\N^{<\N}$ form a basis of clopen sets for $\cN$ and that $B(x,2^{-n})=N_{x\restr{n}}$ for all $x\in \cN$ and $n\in\N$.


For a function $f:X\to Y$ with $X\subseteq \cN$ we use this induced ultrametric and so $f$ is \enquote{locally
$\mathcal F$} becomes: every $x$ admits a basic clopen neighbourhood
$N_s\cap X$ with $f\restr{(N_s\cap X)}\in\mathcal F$.

Next we establish our first application of \cref{ultrametricdecomposition} for simple functions. 

\section{Simple functions}\label{subsectionSimple}

Scattered functions $f$ with successor rank $\alpha+1$ which are constant on $\CB_\alpha(f)$ are ubiquitous and play an important role in the sequel. 
\begin{definition}
A function $f$ is said to be \emph{simple}\index[IdxNotion]{simple function}\index[IdxNotion]{function!simple} if it is scattered with $\CB$-degree $1$. A simple function has successor $\CB$-rank $\alpha+1$ for some ordinal $\alpha$ and we call the unique point in the image of $f$ on  $\CB_{\alpha}(f)$ the \emph{distinguished point of $f$}.
\end{definition}

As announced, we now establish local simplicity which yields decomposability into simple functions by \cref{ultrametricdecomposition}. This result generalizes \cite[Lemma 2.4]{carroy2013quasi}. 

\index[IdxNotion]{Decomposition Lemma}
\begin{lemma}[Decomposition Lemma]\label{JSLdecompositionlemma}
Any scattered function $f\in \functionsonbaire$ is locally simple.

Moreover, if $f$ has successor rank $\beta+1$, then it can be written as a disjoint union of type functions of type  $(\beta+1,1)$. 
\end{lemma}

\begin{proof}
Let $f:A\to B$ be in $\functionsonbaire$ and $x\in A=\CB_0(f)$, we find a basic neighbourhood $N_s\cap A$ of $x$ on which $f$ is simple. Since $f$ is scattered, $\CB_{\CB(f)}(f)=\emptyset$ by \cref{scatterediffemptykernel_general}. Let $\gamma$ be the least ordinal such that $x\notin \CB_\gamma(f)$ and note that this ordinal must be successor. So $\gamma=\beta+1$ and $x\in \CB_\beta(f)\setminus \CB_{\beta+1}(f)$. By definition of the derivative, there exists an open set $U$ of $\N^\N$ such that $f$ is constant on $\CB_\beta(f)\cap U$. Since $\CB_{\beta+1}(f)$ is closed, there exists a basic neighbourhood $N_s\ni x$ with $N_s\subseteq U\setminus \CB_{\beta+1}(f)$. As $V=N_s\cap A$ is open, we have $\CB_\alpha(f\restr{V})=\CB_\alpha(f)\cap V$ for all ordinal $\alpha$ by \cref{CBbasics0}~\cref{CBbasicsfromJSL2}.
Hence, $x\in \CB_\beta(f\restr{V})$, $\CB_{\beta+1}(f\restr{V})=\emptyset$ and $f\restr{V}$ is constant on $\CB_\beta(f\restr{V})$. Therefore $f\restr{V}$ is simple of $\CB$-rank $\beta+1$.
 
Now assume that $\CB(f)=\beta+1$ and let $I= f(\CB_\beta(f))$. Since $f$ is locally constant on the closed set $\CB_\beta(f)$,
we can choose for each $y\in I$ an open set $U_y$ of $\dom(f)$ such that $U_y\cap \CB_\beta(f)=f^{-1}(\{y\})\cap \CB_\beta(f)$.
Applying the generalized reduction property of open sets \cite[22.16]{kechris} to the open cover of $\dom(f)$ given by $V_y=U_y\cup (\dom(f)\setminus \CB_\beta(f))$ for $y\in I$
yields a clopen partition $(C_y)_{y\in I}$ of $\dom(f)$ with $C_y\subseteq V_y$ for all $y\in I$.
Note that for all $y\in I$ we have $C_y\cap \CB_\beta(f)= f^{-1}(\{y\})\cap \CB_\beta(f)$,
which implies that each $f\restr{C_y}$ is simple of $\CB$-rank equal to $\beta+1$ as before.
\end{proof}

\section{Gluing}\label{subsectionGluing}

We close this preliminary section with the study of the Gluing operation which provides a natural \enquote{addition} for functions. Combined with the Pointed Gluing operation, also defined in~\cite{carroy2013quasi} and studied in \cref{sectionPointedGluing}, this operation will allow us to define  maximum (respectively minimum) functions among functions in $\sC$ with $\CB$-rank at most $\alpha$ (respectively at least $\alpha$) for all countable ordinal $\alpha$. 

Given sequences $s\in\N^{<\N}$ and $x\in\N^{\leq\N}$ we denote by $s\conc x$ the concatenation of $s$ with $x$,\index[IdxSymb]{conc@{$s\conc x$, $s\conc A$}}
and for any $A\subseteq\N^{\leq\N}$ the notation $s\conc A$ stands for $\set{s\conc x}[x\in A]$.
The following operation is called \emph{sum of the mappings} in \cite[2.2.E]{engelking}.

\begin{definition}
We call \emph{gluing}\index[IdxNotion]{operation!gluing} of a countable sequence $(A_i)_ {i\in I}$ of subsets of $\cN$ the set $\gl_{i\in I}A_i=\bigcup_{i\in I}(i)\conc A_i$.\index[IdxSymb]{op@{$\gl_iA_i$}}
Given a countable sequence of functions $(f_i:A_i\rao B_i)_{i\in I}$ in $\functionsonbaire$,
We call \emph{gluing} of $(f_i)_{i\in I}$, the map\index[IdxSymb]{op@{$\gl_if_i$}}
\begin{align*}
\gl_{i\in I}f_i : \gl_{i\in I}A_i &\longrightarrow\gl_{i\in I}B_i\\
 (i)\conc x &\longmapsto (i)\conc f_i(x).
\end{align*}
\end{definition}

The following facts are either classical exercises left to the reader, or a direct consequence of \cref{CBrankofclopenunion}.

\begin{fact}\label{BasicsOnGluing}
\
\begin{enumerate}
\item The gluing operation preserves countability\footnote{Because we only consider \emph{countable} gluings here.}, openness, closure, and Polishness of sets; continuity, injectivity, surjectivity and scatteredness of functions.
\item For $I\subseteq \N$ and $(f_i)_{i\in I}$ in $\functionsonbaire$, we have $\CB(\gl_if_i)=\CB(\bigsqcup_if_i)=\sup_i\CB(f_i)$.\index[IdxNotion]{Cantor--Bendixson!rank!of a gluing}
\item Gluing \emph{commutes with the identity natural transformation}, that is for any $(X_i)_{i\in I}\subseteq \cN$ we have $\id_{\gl_{i\in I}X_i}=\gl_{i\in I}(\id_{X_i})$.
\end{enumerate}
\end{fact}

\subsection{Gluing as an upper bound}

Our main task for each of our three operations is to prove sufficient -- and, whenever possible, necessary -- conditions on operands $(g_i)_i$
for the result of the operation to reduce (or be reduced by) an arbitrary function $f$.
The next result represents the first step, characterizing the situation where $\gl_i g_i$ is an upper bound for a given function $f$.
These results are all new, except for the next one that was somehow already present in \cite[Lemma 3.3]{carroy2013quasi}.

\index[IdxNotion]{upper bound condition!gluing}\index[IdxNotion]{gluing operation!upper bound condition}
\begin{proposition}[Gluing as upper bound]\label{Gluingasupperbound}
Let $(g_i:X_i\to Y_i)_{i\in I}$ be countable sequence in $\functionsonbaire$. 
The following are equivalent for every function $f:A\to B$ between topological spaces:
\begin{enumerate} 
\item $f\leq \gl_{i\in I}g_i$, 
\item there is a clopen partition $(A_i)_{i\in I}$ of $A$ such that $f\restr{A_i}\leq g_i$ for all $i\in I$.
\end{enumerate}
\end{proposition}

\begin{proof}
Suppose that for all $i\in I$ there is a continuous reduction $(\sigma_i,\tau_i)$ from $f\restr{A_i}$ to $g_i$. Set then $\sigma:x\mapsto (i)\conc \sigma_i(x)$ if $x\in A_i$ and $\tau:(i)\conc y\mapsto\tau_i(y)$. The map $\tau$ is continuous by \cref{lem:ContUnion}. Since the sets $A_i$ are pairwise disjoint, $\sigma$ is well-defined, and it is continuous because the sets $A_i$ are clopen.
The fact that $(\sigma_i,\tau_i)$ is a reduction for all $i\in I$ implies that $(\sigma, \tau)$ is a reduction from $f$ to $\gl_{i\in I}g_i$.

Suppose now that $(\sigma,\tau)$ continuously reduces $f$ to $\gl_i g_i$, and set $A_i=\sigma^{-1}(N_{(i)})$.
Since $(N_{(i)})_i$ is a clopen partition of $\dom(\gl_{i\in I} g_i)$ and $\sigma$ is continuous, $(A_i)_i $ is a clopen partition of $\dom(f)$.
Conclude by observing that $f\restr{A_j}\leq(\gl_{i\in I}g_i)\restr{N_{(j)}}\equiv g_j$ for all $j\in I$.
\end{proof}

The following simple corollary will be used extensively in the sequel.

\begin{corollary}\label{Gluingasupperbound_cor}
Let $f:A\to B$ in $\functionsonbaire$ and let $\partit$ be a clopen partition of $A$.
Then $f=\bigsqcup_{P\in\partit} f\restr{P}\leq \gl_{P\in\partit} f\restr{P}$.
\end{corollary}

\subsection{Gluing as a lower bound}
Dually, bounding a function $f$ from below with a gluing involves corestrictions of $f$.

\index[IdxNotion]{lower bound condition!gluing}\index[IdxNotion]{gluing operation!lower bound condition}
\begin{proposition}[Gluing as lower bound]\label{Gluingaslowerbound}
Let $(g_i:X_i\to Y_i)_{i\in I}$ be a countable sequence in $\functionsonbaire$.
The following are equivalent for every function $f:A\to B$ between metrizable spaces:
\begin{enumerate}
\item $\gl_{i\in I}g_i\leq f$
\item there exists a family $(U_i)_{i\in I}$ of subsets of $B$ satisfying:
\begin{enumerate}
\item $(U_i)_{i\in I}$ is a relative clopen partition, and
\item for all $i\in I$ we have $g_i\leq f\corestr{U_i}$.
\end{enumerate}
\end{enumerate}
\end{proposition}

\begin{proof}
Suppose first that $(\sigma,\tau)$ continuously reduces $\gl_{i\in I}g_i$ to $f$, then the sets $U_i=\tau^{-1}(N_{(i)})$, $i\in I$, 
form a clopen partition of $\dom \tau=\im f\sigma$. 
As $g_i\leq(\gl_{j\in I}g_j)\corestr{N_{(i)}}=(\tau f\sigma)\corestr{N_{(i)}}=(f\sigma)\corestr{U_i}\leq f\corestr{U_i}$, we have the direct implication.

For all $i\in I$, we fix a continuous reduction $(\sigma_i,\tau_i)$ from $g_i$ to $f\corestr{U_i}$,
and set $\sigma:(i)\conc x\mapsto\sigma_i(x)$ for $x\in X_i$ and $\tau: y\mapsto (i)\conc \tau_i(y)$ for $y\in \dom \tau_i$.
Since the sets $U_i$ are pairwise disjoint and $\dom \tau_i\subseteq U_i$, $(\sigma,\tau)$ is a well-defined pair of maps. Note that $\sigma$ is continuous by \cref{lem:ContUnion}.
Since $(U_i)_i$ is a relative clopen partition and $\dom \tau_i\subseteq U_i$, so is the family $(\dom \tau_i)_i$ and $\tau$ is continuous by \cref{lem:ContUnion} too.
The fact that $(\sigma_i,\tau_i)$ are reductions for all $i\in I$ implies that $(\sigma, \tau)$ is a reduction from $\gl_{i\in I}g_i$ to $f$.
Indeed, for $x\in X_i$ we have $\tau f \sigma ((i)\conc x)=(i)\conc \tau_i f \sigma_i(x)= (i)\conc g_i(x)$.
\end{proof}

We use the following simple corollary extensively in the sequel.

\begin{corollary}\label{Gluingaslowerbound2}
Let $f:A\to B$ be in $\functionsonbaire$. If $\mathcal{U}$ is a family of pairwise disjoint open sets of $B$, then $\gl_{U\in\mathcal{U}}f\corestr{U}\leq f$.
\end{corollary}

In order to prove that a function $f$ is continuously equivalent to a gluing of functions, one needs to consider both restrictions and corestrictions of $f$. 
One needs to find both a clopen partition of the domain of $f$ to use \cref{Gluingasupperbound} and a relatively clopen partition in the image of $f$
to use \cref{Gluingaslowerbound}.
While a clopen partition of the image of a continuous $f$ induces a clopen partition of the domain as in the following corollary,
we will also consider families in the domain and the range of $f$ that are not related in this way (see \cref{SolvableDecomposition}).

\begin{corollary}\label{UsefulcriterionforequivFinGl}
Let $f:A\to B$ be a continuous function between topological spaces
and $(g_i:X_i\to Y_i)_{i\in I}$ be countable sequence in $\functionsonbaire$.

If there exists a clopen partition $B=\bigsqcup_{i\in I} B_i$ with $f\corestr{B_i}\equiv g_i$ for all $i\in I$, then $f\equiv \gl_{i\in I} g_i$.
\end{corollary}

For any countable ordinal $\alpha$, we write $\sC_\alpha$ (resp. $\sC_{\leq \alpha}, \sC_{<\alpha}$) for the set of all functions of $\CB$-rank $\alpha$ (resp. $\leq\alpha,<\alpha$) in $\sC$.
\index[IdxSymb]{Scat@{$\sC_\alpha$, $\sC_{\leq\alpha}$, $\sC_{<\alpha}$}}

For instance, $\sC_0$ consists only of the empty function, and $\sC_1$ is the class of locally constant functions.

A function of finite $\CB$-degree has finite image on the last $\CB$-derivative, so the previous corollary immediately gives the generalization to $\sC$ of \cite[Lemma 4.3]{carroy2013quasi}:

\begin{corollary}\label{FiniteDegreeAreFinGl}
Let $\alpha<\omega_1$ and $f\in\sC_{\alpha+1}$.
If the $\CB$-degree of $f$ is $n+1$, then $f\equiv\gl_{i\leq n}f_i$ for some sequence $(f_i)_{i\leq n}$ of simple functions in $\sC_{\alpha+1}$.
\end{corollary}
\begin{proof}
Take a finite clopen partition of the co-domain of $f$ such that each piece contains exactly one point of the image of $f$ on $\CB_\alpha(f)$ and apply the previous corollary.
\end{proof}

Given functions $f_0$ and $f_1$ in $\functionsonbaire$, we write $f_0\glbin f_1$ instead of $\gl_{i\in 2}f_i$ for the binary gluing operation.
For $n\leq \omega$ and $f\in \functionsonbaire$, we write $nf$ for $\gl_{i\in n} f$.\index[IdxSymb]{op@{$f_0\glbin f_1$}}\index[IdxSymb]{op@{$nf$}}
A direct application of \cref{Gluingasupperbound} gives:

\begin{fact}\label{Gluingcohomomorphism}~
Let $(f_i)_{i\in I}, (g_k)_{k\in K} \subseteq \functionsonbaire$. If there exists $\iota:I\to K$ with $f_i\leq g_{\iota(i)}$ for all $i\in I$, then $\gl_{i\in I}f_i\leq \gl_{k\in K} g_k$. 

In particular, if $f\in\functionsonbaire$ and $m\leq n\leq \omega$, then $mf\leq nf$.
\end{fact}

\begin{definition}
If $\mathcal{F}$ is a set of functions, then we denote by $\FinGl{\mathcal{F}}$\index[IdxSymb]{Fin@{$\FinGl{\mathcal{F}}$}} the set of all finite gluings $\gl_{i< n} f_i$ for some finite sequence $(f_i)_{i<n}$ of functions in $\mathcal{F}$.
We say that a set $\mathcal{C}$ of functions is \emph{generated} by a set $\mathcal{F}$ if for all $g\in\mathcal{C}$ there exists $f\in \FinGl{\mathcal{F}}$ such that $g\equiv f$.
\index[IdxNotion]{finitely generated set of functions}

We say that $\mathcal{C}$ is \emph{finitely generated} if it generated by a finite set. 
\end{definition}

Remark that for $\mathcal{C}$ to be generated by $\mathcal{F}$, we do not require $\mathcal{F}$ to be a subset of $\mathcal{C}$. 
Note the empty function always belongs to $\FinGl{\mathcal{F}}$ as the gluing of the empty sequence and every element of $\FinGl{\mathcal{F}}$ takes the form $\gl_{i\leq k} n_i f_i$ where $k\in\N$, $n_0,\ldots,n_k\in \N$ and $f_0,\ldots,f_k\in \mathcal{F}$ are pairwise distinct. 

Given two quasi-orders $(P,\leq_P)$ and $(Q,\leq_Q)$, we say that a map $\varphi:P\to Q$ is a \emph{co-homomorphism}\index[IdxNotion]{co-homomorphism} when for all $p,p'$ in $P$ we have
$\varphi(p)\leq_Q\varphi(p')$ implies $p\leq_P p'$. In such a case, if $Q$ is a \wqo{} (resp. a \bqo{}) then so is $P$ (\cite[Lemma 3.10]{wellbetterinbetween}).

Recall also that \wqo{}s (resp. \bqo{}s) are stable by finite products, and that all well-orders are \bqo{}s (\cite[Example 3.9 and Proposition 3.14]{wellbetterinbetween}).
\index[IdxNotion]{better-quasi-order (\bqo{})}

\begin{proposition}\label{SecondstepforBQOthm}
Continuous reducibility is a \bqo{} on any finitely generated set of functions.
\end{proposition}

\begin{proof}
Let $\mathcal{C}$ be generated by a finite set of functions $\mathcal{F}=\set{f_0,\ldots, f_k}$ for some $k\in\N$.
For all $f\in\mathcal{C}$ there is an $s_f\in\N^{k+1}$ such that $f\equiv \gl_{i\leq k}s_f(i)f_i$.
But then if $f,g$ in $\mathcal{C}$ satisfy $s_f(i)\leq s_g(i)$ for all $i\leq k$, then $f\leq g$ by \cref{Gluingcohomomorphism},
so $f\mapsto s_f$ is a co-homomorphism from $\mathcal{C}$
to the finite product of well-orders $\N^{k+1}$, which makes $\mathcal{C}$ a \bqo{}.
\end{proof}

The previous proposition motivates our strategy underlined in \cref{Introthm:LevelsAreFinitelyGenerated}.
To prove that each level $\sC_\alpha$ is a \bqo{}, we actually describe a finite set of functions that generates it in the above sense.
The level $\sC_0$ consists only of the empty function, so
the first interesting level is $\sC_1$, that of locally constant functions, to which we now turn.

We use several time in this paper the following well-known exercise.

\begin{fact}\label{InfiniteEmbedOmega}
Any infinite metrizable space contains an infinite discrete subspace.
\end{fact}

We close this section with the analysis of the locally constant functions up to continuous equivalence.

\index[IdxNotion]{locally constant function}\index[IdxNotion]{function!locally constant}
\begin{proposition}\label{LocallyConstantFunctions}
The class of locally constant functions $f:X\to Y$ from a second countable space $X$ to a metrizable space $Y$ is (finitely) generated by\footnote{For notational simplicity we adopt the convention that $1=\set{\iw{0}}$.} $\set{\id_{1},\id_\N}$. In fact, $f\equiv \id_{I}$ where $I$ is a discrete set with cardinality $|\im f|$. 
\end{proposition}
\begin{proof}
Observe first that any constant function on a non-empty topological space is continuously equivalent to $\id_{1}$.
Now if $f$ is locally constant, the sets $A_y=f^{-1}(\set{y})$ for $y\in\im(f)$ form a clopen partition of $\dom(f)$ such that $f\restr{A_y}\equiv \id_{1}$ for all $y\in \im f$.
By \cref{Gluingasupperbound}, we get $f\leq \gl_{y\in \im f} \id_{1}$ and note that since $X$ is second countable, $\im f$ is countable. 

If $\im f$ is finite of cardinality $n$, then it is discrete and so using \cref{Gluingaslowerbound} with $B_y=\{y\}$ for $y\in \im f$,
we get $n \id_{1} \leq f$. Therefore $f\equiv n  \id_{1}$ in this case.

If $\im f$ is infinite, being metrizable it contains a countably infinite discrete subspace by \cref{InfiniteEmbedOmega},
so we similarly obtain $\gl_{n\in\N} \id_{1} \leq f$ using \cref{Gluingaslowerbound}.
Therefore in this case $f\equiv \gl_{n\in\N} \id_{1} \equiv \id_{\N}$, which concludes the proof.
\end{proof}

In the sequel we introduce the necessary concepts which allows us to generalize the previous simple proof to all successor levels of $\sC$.
Crucially, we define \emph{centered functions} which generalizes constant functions in \cref{sectionCentered}.
Similarly to the unique image of a constant function, centered functions admit a special point in their image, their \emph{\cocenter{}}.
While there is, up to equivalence, only one constant function, we will show that there are, up to equivalence, only finitely many centered functions of a given $\CB$-rank 
\cref{finitenessofcenteredfunctions}.
We eventually prove in \cref{PreciseStructureCor1} that every function $f$ is locally centered.

In the next section, we study the general structure of $\sC$ which mainly establishes that the structure of limit levels is as simple as one could hope for and ensures that finite generation of all levels indeed entail that $\sC$ is \bqo{}.

%

\chapter{Pointed Gluing and the General Structure}\label{sectionPointedGluing}

After establishing \cref{SecondstepforBQOthm}, our next task is to prove that \cref{Introthm:LevelsAreFinitelyGenerated} implies \cref{Introthm:BQO}.
The key to do so is to develop the technology to increase the $\CB$-rank\footnote{As we have seen, gluing functions does not.}, which is the objective of this section.
This is the motivation for introducing the Pointed Gluing operation from \cite{carroy2013quasi}, and study its binding conditions in the same way we did for the Gluing operation.
This will be central in \cref{PreciseStructureFinal,sectionDoubleSucc},
and gives the generalization to $\sC$ of \cite[Theorem 5.2]{carroy2013quasi} that we obtain at the end of the present section.

Note that the pointed gluing operation is implicitly used in the theory of Wadge degrees since Wadge's thesis~\cite{phdwadge}, sometimes with slight variations (see for instance \cite[Section 5]{selivanovnew}).
It has been used explicitly in~\cite[Subsection 2.6]{CMRSWadge}

Recall that for a sequence $s\in\N^{<\N}$ we denote by $s^n$ for $n\in\N$ the concatenation of $s$ with itself $n$-many times.
We make the convention that $s^0$ is the empty sequence and $s^\infty$ \index[IdxSymb]{s@{$s^n$, $s^\infty$}}is the infinite concatenation of $s$ with itself.

\begin{definition}
Given a sequence $(F_i)_{i\in \N}$ of subsets of $\cN$, 
we call \emph{pointed gluing} of $(F_i)_ {i\in \N}$ the set $\pgl_{i\in\N} F_i =\{\iw{0}\}\cup\bigcup_{i\in \N}(0)^i\conc(1)\conc F_i$.

\index[IdxNotion]{operation!pointed gluing}\index[IdxSymb]{ptgl@{$\pgl_{i\in\N} F_i$}}\index[IdxSymb]{ptgl@{$\pgl_{i\in\N} f_i$}}
\medskip

Given a sequence of functions $(f_i:A_i\rao B_i)_{i\in \N}$ in $\functionsonbaire$,
we call \emph{pointed gluing} of $(f_i)_{i\in \N}$, the map
 \begin{align*}
\pgl_{i\in \N}f_i: \pgl_{i\in \N}A_i &\longrightarrow\pgl_{i\in \N}B_i\\
 x &\longmapsto\begin{cases}
(0)^i\conc(1)\conc f_i(x') & \mbox{if }x=(0)^i\conc(1)\conc x'\\
\iw{0} & \mbox{otherwise.}
\end{cases}
\end{align*}
\end{definition}


\begin{fact}\label{BasicsOnPointedGluing}
The pointed gluing operation preserves countability, closure, compactness, Polishness of sets; continuity, injectivity, surjectivity, scatteredness of functions.

It commutes with the identity natural transformation, i.e. $\id_{\pgl_{i}X_i}=\pgl_{i}\id_{X_i}$.

The point $\iw{0}$ is a continuity point of $\pgl_{i\in\N} f_i$ for any sequence of functions $(f_i)_{i\in\N}$ in $\functionsonbaire$.
\end{fact}

The following proposition generalizes \cite[Proposition 3.1]{carroy2013quasi}.

A sequence $(\alpha_n)_{n\in\N}$ of ordinals is called \emph{regular} when for all $m\in\N$ there exists a natural $n>m$ such that $\alpha_m\leq\alpha_n$, or
equivalently, when $\sup_n(\alpha_n+1)=\limsup_n(\alpha_n+1)$.

Recall that if $f$ is simple with $\CB(f)=\alpha+1$ the distinguished point\index[IdxNotion]{distinguished point (of a simple function)} of $f$ is the unique $y\in\im(f)$ satisfying $\set{y}=f(\CB_{\alpha}(f))$.



\index[IdxNotion]{Cantor--Bendixson!rank!of a pointed gluing}
\begin{proposition}\label{CBrankofPgluingofregularsequence1}
Let $f=\pgl_{n\in\N} f_n$ for some sequence $(f_n)_{n\in\N}$ of scattered functions in $\functionsonbaire$. 
If $(\CB(f_n))_{n\in\N}$ is regular and $\alpha=\sup_{n\in\N}(\CB(f_n))$, then $\CB_\alpha(f)=\set{\iw{0}}$.
In particular, $f$ is simple of distinguished point $\iw{0}$, and $\CB(f)=\alpha+1$.
\end{proposition}
\begin{proof}
Since $f_n\equiv f\restr{N_{(0)^n\conc(1)}}$, we have $\CB(f_n)=\CB(f\restr{N_{(0)^n\conc(1)}})$ by \cref{CBbasicsfromJSL-tp}.
If $\beta<\alpha$ then by regularity $\CB_\beta(f\restr{N_{(0)^n\conc(1)}})=\CB_\beta(f)\cap N_{(0)^n\conc(1)}$ is non-empty for infinitely many $n$,
which implies that $\iw{0}\in\CB_{\beta+1}(f)$. Therefore $\iw{0}\in\CB_{\alpha}(f)$.
Since $\CB_{\alpha}(f\restr{N_{(0)^n\conc(1)}})$ is empty for all $n$, it follows that $\CB_{\alpha}(f)=\set{\iw{0}}$.
\end{proof}

\begin{fact}\label{GluinglowerthanPgluing}
Given a sequence $(f_i)_{i\in \N}$ in $\functionsonbaire$,
we have $\gl_{i\in \N} f_i\leq \pgl_{i\in \N} f_i$.
\end{fact}

\begin{proof}
Use \cref{Gluingaslowerbound} with $f=\pgl_{i\in \N} f_i$ and $B_i=N_{(0)^i\conc(1)}$ for ${i\in \N}$.
\end{proof}

We write $x_n\to x$\index[IdxSymb]{_@{$x_n\to x$}} to say that a sequence $(x_n)_{n\in\N}$ of points in a topological space converges to a point $x$.
We recall the following useful sufficient criterion for continuity (\cite[Claim 3.2]{carroy2013quasi}).

\begin{lemma}\label{prop:sufficientcondforcont}
Let $A$ and $B$ be metrizable spaces and $f:A\rao B$.
Suppose that $U$ is an open subset of $A$ satisfying:
\begin{enumerate}
\item $f$ is continuous both on $U$ and on $A\setminus U$, and
\item for all $(x_n)_{n\in\N}$ in $U$, if $x_n\rao x\in A\setminus U$ then $f(x_n)\rao f(x)$.
\end{enumerate}
Then $f$ is continuous.
\end{lemma}

\begin{proof}
Since $A$ is metrizable, we prove sequential continuity at every point $x\in A$. Take a sequence $x_n\rao x$, we need to prove that $f(x_n)\rao f(x)$.
Observe first that if $x\in U$ then a tail of $(x_n)_n$ is actually included in the open set $U$, so using continuity of $f\restr{U}$ we have $f(x_n)\rao f(x)$.

So suppose that $x\notin U$, and partition $\N$ in two pieces: $I=\set{n\in\N}[x_n\in U]$ and $J=\N\setminus I$.
If $I$ is finite, then continuity of $f\restr{A\setminus U}$ yields $f(x_n)\rao f(x)$.
Otherwise, let $V\ni f(x)$ be an open set, the second hypothesis gives $N\in \N$ such that for all $n\in I$, if $n>N$ then $f(x_n)\in V$.
If $J$ is finite, fix $m=\max\set{N, \max J}$ to get $f(x_n)\in V$ for all natural numbers $n>m$.
If $J$ is infinite, use continuity of $f\restr{A\setminus U}$ to get $N'\in \N$ such that $f(x_n)\in V$ for all $n\in J$ with $n>N'$.
Therefore $f(x_n)\in V$ for all natural numbers $n>\max\set{N,N'}$, which proves that $f(x_n)\rao f(x)$ in this case too.
\end{proof}

\section{Pointed gluing as an upper bound}

We say that a sequence of functions $(f_i)_{i\in\N}$ \emph{is reducible by (finite) pieces to}\index[IdxNotion]{reducible by (finite) pieces} a sequence $(g_i)_{i\in\N}$ of functions
if there is a family $(I_n)_{n\in\N}$ of pairwise disjoint finite sets of $\N$ such that for all $n\in\N$ we have $f_n\leq\gl_{i\in I_n}g_i$.

For our analysis of the pointed gluing, we introduce the following way of partitioning a function at a given point of its codomain.

\begin{definition}
For $B\subseteq \cN$, $y\in \cN$ and $n\in \N$, the \emph{$n$-th ray of $B$ at $y$}\index[IdxNotion]{ray!of a set at a point} is the clopen subset of $B$ defined by 
\[
\ray{B}{y,n}:=\{x\in B\mid y\restr{n}\segm x \text{ and } y\restr{n+1} \nsegm x\}.
\]
In such a case, if $f:A\to B$ is a function, the \emph{$n$-th ray of $f$ at $y$}\index[IdxNotion]{ray!of a function at a point} is the function $\ray{f}{y,n}:=f\corestr{\ray{B}{y,n}}$.\index[IdxSymb]{B@{$\ray{B}{y,n}$}}\index[IdxSymb]{f@{$\ray{f}{y,n}$}}
\end{definition}

\index[IdxNotion]{upper bound condition!pointed gluing}
\index[IdxNotion]{pointed gluing operation!upper bound condition}
\begin{proposition}[Pointed gluing as upper bound]\label{Pgluingasupperbound}
Let $f\in \functionsonbaire$ be continuous 
and $(g_i)_{i\in\N}$ be a sequence in $\functionsonbaire$.
If $y\in B$ and $(\ray{f}{y,j})_{j\in \N}$ is reducible by pieces to $(g_i)_{i\in \N}$,
then $f\leq\pgl_{i\in \N}g_i$.
\end{proposition}

\begin{proof}
Let $f:A\to B$ in $\functionsonbaire$ be continuous and note that if $y\notin\im f$ then we can actually conclude by \cref{Gluingasupperbound}, so suppose that $y\in\im f$.
Fix a family $(I_n)_{n\in\N}$ witnessing the reduction by pieces, and for all $n\in\N$ a reduction $(\sigma_n, \tau_n)$ from $\ray{f}{y,n}$ to $\gl_{i\in I_n}g_i$.
Observe that for all $n\in\N$ we have $\sigma_n(x)=(i)\conc x'$ for some $i\in I_n$ and $x'\in \dom g_i$.
Since $f$ is continuous and the sets $\ray{B}{y,n}$ are clopen, the sets $A_n:=\dom(\ray{f}{y,n})=\dom(\sigma_n)$ form a clopen partition of $A\setminus f^{-1}(\set{y})$,
so $\tilde{\sigma}=\bigsqcup_n\sigma_n$ is a well-defined function on $A\setminus f^{-1}(\set{y})$.
Define $\sigma$ as follows: $\sigma(x)=\iw{0}$ if $f(x)=y$, and $\sigma(x)=(0)^i\conc(1)\conc x'$ when $\tilde{\sigma}(x)=(i)\conc x'$.

Let $I=\bigsqcup_{n\in\N} I_n$ and $\eta:I\to \N$ be such that $\eta(i)$ is the unique $n\in\N$ such that $i\in I_n$.
Define now $\tau$ as follows: $\tau(\iw{0})=y$, and $\tau(x)=\tau_{\eta(i)}((i)\conc x')$ when $x=(0)^i\conc(1)\conc x'$, $(i)\conc x'\in \im \tau_{\eta(i)}$ and $i\in I$.

As the sets $I_n$ are pairwise disjoint, $\eta$ is injective and $\tau$ is well-defined.
Note that, by construction, $\sigma$ is defined on all of $A$ and $f=\tau(\pgl_{i\in \N}g_i)\sigma$.
So it suffices to show that $\sigma$ and $\tau$ are continuous.

\medskip

We now use \cref{prop:sufficientcondforcont} with the open set $U=\dom f \setminus f^{-1}(\set{y})=\bigcup_{n\in\N}A_n$ to obtain continuity of $\sigma$:
as $\tilde{\sigma}$ is continuous by \cref{lem:ContUnion} so is $\sigma$ on $U$, hence it is enough to check the second condition.
Towards that goal, suppose that a sequence $(x_n)_n$ in $\dom f$ with $f(x_n)\neq y$ converges to $x$ with $f(x)=y$. To show that $\sigma(x_n)\rao\iw{0}$, we prove that given any $m\in\N$, we have $\sigma(x_n)\in N_{(0)^m}$ for sufficiently large $n$.

For all $n\in\N$, let $j_n\in\N$ be such that $f(x_n)\in \ray{B}{y,j_n}$. By definition of $\sigma$ we have $\sigma(x_n)\in N_{(0)^i}$ for some $i\in I_{j_n}$.
Now, as the sets $(I_{j_n})_{n\in\N}$ are nonempty and pairwise disjoint, we have $\min I_{j_n}\to\infty$ as $n\to \infty$.
So there exists $N\in \N$ such that for all $n\geq N$ we have $\min I_{j_n}\geq m$ and therefore $\sigma(x_n)\in N_{(0)^m}$.

We proceed similarly for $\tau$ using this time $U=\dom \tau \setminus\set{\iw{0}}$. Note that $\tau$ is continuous on $U$ by \cref{lem:ContUnion}. 
Now take a sequence $(x_n)_{n\in\N}$ in $U$ converging to $\iw{0}$, we prove that $\tau(x_n)\rao y$. 
Given $m\in\N$, set $M=\max(\bigcup_{j\leq m}I_j)$. Since $x_n\rao\iw{0}$ there is $N\in\N$ such that $x_n\in N_{(0)^{M+1}}$ for all $n>N$.
Therefore by definition of $\tau$, for all $n>N$ we have $\tau(x_n)\in\bigcup_{j>m}\im(\tau_j)$, which means that $\tau(x_n)\in \nbhd{y}{m}$.
We thus get $\tau(x_n)\rao y$, so $\tau$ is continuous by \cref{prop:sufficientcondforcont}.
\end{proof}

\cref{Pgluingasupperbound} gives the following rough -- yet very useful -- way to upper bound a continuous function using the pointed gluing operation.

\begin{corollary}\label{Pgluingofraysasupperbound}
Let $f:A\to B$ be a continuous function in $\functionsonbaire$.
Then $f\leq \pgl_{i\in \N}\ray{f}{y,i}$ for all $y\in B$.
\end{corollary}


The following corollary states another relationship of the pointed gluing with the finite gluing:

\begin{corollary}\label{SplittingaPgluingonatail}
Let $(f_i)_{i\in \N}$ be continuous in $\functionsonbaire$.
For all $n\in\N$, we have $\pgl_{i\in\N}f_i\equiv(\gl_{i<n}f_i)\glbin \pgl_{i\geq n}f_{i}$.
\end{corollary}
\begin{proof}
Using \cref{Pgluingasupperbound} with $y=\iw{0}$, we get $(\gl_{i<n}f_i)\glbin \pgl_{i\geq n}f_{i}  \leq \pgl_{i\in \N} f_i$.
Moreover, we see that  $\pgl_i f_i\leq (\gl_{i<n}f_i)\glbin \pgl_{i\geq n}f_{i}$ using \cref{Gluingaslowerbound} with the clopen partition $\set{N_{(0)^i\conc(1)}}[i<n]\cup\set{N_{(0)^n}}$.
\end{proof}

Recall that we saw in~\cref{CBrankofPgluingofregularsequence1} that the pointed gluing of a sequence with regular $\CB$-ranks is simple. Given a simple function $f$ with distinguished point $y$, we have $f\leq \pgl_{n\in\N} \ray{f}{y,n}$ by \cref{Pgluingofraysasupperbound}. While it is not true in general (see \cref{SomeCounterExamples}) that $f\equiv \pgl_{n\in\N} \ray{f}{y,n}$, we have the following:

\begin{proposition}\label{CBrankofPgluingofregularsequence2simple}
Let $f\in\functionsonbaire$ be a scattered simple function with distinguished point $y$ and $\CB(f)=\alpha+1$. Then $\sup_{n\in\N} \CB(\ray{f}{y,n})= \alpha$. If moreover $\alpha$ is successor, then the sequence $(\CB(\ray{f}{y,n}))_{n\in\N}$ is regular.
\end{proposition}
\begin{proof}
Take a simple $f:A\to B$ is in $\sC_{\alpha+1}$ for some $\alpha<\omega_1$ with distinguished point $y$, and set $\alpha_i=\CB(\ray{f}{y,i})$ for all $i\in\N$.
By simplicity, $\CB_\alpha(f)\subseteq f^{-1}(\{y\})$ and so for all $i\in \N$ we have $\CB_\alpha(\ray{f}{y,i})=\CB_\alpha(f)\cap f^{-1}(\ray{B}{y,i})=\emptyset$. This implies $\alpha_i\leq \alpha$ and $\sup_i\alpha_i \leq \alpha$.

Next we claim that for all $n\in \N$ we have $\sup_{i\geq n} \alpha_i=\alpha$.
Towards a contradiction, suppose that $\sup_{i\geq n} \alpha_i=\beta<\alpha$ for some $n\in\N$. Set $g:=f\restr{\dom(f)\setminus(\bigcup_{i\leq m}\dom(\ray{f}{y,i}))}$ and observe that $\CB_\beta(g)\subseteq f^{-1}(\set{y})$ because $\CB_\beta(\ray{g}{y,n})$ is always empty: for $i<n$ it is by definition of $g$, and for
$i\geq n$ this follows from our assumption that $\ray{g}{y,n}=\ray{f}{y,n}$ has $\CB$-rank $\alpha_n\leq\beta$ using \cref{CBbasics0}~\cref{CBbasicsfromJSL2}. Thus $g\restr{\CB_\beta(g)}$ is either empty or constant, in any case $\CB(g)\leq\beta+1\leq\alpha$.
Since $f=g\sqcup(\bigsqcup_{i\leq m}\ray{f}{y,i})$ is a disjoint union of functions we have $\CB(f)\leq \max \{\beta+1,\alpha\}=\alpha$ by \cref{CBrankofclopenunion}, a contradiction. This proves the claim and concludes the first part.

Finally, if $\alpha=\gamma+1$, then the claim implies that for all $n\in N$ there exists $i\geq n$ with $\alpha_i=\gamma+1$ so $(\alpha_i)_{i\in \N}$ is regular.
%
\end{proof}

\section{Maximum functions}

As an important consequence of the ``upper bound'' criterion stated in \cref{Pgluingasupperbound}, we identify functions which are maximum among functions in $\sC$ with $\CB$-rank at most $\alpha$.

Recall that $\omega f$\index[IdxSymb]{w@{$\omega f$}} denotes the gluing of the constant sequence with value $f$. Similarly, we write $\pgl f$ for the pointed gluing of the constant sequence with value $f$.

\begin{definition}[Minimum and maximum functions]\label{Def_MinMaxFunc}
We define by induction on $\alpha$ a function \index[IdxNotion]{minimum and maximum functions}$\Maximalfct{\alpha}$ in $\sC_\alpha$ and a function $\Minimalfct{\alpha+1}$ in $\sC_{\alpha+1}$.
Set $\Maximalfct{0}=\emptyset$ and $\Minimalfct{1}=\pgl\emptyset$.\index[IdxSymb]{l@{$\Maximalfct{\alpha}$}}\index[IdxSymb]{k@{$\Minimalfct{\alpha+1}$}}
Suppose that for all $\beta<\alpha$ the functions $\Maximalfct{\beta}$ and $\Minimalfct{\beta+1}$ are defined.

If $\alpha=\beta+1$, then we set $\Minimalfct{\alpha+1}=\pgl \Minimalfct{\alpha}$ and $\Maximalfct{\alpha}=\omega\pgl \Maximalfct{\beta}$.

Otherwise, that is if $\alpha$ is limit, then fix a sequence $(\alpha_n)_{n\in\N}$ cofinal in $\alpha$ and set $\Minimalfct{\alpha+1}=\pgl_n\Minimalfct{\alpha_n+1}$.
Fix as well an enumeration $(\beta_n)_{n\in\N}$ of $\alpha$ and set $\Maximalfct{\alpha}=\gl_n\Maximalfct{\beta_n}$.
\end{definition}

Note that $\Minimalfct{1}\equiv\id_{1}$ and $\Maximalfct{1}\equiv \id_{\N}$. For $\alpha$ limit, up to continuous equivalence, the definition of these functions does not depend on the choice of the sequences $(\alpha_n)_n$ and $(\beta_n)_n$.
For $\Minimalfct{\alpha+1}$, it follows from a classical well-known result of Mazurkiewicz and Sierpinskì \cite[Théorème 1]{SierpMazCompDen}
- we will get an alternative proof of that later with \cref{Compactdomains};
as for $\Maximalfct{\alpha}$, by \cref{Gluingasupperbound,Gluingaslowerbound},
different choices of enumerations $(\beta_n)_n$ produce continuously equivalent functions $\Maximalfct{\alpha}$.

Observe that since the Gluing and Pointed Gluing operations commute with the identity natural transformation and $\Minimalfct{1}=\id_{\set{\iw{0}}}$, the functions from \cref{Def_MinMaxFunc} are identity functions on their domains.
The spaces $\dom \Minimalfct{\alpha}$ are somewhat classical objects, they were defined exactly this way in \cite{carroy2013quasi}.
Their properties can prove useful in other -- yet related -- contexts, as illustrated by \cite[Proposition 3.12]{CMRSWadge}.
We will see here (in \cref{Minfunctions}) that each function $\Minimalfct{\alpha+1}$ is minimum in $\sC_{\geq \alpha+1}$.

The functions $\Maximalfct{\alpha}$ were defined in \cite[Definition 5.34]{phdcarroy}
and a version of the following result was proved for functions with Polish $0$-dimensional domain (\cite[Proposition 5.36 and Corollary 6.6]{phdcarroy}).

\begin{proposition}[Maximum functions]\label{Maxfunctions} \index[IdxNotion]{maximum function proposition}
For all $\alpha<\omega_1$:
\begin{enumerate}
\item the function $\Maximalfct{\alpha}$ is a maximum of $\sC_{\leq\alpha}$,
\item the function $\pgl \Maximalfct{\alpha}$ is a maximum for simple functions in $\sC_{\leq \alpha+1}$,
\item for all $n\in\N$, $(n+1)\Minimalfct{\alpha+1}$ is a maximum among functions of $\CB$-type $(\alpha+1,n+1)$ with compact domains.
\end{enumerate}
\end{proposition}

\begin{proof}
First notice that if $\alpha\leq \beta$, then we have $\Maximalfct{\alpha}\leq\Maximalfct{\beta}$ and $\Minimalfct{\alpha+1}\leq\Minimalfct{\beta+1}$.
For $\alpha=0$, we have $\Maximalfct{0}=\emptyset$ and $\Minimalfct{1}=\pgl\Maximalfct{0}=\id_{\set{\iw{0}}}\equiv\id_{\set{0}}$, so all items follows from \cref{LocallyConstantFunctions}.

\smallskip

We prove the first two items simultaneously by strong induction on $\alpha$: suppose they both hold for all $\beta<\alpha$. 
To see that $\Maximalfct{\alpha}$ is a maximum in $\sC_{\leq\alpha}$, let $f\in\sC$ with $\CB(f)\leq\alpha$.
By the Decomposition \cref{JSLdecompositionlemma}, $f$ is locally simple. If $\alpha$, is limit $f=\bigsqcup_i f_i$ with $f_i$ simple and $\CB(f_i)=\beta_i+1<\alpha$ and so by induction hypothesis the second item ensures that $f_i\leq \pgl \Maximalfct{\beta_i}\leq \Maximalfct{\beta_i+1}$. If $\alpha$ is successor, $f=\bigsqcup_i f_i$ with $f_i$ simple and $\CB(f_i)=\beta+1=\alpha$ and again the induction hypothesis implies that $f_i\leq \pgl \Maximalfct{\beta}$. In both, cases we get $\gl_{i}f_i\leq \Maximalfct{\alpha}$ and so $f\leq \Maximalfct{\alpha}$ by \cref{Gluingasupperbound}.

Now take $f$ simple in $\sC_{\leq \alpha+1}$ and call $y$ its distinguished point.
By \cref{Pgluingofraysasupperbound} we have $f\leq\pgl_{j\in\N}\ray{f}{y,j}$, but by simplicity of $f$ we also have $\CB(\ray{f}{y,j})\leq\alpha$ for all $j\in\N$. As $\Maximalfct{\alpha}$ is a maximum in $\sC_{\alpha}$, we get $\ray{f}{y,j}\leq\Maximalfct{\alpha}$ for all $j\in \N$ and so $f\leq \pgl\Maximalfct{\alpha}$ by \cref{Pgluingasupperbound}.

\smallskip

We finally prove the second point by a double induction. Suppose that for all $\beta<\alpha$, if $f\in \sC$ has compact domain and $\CB$-type $(\beta+1,n+1)$
for some natural $n$, then $f\leq(n+1)\Minimalfct{\beta+1}$.

Take now $f$ with compact domain and $\CB(f)=\alpha+1$. Suppose first that $f$ is simple, that is $n=0$.
Then for all $j\in \N$ the $\CB$-rank of the $\ray{f}{y,j}$ is at most $\alpha$ by \cref{CBrankofPgluingofregularsequence2simple}.
Moreover, each ray has compact domain, so it is either empty or $\tp(\ray{f}{y,j})=(\beta_j+1,n_j+1)$ for some $\beta_j<\alpha$ and $n_j\in\N$ by \cref{CBbasics0}.
In any case, by induction hypothesis the sequence $(\ray{f}{y,j})_j$ is reducible by pieces to the sequence
$(\Minimalfct{\alpha_n})_{n\in\N}$, where $(\alpha_n)_n$ is either constant if $\alpha$ is itself successor, or cofinal in $\alpha$ if it is limit.
In any case, by \cref{Pgluingasupperbound} we obtain $f\leq\Minimalfct{\alpha+1}$, as desired.

Suppose now that $tp(f)=(\alpha+1,n+1)$ for some $n>0$, then by \cref{FiniteDegreeAreFinGl} we have $f\equiv\gl_{i\leq n}f_i$ with $f_i$ simple for all $i\leq n$.
Now using the simple case, we indeed get $f\leq (n+1)\Minimalfct{\alpha+1}$.
\end{proof}

\section{Pointed gluing as a lower bound and minimum functions}

Given a sequence $(A_n)_{n\in\N}$ of subsets of a space $A$, and a point $a\in A$, we say that $(A_n)_n$ \emph{converges} to $a$ and we write $A_i\rao a$\index[IdxSymb]{__@{$A_i\rao a$}}
when for any open neighborhood $U$ of $a$, there is a natural number $m$ such that for all $n\geq m$ we have $A_n\subseteq U$. For instance, in any space $B$, the sequence of rays $(\ray{B}{y,j})_{j\in \N}$ at some point $y$ converges to $y$.

Characterizing when a pointed gluing reduces continuously to a given function appears to be considerably harder. The following rough criterion is sufficient for the sequel.

\begin{lemma}\label{Pgluingaslowerbound}Let $f:A\to B$ be a function between metrizable spaces and $(g_n)_{n\in\N}$ be a sequence in $\functionsonbaire$.

Assume that there is a point $x\in A$ and a sequence $(A_n)_{n\in\N}$ of \ynote*{unnecessary \enquote{clopen} hypothesis removed}{subsets of $A$} satisfying:
\begin{enumerate}
\item For all $n\in\N$, $f(x)\notin \overline{f(A_n)}$, 
\item The sets $f(A_n)$, for $n\in\N$, form a relative clopen partition, 
\item $A_n\rao x$,
\item For all $n\in\N$, we have $g_n\leq f\restr{A_n}$.
\end{enumerate}
Then $\pgl_{n\in \N}g_n\leq f$.
\end{lemma}

\begin{proof}
Choose for every $n\in\N$ a reduction $(\sigma_n, \tau_n)$ from $g_n$ to $f\restr{A_n}$.
We define $\sigma$ and $\tau$ as follows: set $\sigma(\iw{0})=x$, $\sigma((0)^n\conc(1)\conc x')=\sigma_{n}(x')$ if $x'\in \dom \sigma_n$, and $\tau(f(x))=\iw{0}$,
$\tau(y)=(0)^n\conc(1)\conc\tau_n(y)$ if $y\in \dom(\tau_n)$. Note that $\tau$ is well-defined since the sets $\dom(\tau_n)\subseteq f(A_n)$ are pairwise disjoint and do not contain $f(x)$.

The pair $(\sigma,\tau)$ is the desired reduction, it remains to show that it is continuous.
To see that $\sigma$ is continuous, we proceed using \cref{prop:sufficientcondforcont} with $U=\dom \sigma \setminus \{\iw{0}\}$. 
As $\sigma$ is continuous on $U$ by \cref{lem:ContUnion}, we check the second condition of \cref{prop:sufficientcondforcont}.
Any sequence $(x_n)_n$ in $\dom \sigma \setminus \{\iw{0}\}$ converging to $\iw{0}$ yields by definition
a sequence $(k_n)_n$ in $\N$ such that $k_n\to \infty$ and $x_n\in N_{(0)^{k_n}\conc(1)}$ for all $n\in\N$.
Hence $\sigma(x_n)\in A_{k_n}$ for all $n$ by definition of $\sigma$ and since $A_n\rao x$, it follows that $\sigma(x_n)\rao x$, which proves that $\sigma$ is continuous.

To prove the continuity of $\tau$ we use \cref{prop:sufficientcondforcont} again, this time with $U=\dom \tau \setminus \set{f(x)}$. 
As the sets $f(A_n)$ form a relative clopen partition, so do the sets $\dom \tau_n$, so $\tau$ is continuous on $U$ by \cref{lem:ContUnion}. To prove the second condition, take a sequence $(y_n)_n$ in $\im (f\sigma) \setminus\set{f(x)}$ such that $y_n\rao f(x)$.
For all $n$, let $k_n$ be the unique natural number such that $y_n\in f(A_{k_n})$. We need to show that $\tau(y_n)\rao\iw{0}$, and therefore, by definition of $\tau$, it is enough to see that $k_n\rao\infty$.
But this is indeed the case, since otherwise $(y_n)_n$ would admit a subsequence included in a set $f(A_i)$ for some $i\in\N$
and therefore we would have $f(x)\in \overline{f(A_i)}$, a contradiction.
\end{proof}

The conditions of the previous statement are quite difficult to satisfy in practice. In fact, we exclusively make use of the following specific case which relies on conditions that are simpler to verify, yet apparently much stronger. 
\index[IdxNotion]{lower bound condition!pointed gluing}
\index[IdxNotion]{pointed gluing operation!lower bound condition}
\begin{proposition}[Pointed gluing as lower bound]\label{Pgluingaslowerbound2}
Let $f:A\to B$ be a continuous function in $\functionsonbaire$ and $(g_i)_{i\in\N}$ be a sequence in $\functionsonbaire$.

Assume that there is a point $x\in A$ such that 
for all $i\in\N$ and all open neighborhood $U\ni x$
there is a continuous reduction $(\sigma,\tau)$ from $g_i$ to $f$ such that $\im (\sigma )\subseteq U$ and $f(x)\notin\overline{\im(f\sigma)}$.
Then $\pgl_{i\in \N}g_i\leq f$, moreover the reduction can be chosen with
$\sigma(\iw{0})=x$.

In fact, $\pgl_{i\in\N}g_i\leq f\restr{V}$ for all clopen neighborhood $V$ of $x$.
\end{proposition}

\begin{proof}
We define by induction a sequence $(A_n)_n$ of 
 subsets of $A$ such that the sets $f(A_n)$ are separated by open sets,
for all $n\in\N$ we have $f(x)\notin \overline{f(A_n)}$, $A_n\subseteq \nbhd{x}{n}$ and there exists a continuous reduction $(\sigma_n,\tau_n)$ from $g_n$ to $f\restr{A_n}$.
This sequence satisfies the conditions of \cref{Pgluingaslowerbound}, so the result follows.

Assume that the sets $(A_i)_{i<n}$ have been defined as above.
Let $k$ be such that $N_{f(x)\restr{k}}$ is disjoint from the closed set $\bigcup_{i<n}\overline{f(A_i)}$.
Use continuity of $f$ at $x$ to choose a clopen neighborhood $U$ of $x$ included in $N_{x\restr{n}}$ such that $f(U)\subseteq N_{f(x)\restr{k}}$.
The hypothesis guarantees the existence of a continuous reduction $(\sigma_n,\tau_n)$ from $g_n$ to $f\restr{U}$ with $f(x)\notin\overline{\im(f\sigma_n)}$.
This allows us to choose a clopen neighborhood $W\ni f(x)$ disjoint from $\overline{\im(f\sigma_n)}$ and we set $A_n=U\setminus f^{-1}(W)$.
The set $A_n\subseteq U\subseteq N_{x\restr{n}}$ is clopen as a difference of two clopen sets, $f(A_n)\cap W=\emptyset$ so $f(x)\notin \overline{f(A_n)}$,
$f(A_n)\subseteq N_{f(x)\restr{k}}\setminus W$ so $N_{f(x)\restr{k}}\setminus W$ is an open set witnessing separation of the sets $A_n$.
Finally, notice that $(\sigma_n,\tau_n)$ is actually a reduction from $g_n$ to $f\restr{A_n}$ since $\im \sigma_n \subseteq A_n$.
 \end{proof}

Seeing pointed gluing as an upper bound gave us maximum functions in $\sC_{\leq \alpha}$ (\cref{Maxfunctions}); similarly the lower bound criterion gives us minimum functions in $\sC_{\geq \alpha+1}$. %

\begin{proposition}[Minimum functions]\label{Minfunctions}\index[IdxNotion]{minimum function proposition}
For all $\alpha<\omega_1$ and $f\in \sC$, if $\alpha+1\leq \CB(f)$, then $\Minimalfct{\alpha+1}\leq f$.
\end{proposition}

\begin{proof}
We prove by strong induction for every $\alpha<\omega_1$, the following statement: $\Minimalfct{\alpha+1}\leq f$ for every simple function $f\in \sC_{\alpha+1}$. To see that this is enough, we remark that if this statement holds for all $\beta<\alpha$, then actually for all $f\in \sC_{<\alpha}$ if $\beta+1\leq\CB(f)$ then $\Minimalfct{\beta+1}\leq f$. 

To prove this remark, suppose that  $f\in \sC$ with $\CB(f)<\alpha$. Since $f$ is locally simple by the Decomposition \cref{JSLdecompositionlemma}, we can write $f=\bigsqcup_{i\in I} f_i$ with $f_i$ simple, and $\sup_{i\in I} \CB(f_i)=\CB(f)$ by \cref{CBrankofclopenunion}. So if $\beta+1\leq \CB(f)$, $\beta+1\leq \CB(f_i):=\beta_i+1$ for some $i\in I$ and we have $\Minimalfct{\beta+1}\leq \Minimalfct{\beta_i+1} \leq f_i\leq f$, where $ \Minimalfct{\beta_i+1} \leq f_i$ is granted by the statement, which proves the remark.

For $\alpha=0$, note that $\Minimalfct{1}$ continuously reduces to any non\-empty function.
So suppose that $\alpha>0$ and that the statement holds for every $\beta<\alpha$.

Take a simple function $f\in\sC_{\alpha+1}$ and let $y\in\im f$ be the distinguished point of $f$.
Seeing that $\Minimalfct{\alpha+1}$ is a pointed gluing, we seek to apply \cref{Pgluingaslowerbound2} for some chosen point $x\in\CB_\alpha(f)$. So let $U\ni x$ be any open set of $\dom f$. Notice first that as $x\in\CB_\alpha(f)$ and $x\in U$ we get that $f_U=f\restr{U}$ is simple and $\CB(f_U)=\alpha+1$ by \cref{CBbasics0}~\cref{CBbasicsfromJSL2}. Moreover by \cref{CBrankofPgluingofregularsequence2simple} we have $\sup _n\CB(\ray{f_{U}}{y,n})=\alpha$.

Successor case: $\alpha=\beta+1$ and $\Minimalfct{\alpha+1}=\pgl \Minimalfct{\beta+1}$. Since $\sup _n\CB(\ray{f_{U}}{y,n})=\beta+1$, there exists $N$ such that $\beta+1=\CB(\ray{f_U}{y,N})$. By the induction hypothesis combined with our first remark, we get $\Minimalfct{\beta+1}\leq \ray{f_{U}}{y,N}$.
Notice that if $(\sigma, \tau)$ witnesses this reduction then both $\im \sigma \subseteq U$ and $\overline{\im f\sigma}\subseteq \ray{B}{y,N}$ hold. This ensures that $\Minimalfct{\alpha+1}=\pgl \Minimalfct{\beta+1} \leq f$ by \cref{Pgluingaslowerbound2}.

\ynote*{minimum functions: limit case}{
Limit case: $\alpha$ is limit and by \cref{Pgluingaslowerbound2} it is sufficient to prove that for every $\beta<\alpha$ we have $\Minimalfct{\beta+1}\leq f_U$ witnessed by a reduction $(\sigma, \tau)$ such that $y\notin \overline{\im f\sigma}$. Fix $\beta<\alpha$ and note that since $\sup _n\CB(\ray{f_{U}}{y,n})=\alpha$, there exists $N$ with $\beta<\CB(\ray{f_U}{y,N})$. It follows that $\CB_{\beta}(\ray{f_U}{y,N})\neq \emptyset$ and so we can choose $z\in \CB_{\beta}(\ray{f_U}{y,N})$.  Let $\gamma$ be the least ordinal such that $z\in \CB_{\gamma}(\ray{f_U}{y,N})\setminus \CB_{\gamma+1}(\ray{f_U}{y,N})$ and note that $\beta\leq \gamma<\alpha$. Since $z$ is isolated in $\CB_{\gamma}(\ray{f_U}{y,N})$ we can find a clopen neighborhood $W$ of $z$ on which $g=(\ray{f_U}{y,N})\restr{W}$ is simple of rank $\gamma+1$.  By the induction hypothesis and our first remark, we obtain $\Minimalfct{\beta+1}\leq g\leq \ray{f_U}{y,N}$, which allows to conclude.}
\end{proof}

\section{General Structure}

We can now complete the analysis of all functions in $\sC$ with compact domain, thus generalizing \cite[Theorem 4.2]{carroy2013quasi}.

\index[IdxNotion]{functions with compact domain (classification)}
\begin{theorem}[Classification of functions with compact domains]\label{Compactdomains}
Assume that $f$ and $g$ are two functions in $\sC$ with a compact domain, then $f$ continuously reduces to $g$ if and only if $tp(f)\leq_{lex}tp(g)$.
More specifically $f\equiv(n+1)\Minimalfct{\alpha+1}$, where $tp(f)=(\alpha+1,n+1)$.

In particular, continuous reducibility is a pre-well-order\footnote{A \emph{pre-well-order} is a linear \wqo{} (for more on this see \cite[Section 34.A]{kechris}).}
of length $\omega_1$ on functions in $\sC$ with compact domain.
\end{theorem}

\begin{proof}
By definition for $\alpha<\beta$ we have $\Minimalfct{\alpha+1}\leq\Minimalfct{\beta+1}$ and since $\CB(\Minimalfct{\alpha+1})=\alpha+1<\beta+1=\CB(\Minimalfct{\beta+1})$
we get $\Minimalfct{\beta+1}\not\leq\Minimalfct{\alpha+1}$, so the functions $\Minimalfct{\alpha+1}$ for $\alpha<\omega_1$ form a well-order of length $\omega_1$.

Therefore by \cref{Gluingcohomomorphism} it is enough to prove that for $f\in\sC$ with compact domain then $f\equiv(n+1)\Minimalfct{\alpha+1}$, where $tp(f)=(\alpha+1,n+1)$.
Now by \cref{FiniteDegreeAreFinGl} we only have to show it for simple functions, which follows from \cref{Maxfunctions,Minfunctions}.
\end{proof}

In particular we obtain the promised alternative proof that the definition of $\Minimalfct{\alpha+1}$ does not depend on the choice of the cofinal sequence when $\alpha$ is limit.

The results on maximum and minimum functions can be leveraged to obtain important information about the structure of continuous reducibility on the whole of $\sC$.
As announced, this generalizes \cite[Theorem 5.2]{carroy2013quasi} and we (still) refer to it as the General Structure Theorem.\index[IdxNotion]{General Structure Theorem}

\usetikzlibrary{decorations.pathreplacing}

\begin{figure}[ht]
\centering
\index[IdxFig]{Diagram of the General Structure}
\tikzset{every picture/.style={line width=0.75pt}} 

\begin{tikzpicture}[remember picture, x=0.5pt,y=0.5pt,xshift=0,
 every node/.append style={anchor=center, draw=none, inner sep=3pt, rounded corners=3pt},
]

\tikzset{%
  show curve controls/.style={
    postaction={
      decoration={
        show path construction,
        curveto code={
          \draw [blue] 
            (\tikzinputsegmentfirst) -- (\tikzinputsegmentsupporta)
            (\tikzinputsegmentlast) -- (\tikzinputsegmentsupportb);
          \fill [red, opacity=0.5] 
            (\tikzinputsegmentsupporta) circle [radius=.5ex]
            (\tikzinputsegmentsupportb) circle [radius=.5ex];
        }
      },
      decorate
}}}
\tikzset{
    fnode/.style={anchor=center, draw, inner sep=3pt, rounded corners=3pt
        },
      edraw/.style={fill=gray!50,draw=none}
}
\def\scale{1.}
\def\xshift{0cm}
\def\X{70}
\def\O{\X}
\def\o{30}
\def\y{60}
\def\r{0.3}
\def\e{2}
\def\last{3}
\def\base{1.35}

\coordinate(0) at (0,0);
\coordinate(1) at (0.5*\X,0) ;  
\coordinate(l0) at (1.6*\X,0); 


\foreach \i in {1,2,...,\last}{
\coordinate (k\i)  at ({\X+\O+10},0);
\pgfmathparse{(\base^\i)*\X+\i*\o}
 \xdef\O{\pgfmathresult}
 \coordinate(l\i) at (\X+\O+10,{(1+0.003*\O)*\y});
}

\xdef\lastp{\the\numexpr(\last)+1}

\foreach \i in {\last} 
{
\pgfmathparse{(\base^(\i-1))*\X+(\i-1)*\o}
\xdef\O{\pgfmathresult}
\coordinate(k\i) at ({\X+\O},0);
}

\node (temp) at ($(l\last)-(k\last)$) {};
\pgfgetlastxy{\Xt}{\Yt}
\def\Xlim{}
\pgfmathsetmacro{\Xlim}{\scale*\Xt+\scale*\Yt}

\node (l) at ($(l\last)+(\scale*\Yt,-\scale*\Yt)$) {};

%

%

\fill[fill=gray!50,draw=none] ($(1)!0.5!(l0)$) ellipse ({0.75*\X} and {\r*\X});
\def\Rtemp{0}
\def\Atemp{0}

\foreach \i in {1,2,...,\last}{
	\coordinate (temp) ($(k\i)-(l\i)$);
	\pgfgetlastxy{\Xt}{\Yt}
	\pgfmathsetmacro{\Rtemp}{veclen(\scale*\Xt+\xshift,\scale*\Yt)}
	\pgfmathsetmacro{\Atemp}{atan2(\scale*\Yt,\scale*\Xt+\xshift)}
  	\draw[edraw, rotate=\Atemp] ($(k\i)!0.5!(l\i)$) ellipse ({15+\Rtemp} and {\r*\X});
	}

\node[fnode](0) at (0) {$\Maximalfct{0}$} ;
\node[fnode](1) at (1)   {$\Minimalfct{1}$};
\node[fnode](l0) at (l0) {$\Maximalfct{1}$};

\draw (1) -- (l0) node[midway, fill=gray!50, inner sep=0.5pt] {\scriptsize $\sC_1$};

\foreach \i in {1,2,...,\last}{
\def\j{0}
\pgfmathsetmacro{\j}{\the\numexpr(\i)+1}
\node[fnode](l\i) at (l\i) {$\Maximalfct{\j}$};
\node[fnode](k\i) at (k\i)   {$\Minimalfct{\j}$};
\draw (k\i) -- (l\i) node[midway, fill=gray!50, inner sep=1pt] {\scriptsize $\sC_\j$};
}
\xdef\lastpp{\the\numexpr(\last)+2}


\node[fnode](l) at (l)   {$\Maximalfct{\omega}$};

\draw  (0) -- (1);
\draw (l0) -- (k1);
\draw (l1) -- (l2) -- (l3); 
\draw (k1) -- (k2) -- (k3);

\draw  (k\last) -- ($(k\last)!0.42!(l)$);
\draw[dotted]($(k\last)!0.46!(l)$)  --($(k\last)!0.54!(l)$);
\draw ($(k\last)!0.58!(l)$) -- (l);

\draw  (l\last) -- ($(l\last)!0.4!(l)$);
\draw[dotted]($(l\last)!0.45!(l)$)  --($(l\last)!0.55!(l)$);
\draw ($(l\last)!0.6!(l)$) -- (l);

\draw[-] (l1) .. controls ([yshift=-35, xshift=5] l1) and ([xshift=-20, yshift=10] k3) .. (k3);


\def\xshift{0cm}
\def\yshift{0}
\def\O{\X}
\def\o{50}
\def\Y{-180}
\def\r{0.4}

\coordinate(ll0) at (1.3*\X,\Y); 


\foreach \i in {1,2,...,\last}{
\coordinate (kk\i)  at ({\X+\O},\Y);
\pgfmathparse{(\base^\i)*\X+\i*\o}
 \xdef\O{\pgfmathresult}
 \coordinate(ll\i) at (\X+\O,{(1+0.003*\O)*\y+\Y});
}

\xdef\lastp{\the\numexpr(\last)+1}


\node (temp) at ($(ll\last)-(kk\last)$) {};
\pgfgetlastxy{\Xt}{\Yt}
\node (ll) at ($(ll\last)+(\Yt,-\scale*\Yt)$) {};


%
\def\Rtemp{0}
\def\Atemp{0}

\foreach \i in {1,2,...,\last}{
	\coordinate (temp) ($(kk\i)-(ll\i)$);
	\pgfgetlastxy{\Xt}{\Yt}
	\pgfmathsetmacro{\Rtemp}{veclen(\scale*\Xt+\xshift,\scale*\Yt+\yshift)}
	\pgfmathsetmacro{\Atemp}{atan2(\scale*\Yt+\yshift,\scale*\Xt+\xshift)}
  	\draw[edraw, rotate=\Atemp] ($(kk\i)!0.5!(ll\i)$) ellipse ({15+\Rtemp} and {\r*\X});
	}

\node[fnode](ll0) at (ll0) {$\Maximalfct{\lambda}$};


\foreach \i in {1,2,...,\last}{
\def\j{0}
\pgfmathsetmacro{\j}{\the\numexpr(\i)}
\node[fnode](ll\i) at (ll\i) {$\Maximalfct{\lambda+\j}$};
\node[fnode](kk\i) at (kk\i)   {$\Minimalfct{\lambda+\j}$};
\draw (kk\i) -- (ll\i) node[midway, fill=gray!50, inner sep=1pt] {\scriptsize $\sC_{\lambda+\j}$};
}
\xdef\lastpp{\the\numexpr(\last)+2}


\node[fnode](ll) at (ll)   {$\Maximalfct{\lambda+\omega}$};

\draw (ll0) -- (kk1);
\draw (ll1) -- (ll2) -- (ll3); 
\draw (kk1) -- (kk2) -- (kk3);

\draw  (kk\last) -- ($(kk\last)!0.42!(ll)$);
\draw[dotted]($(kk\last)!0.47!(ll)$)  --($(kk\last)!0.53!(ll)$);
\draw ($(kk\last)!0.58!(ll)$) -- (ll);

\draw  (ll\last) -- ($(ll\last)!0.4!(ll)$);
\draw[dotted]($(ll\last)!0.46!(ll)$)  --($(ll\last)!0.54!(ll)$);
\draw ($(ll\last)!0.6!(ll)$) -- (ll);

\node(dddots) at ($(ll)+(\X,0)$) {$\cdots$};
\draw (ll) -- (dddots);

\draw[-] (ll1) .. controls ([yshift=-35, xshift=5] ll1) and ([xshift=-20, yshift=10] kk3) .. (kk3);


\coordinate(c1) at  ($(l)-(310,30)$);
\coordinate(c2) at ($(c1)-(80,80)$);
\coordinate(c3) at ($(c2)-(0,22)$);

\draw[-] (l).. controls ([xshift=20] l) and ([xshift=20] $(l)-(0,30)$) .. ($(l)-(0,30)$) .. controls ([xshift=-20] $(l)-(0,30)$) and ([xshift=30] c1) ..  (c1)  .. controls ([xshift=-20] c1) and ([yshift=20] c2) ..  (c2);
\draw[dotted] (c2) -- (c3);
\draw (c3) .. controls ([yshift=-5]c3) and ([xshift=-15] ll0) .. (ll0);

\end{tikzpicture}

\caption{The general structure of continuous reducibility on $\sC$.}
\end{figure}

\begin{theorem}[General Structure]\label{JSLgeneralstructure}
For all functions $f$ and $g$ in $\sC$, and all ordinals $\lambda<\omega_1$ limit or null, we have:
\begin{enumerate}
\item If $\CB(f)\leq \CB(g)=\lambda$, then $f\leq g$.
\item For all $n\in \N$, $\CB(f)=\lambda+n$ and $\lambda+2n+1\leq \CB(g)$ imply $f\leq g$.
\end{enumerate}
In particular, $2\CB(f)<\CB(g)$ implies $f\leq g$.
\end{theorem}

\begin{proof}
Proceed by induction on $\lambda<\omega_1$. Suppose that the theorem is proven for all $\alpha<\lambda$ with $\lambda$ limit or null.

\smallskip

We first prove the second item. As $f\leq\Maximalfct{\lambda+n}$ by \cref{Maxfunctions} and $\Minimalfct{\lambda+2n+1}\leq g$ by \cref{Minfunctions}, it is enough to prove that
$\Maximalfct{\lambda+n}\leq \Minimalfct{\lambda+2n+1}$ for all $n\in\N$, and to do so we proceed by induction on $n\in\N$.
For $n=0$, if $\lambda=0$ then $\Maximalfct{0}=\emptyset\leq\Minimalfct{1}$, so we can suppose that $\lambda\neq0$.
Take $(\alpha_n)_n$ cofinal in $\lambda$ and $(\beta_n)_n$ an enumeration of $\lambda$ then, by induction hypothesis, for some injection $p:\N\rao\N$ we have
$\Maximalfct{\alpha_n}\leq\Minimalfct{\beta_{p(n)}+1}$ so by \cref{Gluingasupperbound,GluinglowerthanPgluing} we get
\(\Maximalfct{\lambda}\equiv\gl_n\Maximalfct{\alpha_n}\leq\gl_n\Minimalfct{\beta_n+1}\leq\pgl_n\Minimalfct{\beta_n+1}\equiv\Minimalfct{\lambda+1}.\)

If now $\alpha=\lambda+(n+1)$ then using the induction hypothesis, \cref{Pgluingasupperbound,GluinglowerthanPgluing} we see that
$\Maximalfct{\alpha}\equiv\omega\pgl\Maximalfct{\lambda+n}\leq\pgl\pgl\Minimalfct{\lambda+2n+1}\equiv\Minimalfct{\lambda+2n+3}= \Minimalfct{\lambda+2(n+1)+1}$.

\smallskip

To see the first item, take a function $g:A\to B$ of $\CB$-rank $\lambda$. By \cref{Maxfunctions} $f\leq\Maximalfct{\lambda}$ for any $f\in \sC_{\leq \alpha}$, so it is enough to show that if $\Maximalfct{\lambda}\leq g$.
If $\lambda=0$ then $g$ is the empty function, so suppose that $\lambda$ is limit. 

We are going to find a sequence $(s_n)_{n\in\N}\subseteq\N^{<\N}$ of finite sequences pairwise incomparable for the prefix relation such that the sequence $(\CB(g\corestr{N_{s_n}}))_n$
is either constant equal to $\lambda$ or strictly below $\lambda$ and cofinal in $\lambda$. Thanks to the induction hypothesis, an application of \cref{Gluingaslowerbound}
to the (pairwise disjoint) clopen sets $(N_{s_n})_n$ allows then to conclude.

Consider the tree $T=\set{s\in\N^{<\N}}[\CB(g\corestr{N_s})=\lambda]$, notice that $T\neq\emptyset$ because it contains at least the empty sequence.
If $[T]$ is infinite then an application of \cref{InfiniteEmbedOmega} allows to find the desired sequence, so we can suppose that $[T]$ is finite.
Let $F$ be the set of $\sqsubset$-minimal elements of $\N^{<\N}\setminus T$. Then $\set{\CB(g\corestr{N_s})}[ s\in F]$ is a subset of $\lambda$ and we claim that it is cofinal in $\lambda$, which allows us to find the desired sequence. Towards a contradiction assume that for some $\beta<\lambda$ we have $\CB(g\corestr{N_s})<\beta$ for all $s\in F$. Then, by \cref{CBbasics0}~\cref{CBbasicsfromJSL2},  $\CB_{\beta}(g)\cap g^{-1}(N_s)=\emptyset$ for all $s\in F$ and so $\CB_{\beta}(g)\subseteq g^{-1}([T])$. But as $[T]$ is finite, we have $\CB_{\beta+1}(g)=\empty$ and so $\CB(g)\leq \beta+1$, a contradiction.
\end{proof}

\begin{remark}
Note that we actually have $\Maximalfct{n}\leq \Minimalfct{2n}$ for all $n\in\N$ and so the General Structure~\cref{JSLgeneralstructure} for functions of finite $\CB$-ranks is slightly stronger: for every $f,g$ with finite $\CB$-rank, $2\CB(f)\leq \CB(g)$ implies $f\leq g$.
\end{remark}

As a consequence we can finally prove that \cref{Introthm:LevelsAreFinitelyGenerated} implies \cref{Introthm:BQO,Introthm:BQOonScat}.
When a \bqo{} is also a partial order we call it a \emph{better partial order}.
We will use that a sum of \bqo{} indexed by a better partial order is itself still a \bqo{} (\cite[(9.14)]{simpsonbqo}).
The following result has the exact same proof as \cite[Theorem 5.4]{carroy2013quasi}.
\index[IdxNotion]{better-quasi-order (\bqo{})}

\begin{proposition}\label{FGgivesBQO_2}
If $\sC_{\beta}$ is \bqo{} for all $\beta<\alpha$, then $\sC_{<\alpha}$ is \bqo{}.
In particular, \cref{Introthm:LevelsAreFinitelyGenerated} implies \cref{Introthm:BQO,Introthm:BQOonScat}.
\end{proposition}

\begin{proof}
Consider the partial order $(\N, \leq^\bullet)$, where 
\[
m\leq^\bullet n \Lglra m=n \text{ or }2m<n.
\]
It is immediate to see that $(\N,\leq^\bullet)$ is a partial order and a \wqo{} and it is  \bqo{} by \cite[Example 1.6]{CarroyYPFromWell}. 
Similarly, we define a partial order $\leq^\bullet$ on $\omega_1$ by $\alpha_0\leq^\bullet \alpha_1$ if and only if $\alpha_0=\alpha_1$ or $2\alpha_0<\alpha_1$. As any ordinal $\alpha$ can be uniquely written as $\alpha=\lambda+n$ with $\lambda$ limit or null and $n\in\N$ and $2\alpha=\lambda+2n$, we have: 
\[
\lambda_0+n_0 \leq^\bullet \lambda_1+n_1 \Lglra \lambda_0<\lambda_1\mbox{ or } \bigl(\lambda_0=\lambda_1\mbox{ and } n_0\leq^\bullet n_1 \bigr).
\]
As this partial order is the sum of the \bqo{} $(\N,\leq^\bullet)$ along the well-ordered set of countable limit ordinals, it is \bqo{}.
Finally, consider the sum $S_{<\alpha}=\sum_{\beta<\alpha} \sC_{\beta}$ of the levels $\sC_\beta$, quasi-ordered by continuous reducibility, along the better partial order $\leq^\bullet$. By our hypothesis, $\sC_\beta$ is \bqo{} for all $\beta<\alpha$, so $S_{<\alpha}$ is \bqo{} too. 
 
Now observe that by the General Structure \cref{JSLgeneralstructure}, the mapping $h:\sC_{<\alpha}\to S_{<\alpha}$ given by $h(f)= (\CB(f),f)$ is a co-homomorphism. Namely, we have $f\leq g$ whenever $h(f)\leq h(g)$ in $S$, which by definition is equivalent to 
\[
\CB(f)<^\bullet \CB(g) \text{ or } (\CB(f)=\CB(g) \text{ and } f\leq g).
\]
Therefore, it follows that continuous reducibility is \bqo{} on $\sC_{<\alpha}$.

In particular, assuming \cref{Introthm:LevelsAreFinitelyGenerated} holds, then $\sC_\alpha$ is \bqo{} for all $\alpha<\omega_1$ by \cref{SecondstepforBQOthm}.  So $\sC$ is \bqo{} (\cref{Introthm:BQOonScat}) and hence \cref{Introthm:BQO} follows by \cref{FirststepforBQOthm}. 
\end{proof}

The General Structure \cref{JSLgeneralstructure} has many interesting consequences, we point out two of them that will reveal useful in the sequel.

\begin{corollary}\label{ConsequencesGeneralStructureThm}
Let $\lambda$ be a limit ordinal.
\begin{enumerate}
\item If $(f_n)_{n\in\N}$ is in $\sC_{<\lambda}$, then $\pgl_nf_n\leq\Minimalfct{\lambda+1}$,
and if moreover $(\CB(f_n))_n$ is regular and $\sup_n(\CB(f_n))=\lambda$, then $\pgl_n(f_n)\equiv \Minimalfct{\lambda+1}$.
\item If $\CB(f)\geq \lambda+2$, then $\pgl\Maximalfct{\lambda}\leq f$,\label{ConsequencesGeneralStructureThm2}
\end{enumerate}
\end{corollary}

\begin{proof} 
Fix any increasing cofinal sequence $(\alpha_n)_{n\in\N}$ in $\lambda$. There is an increasing sequence $(k_n)_{n\in\N}\in\cN$ such that $2\CB(f_n)\leq\alpha_{k_n}$ for all $n$.
So by \cref{JSLgeneralstructure} we have $f_n\leq \Minimalfct{\alpha_{k_n}+1}$ and in turn $\pgl_nf_n\leq\Minimalfct{\lambda+1}$ by \cref{Pgluingasupperbound}.
If moreover $(\CB(f_n))_n$ is regular and $\sup_n(\CB(f_n))=\lambda$, then by \cref{CBrankofPgluingofregularsequence1} $\CB(\pgl_nf_n)=\lambda+1$,
so by \cref{Minfunctions} $\Minimalfct{\lambda+1}$, being a minimum, it reduces to $\pgl_nf_n$.

For the second item, by \cref{JSLgeneralstructure} we have $\Maximalfct{\lambda}\leq \Minimalfct{\lambda+1}$ and so $\pgl \Maximalfct{\lambda}\leq \pgl \Minimalfct{\lambda+1} =\Minimalfct{\lambda+2}$. Since $\CB(f)\geq\lambda+2$, we have $\Minimalfct{\lambda+2}\leq f$ by \cref{Minfunctions} and so $\pgl \Maximalfct{\lambda}\leq f$, as desired.

For the last point, suppose towards a contradiction that they are equivalent and let $\Minimalfct{\lambda+1}=\pgl_n\Minimalfct{\alpha_n+1}$ for $(\alpha_n)_n$ cofinal in $\lambda$. By \cref{Rigidityofthecocenter}, it follows that $\Maximalfct{\lambda}\leq \gl_{n<M} \Minimalfct{\alpha_n+1}$ for some $M\in \N$, but then $\CB(\Maximalfct{\lambda})=\lambda\leq \sup_{n<M}(\alpha_n +1)<\lambda$. 
\end{proof}

{\cref{LocallyConstantFunctions} gives us the structure of $\sC_1$ (in other words, locally constant functions), and
the General Structure \cref{JSLgeneralstructure} tells us that functions with same limit rank are all continuously equivalent.
The next cases are those of $\sC_2$, and $\sC_{\lambda+1}$ for limit $\lambda$.
Interestingly, these two cases have more in common than just being the next ones to deal with,
see \cref{cor:CenteredSucessor,FGatsuccessoroflimit,OptimalityOfGenerators,qu:isomorphismofthelevels}.


\chapter{Centered functions}\label{sectionCentered}

At this point in the paper, one could wonder if it is possible to use combinations of all the results shown already and analyze all functions in $\sC$
using their decomposition into simple functions guaranteed by \cref{JSLdecompositionlemma}.

As we saw in \cref{CBrankofPgluingofregularsequence1}, the pointed gluing of a sequence of functions is simple when the corresponding sequence of $\CB$-ranks is regular. Moreover simple functions can be decomposed in a sequence of rays at the distinguished point with regular $\CB$-ranks.
However, the latter sequence may not always be regular for continuous reducibility and the pointed gluing of this sequence of rays may not be equivalent to the simple function. 

Scattered functions which are equivalent to a pointed gluing of some $\leq$-regular sequence play an important role in our analysis and enjoy a natural alternative definition.
We call them centered functions and we now study them in detail.

This section culminates with the results that the \bqo{} character of continuous reducibility implies that centered functions are ubiquitous.
Finally, we show that finite generation at levels of $\sC$ ensures a finite number of centered functions of any given rank, which is a crucial step in our proof.

\section{Definition and characterization}\label{ssec:centered_def}

\begin{definition}
Let $A$ be a topological space. A \emph{center} for a function $g:A\to B$ is a point $x\in A$ such that for every neighbourhood $U$ of $x$ we have $g\leq g\restr{U}$.
We say that a function $g:A\to B$ is \emph{centered} if it admits a center.\index[IdxNotion]{centered function}\index[IdxNotion]{function!centered}\index[IdxNotion]{center (of a function)}
\end{definition}

The functions $\Minimalfct{\alpha+1}$ and $\pgl\Maximalfct{\alpha}$ for all $\alpha<\omega_1$ are examples of centered functions ($\iw{0}$ is a center). More generally
saying that a sequence $(f_i)_{i\in \N}$ of functions is \emph{regular} if $\{j\in \N \mid f_i \leq f_j\}$ is infinite
for all $i\in \N$, we have

\begin{fact}\label{Pgluingofregulariscentered}
If $(f_i)_{i\in\N}$ is a regular sequence in $\functionsonbaire$, then $\iw{0}$ is a center for $\pgl_i f_i$.
\end{fact}

\begin{proof}
Set $f:=\pgl_if_i$.
By \cref{Pgluingaslowerbound2}, it is enough to show that for every clopen neighbourhood $U$ of $\iw{0}$ and every $n\in \N$,
there exists a continuous reduction $(\sigma,\tau)$ from $f_n$ to $f$ such that $\im \sigma\subseteq U$ and $\iw{0}\notin\overline{\im(f\sigma)}$.

So let $U\ni \iw{0}$ be clopen and $m\in \N$. By regularity of $(f_i)_i$ we can find $n$ large enough such that $N_{(0)^n}\subseteq U$ and $f_m\leq f_n$, which clearly induces the desired reduction.
\end{proof}

The property of being centered is stable under continuous equivalence as a result of the following fact.

\begin{fact}\label{Centerinvariance}
Let $f,g$ be functions between topological spaces.
Suppose that $x$ is a center for $f$ and that $(\sigma,\tau)$ continuously reduces $f$ to $g$.
\begin{enumerate}
\item If $U$ is a neighborhood of $x$, then $x$ is a center for $f\restr{U}$.\label{Centerinvariance0}
\item For every neighborhood $U$ of $\sigma(x)$, we have $f\leq g\restr{U}$.\label{Centerinvariance1}
\item If moreover $g\equiv f$, then $\sigma(x)$ is a center for $g$.\label{Centerinvariance2}
\item If $(A_i)_{i\in I}$ is an open covering of $\dom g$, then there exists $i\in I$ with $f\leq g\restr{A_i}$. \label{Centerinvariance3}
\end{enumerate}
\end{fact}

\begin{proof}
For the first item, let $V\subseteq U$ be a neighborhood of $x$. We have $f\restr{U}\leq f$ as witnessed by the identity. Moreover $f\leq f\restr{V}$ since $x$ is a center for $f$, so
$f\restr{U}\leq f\restr{V}$ by transitivity.

For the second item, if $U$ is a neighbourhood of $\sigma(x)$, then by continuity of $\sigma$, $\sigma^{-1}(U)$ is a neighbourhood of $x$.
We have $f\restr{\sigma^{-1}(U)}\leq g\restr{U}$ as witnessed by $(\sigma\corestr{U},\tau)$ and $f\leq f\restr{\sigma^{-1}(U)}$ since $x$ is a center for $f$, so $f\leq g\restr{U}$.

If moreover $g\leq f$ then by transitivity $g\leq g\restr{U}$ holds. So in this case $\sigma(x)$ is a center for $g$. This gives the third point.

If now $(A_i)_i$ covers $\dom(g)$, then $\sigma(x)\in A_i$ for some $i\in I$, but as $A_i$ is open, using the second point we obtain the last one.
\end{proof}

We next establish an important interaction between the existence of centers and the property of being scattered. 

\begin{proposition}\label{scatteredhavecocenter}
Suppose that $f:A\to B$ is centered with $A$ metrizable and $B$ Hausdorff.
Then $f$ is scattered if and only if all centers have the same image by $f$.

Moreover when $f$ is scattered, then it is simple and any center of $f$ is mapped to its distinguished point.
\end{proposition} 
\index[IdxNotion]{scattered!function}\index[IdxNotion]{function!scattered}

\begin{proof}
Assume that $x$ is a center for a function $f$, we start by making two remarks. First, $x$ is $f$-isolated if and only if $f$ is constant. Indeed, if $f\restr{U}$ is constant on a neighborhood $U$ of $x$, then the whole $f$ is constant since we have $f\leq f\restr{U}$.

Second, for all ordinal $\alpha$ if $\CB_\alpha(f)\neq \emptyset$ then $x\in \CB_\alpha(f)$ and $x$ is a center for $f\restr{\CB_\alpha(f)}$. To see this, note that for every neighborhood $U$ of $x$ we have $f\leq f\restr{U}$ via some $(\sigma,\tau)$, hence $\sigma(\CB_\alpha(f))\subseteq \CB_\alpha(f))\cap U$ and $f\restr{\CB_\alpha(f)}\leq f\restr{\CB_\alpha(f)\cap U}$ by \cref{CBbasicsfromJSL}. So if $\CB_\alpha(f)$ is nonempty, then so is $\CB_\alpha(f)\cap U$ for all neighborhood $U$ of $x$. It follows that $x$ belongs to the closed set $\CB_\alpha(f)$ and it is a center for $f\restr{\CB_\alpha(f)}$. 

Now assume that $f$ is scattered with $\gamma=\CB(f)$, then by \cref{prop:nlc_implies_nonscattered} we have $\CB_\gamma(f)=\emptyset$. By our second remark, all centers of $f$ belong to $\CB_{\beta}(f)$ for all $\beta<\gamma$. In particular, we must have $\gamma=\alpha+1$ and all centers of $f$ belong to $\CB_\alpha(f)$. Now all centers of $f$ are centers of $f\restr{\CB_\alpha(f)}$ and they are $f\restr{\CB_\alpha(f)}$-isolated. So by our first remark, $f$ is constant on $\CB_\alpha(f)$, hence $f$ is simple and all centers have the same image.

Finally, suppose now that $x_0$ and $x_1$ are centers for $f$ such that $f(x_0)\neq f(x_1)$.
Then using our two previous remarks, we see by induction that $x_0,x_1\in\CB_\alpha(f)$ for all $\alpha$. Therefore $f$ has nonempty perfect kernel and so it is not scattered by \cref{scatterediffemptykernel}.
\end{proof}

Therefore, in particular, if $f\in\sC$ is scattered, then there is a unique point $y$ in its image such that for every center $x$ for $f$ we have $f(x)=y$.
We call this point the {\emph{\cocenter{}}} of $f$.\index[IdxNotion]{function!centered!existence of cocenter}\index[IdxNotion]{cocenter (of a centered function)}
The previous proposition has the following fundamental consequence which highlights a strong form of rigidity of centered functions with respect to continuous reducibility.

\begin{proposition}\label{Rigidityofthecocenter}
Let $f,g\in \sC$ be centered with \cocenter{} $y_f$ and $y_g$ respectively. 
Assume that $f\equiv g$ and that $(\sigma,\tau)$ is a continuous reduction from $f$ to $g$. Then the following holds:
\begin{enumerate}
\item $\tau(y_g)=y_f$,
\item for all $n\in\N$ we have $y_g\notin\overline{g\sigma(\dom \ray{f}{y_f,n})}$,
\item for all $m,n\in\N$ there is $M\geq m$ such that $\ray{f}{y_f,n}\leq \gl_{i=m}^{M}\ray{g}{y_g,i}$, 
\item $(\ray{f}{y_f,n})_{n\in\N}$ is reducible by finite pieces to $(\ray{g}{y_g,n})_{n\in\N}$.
\end{enumerate}
\end{proposition}

\begin{proof}
First item. Let $x$ be a center for $f$. Since $f\equiv g$, $\sigma(x)$ is a center for $g$ by \cref{Centerinvariance} and so $g\sigma(x)=y_g$ by \cref{scatteredhavecocenter}.
Hence $\tau(y_g)= \tau g\sigma(x_f)=f(x_f)=y_f$ as desired. 

\smallskip

Second item: suppose not, then there is a sequence $(x_i)_{i\in \N}\subseteq\dom \ray{f}{y_f,n}$ with $g\sigma(x_i)\to y_g$, so $f(x_i)=\tau g\sigma(x_i) \to \tau(y_g)=y_f$ by continuity of $\tau$.
But by definition of rays $f(x_i)\notin \nbhd{y_f}{n+1}$ for all $i\in\N$, so $(f(x_i))_i$ cannot converge to $y_f$, a contradiction.

\smallskip

Third item: let $m,n\in\N$, use the continuity of $g$ to get $U\ni \sigma(x)$ open with $g(U)\subseteq N_{y_g\restr{m}}$.
Since $\sigma(x)$ is a center for $g$, $g\equiv g\restr{U}$ and we can choose a continuous reduction $(\sigma',\tau')$ witnessing $f\leq g\restr{U}$. 
By the previous item, there exists $M>m$ such that $\nbhd{y_g}{M+1}$ is disjoint from $\overline{g\sigma'(\dom \ray{f}{y_f,n})}$. This shows that $(\sigma',\tau')$ restricts to the desired reduction.
\smallskip

Last item: By a recursive application of the previous item.
\end{proof}

Given a finite set $F$ of functions, we define $\pgl F=\pgl \iw{\gl F}$, and\index[IdxSymb]{ptgl@{$\pgl F$}}
we write $f\leq \FinGl{F}$\index[IdxSymb]{2@{$f\leq \FinGl{F}$}} to say that $f\leq g$ for some $g\in \FinGl{F}$.

\begin{corollary}\label{ResidualCorestrictionOfCentered}
Suppose that $f:A\to B$ in $\functionsonbaire$, $G\subseteq \sC$ finite and $f\equiv \pgl G$. Then $f$ is centered and for every \ynote*{Residual of pointed gluing of finite set}{clopen} set $V\subseteq B$ which excludes its \cocenter{}, we have $f\corestr{V}\leq \FinGl{G}$.
\end{corollary}

\begin{proof}
Suppose that $f\equiv g=\pgl G$ and let $(\sigma,\tau)$ witness $f\leq g$. By \cref{Pgluingofregulariscentered}, $g(\iw{0})=\iw{0}$ is the \cocenter{} of $g$, so $f$ is centered by \cref{Centerinvariance} and $y=\tau(\iw{0})$ is the \cocenter{} of $f$. Therefore, $(\ray{f}{y,n})_n$ is reducible by finite pieces to $(\ray{g}{\iw{0},n})_n=\iw{\gl G}$ by \cref{Rigidityofthecocenter}.
So for all $n\in\N$ we get $\ray{f}{y,n}\leq\FinGl G$ and if $V\subseteq B$ is open with $V\notni y$ then there exists $M\in\N$ with $V\subseteq \bigcup_{n<M}\ray{B}{y,n}$. Hence, it follows that $f\corestr{V}\leq \bigsqcup_{n<M}\ray{f}{y,n}\leq  \FinGl{G}$, as desired.
\end{proof}

The following theorem ties centered functions in $\sC$ to the pointed gluing operation, as promised. 
\index[IdxNotion]{operation!pointed gluing}
\index[IdxNotion]{pointed gluing operation!relation to centered functions}
\begin{theorem}\label{CenteredasPgluing}
Let $f:A\rao B$ be in $\sC$.
\begin{enumerate}
\item If $f$ is centered with \cocenter{} $y$, then $f\equiv\pgl_{n\in \N}\ray{f}{y,n}$, \label{CenteredasPgluing_def} 
\item The function $f$ is centered if and only if $f\equiv\pgl_{i\in\N} f_i$ for some monotone (or regular) sequence $(f_i)_{i\in\N}$ of functions .\label{CenteredasPgluing_mono}
\end{enumerate}
In particular, if $f$ is centered with \cocenter{} $y$, then $f$ is simple with distinguished point $y$ and $\CB(f)=(\sup_n\CB(\ray{f}{y,n}))+1$.
\end{theorem}

\begin{proof}
Assume that $f\in \sC$ is centered with \cocenter{} $y$.
For \cref{CenteredasPgluing_def}, by \cref{Pgluingofraysasupperbound} we know that $f\leq\pgl_n\ray{f}{y,n}$, so it remains to prove $f\geq\pgl_n\ray{f}{y,n}$.
To do so we apply \cref{Pgluingaslowerbound2} with $x$ a center for $f$ and let $U$ be a neighborhood of $x$: by centeredness there is a continuous reduction $(\sigma,\tau)$ from $f$ to $f\restr{U}$.
By \cref{Rigidityofthecocenter}, for all $n\in\N$ we have $f(x)=y\notin\overline{f\sigma(f^{-1}(\ray{B}{y,n}))}$, so $(\sigma,\tau)$ restricts to witness the required reductions $\ray{f}{y,n}\leq f$.

\smallskip

The reverse implication of \cref{CenteredasPgluing_mono} follows from \cref{Pgluingofregulariscentered} and \cref{Centerinvariance}.
Next we show that $f\equiv \pgl_i f_i$ for some monotone sequence $(f_i)_i$, establishing the forward implication.

By \cref{Rigidityofthecocenter}, for all $m,n\in\N$ there exists $m'>m$ such that $\ray{f}{y,n}\leq \gl_{i=m}^{m'}\ray{f}{y,i}$.
We can recursively apply this fact to get a pairwise disjoint family of finite sets $(I_n)_n$ such that, setting $f_n=\gl_{i\in I_n}\ray{f}{y,i}$,
the sequence $(f_n)_n$ is monotone, and for all $n\in\N$ we have $\ray{f}{y,n}\leq f_n$.
Moreover $(f_n)_n$ is clearly reducible by pieces to $(\ray{f}{y,n})_n$, so by a double application of \cref{Pgluingasupperbound}, we have $\pgl_nf_n\equiv\pgl_n\ray{f}{y,n}$. 

\smallskip

Finally, by \cref{scatteredhavecocenter} $f$ is simple and its distinguished point coincides with its \cocenter{} $y$. Therefore $f$ has successor rank $\alpha+1$ and by \cref{CBrankofPgluingofregularsequence2simple} we have $\sup_n \CB(\ray{f}{y,n})=\alpha$ and so $\CB(f)=(\sup_n \CB(\ray{f}{y,n})) +1$.
\end{proof}

\section{Centered functions and structure of continuous reducibility}

Here we collect some important consequences of the structure of continuous reducibility on centered functions.

The first result shows that if continuous reducibility is \bqo{} up to a certain level of $\sC$, then every function at the next level is locally centered.

Recall that \bqo{} is stable under (the taking of) infinite sequences (see for instance \cite[Theorem 3.18]{wellbetterinbetween}):
if $(Q,\leq_Q)$ is \bqo{}, then the class $Q^\N$ of infinite sequences of elements of $Q$ is also \bqo{} when ordered as follows\footnote{In fact, $Q$ is $2$-\bqo{} is enough to ensure that $Q^\N$ is \wqo{}, see \cref{chapterFormalization}~\cref{sec:2bqo}.}. Given $s,t$ in $Q^\N$ say $s\leq_{Q^\N}t$\index[IdxSymb]{2@{$\leq_{Q^\N}$}} when there is a strictly increasing function $\iota:\N\to\N$ satisfying $s(x)\leq_Qt(\iota(x))$.
\index[IdxNotion]{better-quasi-order (\bqo{})}
\begin{theorem}\label{LocalCenterednessFromBQO}
For all $\alpha<\omega_1$, if $\sC_{<\alpha}$ is a \bqo{},
then every function in $\sC_\alpha$ is locally centered.
\end{theorem}
\ynote*{local centeredness from BQO}{
\begin{proof}
By strong induction on $\alpha$, we show that $\sC_{<\alpha}=\bigcup_{\beta<\alpha}\sC_{\beta}$ is \bqo{} implies that for every function $f:A\to B$ in $\sC_\alpha$ and every point $x\in A$ there exists a neighbourhood $U$ of $x$ such that $x$ is a center for $f\restr{U}$. By \cref{Centerinvariance}~\cref{Centerinvariance0}, this ensures that we can shrink $U$ to a basic neighbourhood $N_s\cap A$ to obtain local centeredness.  

\textbf{$\alpha=0$:} Trivial, since the empty function is the only element of $\sC_{0}$.

\textbf{$\alpha$ limit:} Suppose that $\CB(f)=\alpha$ is limit and let $x\in A$. There exists a neighbourhood $V$ such that $\CB(f\restr{V})<\alpha$ and so we can apply the induction hypothesis on $f\restr{V}$ at $x$.

\textbf{$\alpha$ successor:} By the Decomposition \cref{JSLdecompositionlemma}, $f$ is locally simple,
so without loss of generality we can suppose that $f$ is simple. We let $\bar{y}$ be the distinguished point of $f$
and write $\ray{B}{i}$ and $\ray{f}{i}$ for $\ray{B}{\bar{y},i}$ and $\ray{f}{\bar{y},i}$ respectively. Note that we have $\CB(\ray{f}{i})\leq \beta< \alpha$ for all $i\in \N$ by \cref{CBrankofPgluingofregularsequence2simple}.
Let $x\in A$, we show that it admits a neighbourhood $U$ such that $x$ is a center for $f\restr{U}$.

If there exists $s\segs x$ such that $\CB(f\restr{N_s})<\CB(f)$ then, by induction hypothesis, we are done.
Otherwise, for all $s\segs x$ we have $\CB(f\restr{N_s})=\CB(f)$.
Writing $\alpha=\beta+1$, \cref{CBbasics0}~\cref{CBbasicsfromJSL2} ensures that $N_s\cap \CB_{\beta}(f)\neq \emptyset$ for all $s\segs x$,
so $x$ belongs to the closed set $\CB_{\beta}(f)$. Hence, we have $f(x)=\bar{y}$, each function $f\restr{N_s}$ is simple and for each $n$, the sequence of rays $(\ray{f}{i}\restr{N_{x\restr{n}}})_{i\in\N}$ takes value in $\sC_{<\alpha}$ by \cref{CBbasics0}~\cref{CBbasicsfromJSL2}.
As $\sC_{<\alpha}$ is \wqo{} under continuous reducibility, we can choose by induction a non-decreasing sequence $(j_n)_n$ in $\N$
such that the sequence of functions $\rho_n=(\ray{f}{i}\restr{N_{x\restr{n}}})_{i\geq j_n}$ is regular for all $n$ (\cite[Proposition 2.9]{yann2017towardsbetter}).
Note that $m<n$ implies $\rho_m\geq_{(\sC_{<\alpha})^\N} \rho_n$, since $N_{x\restr{m}}\supseteq N_{x\restr{n}}$ implies $\ray{f}{i}\restr{N_{x\restr{m}}}\geq \ray{f}{i}\restr{N_{x\restr{n}}}$ for all $i\geq j_n$,
so $(\rho_n)_n$ is decreasing in $(\sC_{<\alpha})^\N$.
Since $\sC_{<\alpha}$ is \bqo{}, $(\sC_{<\alpha})^\N$ is \wqo{} and so there exists $m$ such that for all $n>m$ we have $\rho_m\equiv_{(\sC_{<\alpha})^\N} \rho_n$.

Define $U=N_{x\restr{m}}\setminus f^{-1}(\bigcup_{i<{j_m}}\ray{B}{i})$ and note that $f\restr{U}\leq \pgl \rho_m$ follows from \cref{Pgluingofraysasupperbound}. Next, we show that $\pgl \rho_m \leq f\restr{U}$ witnessed by a reduction $\sigma(\iw{0})=x$ using \cref{Pgluingaslowerbound2}. To do so, it is enough to prove that for every $i\geq j_m$ and every $n>m$,
there exists $(\sigma_i,\tau_i)$ that continuously reduces $\ray{f}{i}\restr{N_{x\restr{m}}}$ to $f\restr{U}$ such that
$\im \sigma \subseteq N_{x\restr{n}}$ and $\bar{y}\notin \overline{\im f\restr{U} \sigma_i}$.
This is possible since $\rho_m\leq \rho_n$ and so for every $i\geq j_m$ there exists $i'\geq j_n\geq j_m$ with $\ray{f}{i}\restr{N_{x\restr{m}}}\leq \ray{f}{i'}\restr{N_{x\restr{n}}}$, as desired. Hence $f\restr{U}\equiv \pgl \rho_m$ and since $\rho_m$ is regular, $\iw{0}$ is a center for $\pgl \rho_m$ by \cref{Pgluingofregulariscentered}. Because $\pgl \rho_m \leq f\restr{U}$ is witnessed by a reduction with $\sigma(\iw{0})=x$, it follows that $x$ is a center of $f\restr{U}$ by \cref{Centerinvariance}, as desired.
\end{proof}
}

Pointed gluing behaves rather specifically in a finitely generated class, as pointed out by the following proposition that we use repeatedly in the sequel.

\begin{proposition}\label{FinitegenerationandPgluing}
Let $F\subseteq \functionsonbaire$ be finite and $(f_i)_{i\in\N}$ be a sequence in $\functionsonbaire$.
\begin{enumerate}
\item If $f_i\leq\FinGl{F}$ for all $i\in\N$, then $\pgl_if_i\leq \pgl F$.
\item If for all $f\in F$ and all $i\in\N$ there is $j\geq i$ such that $f\leq f_j$, then $\pgl F\leq\pgl_if_i$.
\end{enumerate}
\end{proposition}
\begin{proof}
For both points, we build a reduction by pieces and then use \cref{Pgluingasupperbound}.

For the first item: by hypothesis for all $n\in \N$ there is an integer $k_n$ such that $f_n\leq k_n F$. Set then $K_n=\sum_{i<n}k_i$, and $I_n=\left[K_{n}, K_{n+1}\right)$ for all $n\in\N$.
This is a family of pairwise disjoint finite subsets of $\N$ witnessing a reduction by pieces from $(f_i)_i$ to $\iw{\gl F}$.

For the second item, we build by induction a reduction by pieces.
Given $n\in\N$, suppose that we have built a family $(I_m)_{m<n}$ of pairwise disjoint finite subsets of $\N$ such that for all $m<n$ we have $\gl F\leq\gl_{i\in I_m}f_i$.
Set $j=\max(\bigcup_{m<n}I_m)+1$ and use the hypothesis to fix an injective function $\iota:F\rao[j,\infty)$ such that for all $g\in F$ we have $g\leq f_{\iota(g)}$.
Setting $I_n=\iota(F)$, which is finite since $F$ is, yields the desired reduction by pieces. 
\end{proof}

\index[IdxNotion]{finitely generated set of functions}

The next theorem shows how finitely generated levels guarantee a neat representation of centered functions. Let us write $\sC_{[\lambda,\lambda+n]}= \bigcup_{i\leq n} \sC_{\lambda +i}$ for $\lambda$ limit or null and $n\in\N$.\index[IdxSymb]{Scat@{$\sC_{[\lambda,\lambda+n]}$}}


\begin{theorem}\label{finitenessofcenteredfunctions}
Let $\lambda$ be zero or a limit ordinal and $n\in\N$.
Assume that $\sC_{[\lambda,\lambda+n]}$ is generated by some finite set $F$.
Then for every centered function $g\in\sC_{[\lambda,\lambda+n+1]}$ either $g\equiv \Minimalfct{\lambda+1}$ or there exists a non empty $G\subseteq F$ such that $g\equiv\pgl G$. 

In particular, there are finitely many centered functions up to equivalence in $\sC_{\lambda+n+1}$.
\end{theorem}
\begin{proof}
Let $g \in \sC_{[\lambda,\lambda+n+1]}$ be centered, hence of successor $\CB$-rank by \cref{CenteredasPgluing}.
In particular, $g\not\equiv\Maximalfct{\lambda}$, so $\lambda<\CB(g)\leq \lambda +n+1$.

By \cref{CenteredasPgluing}, there is a $\leq$-monotone sequence $(g_i)_i$ such that $g\equiv\pgl_ig_i$ and
for every $i$ we have $\CB(g_i)<\CB(g)\leq  \lambda +n+1$ and $(\sup_i \CB(g_i))+1=\CB(g)>\lambda$. In particular, we have $\sup_i \CB(g_i) \geq \lambda$.
Assume first that $\sup_i \CB(g_i) =\lambda$. If $\lambda=0$, then $g\equiv \Minimalfct{1}= \pgl \emptyset$. Otherwise $\lambda$ is limit and $\Minimalfct{\lambda+1} \equiv \pgl_i g_i \equiv g$ by \cref{ConsequencesGeneralStructureThm}. 

Assume now that $\sup_i \CB(g_i) > \lambda$. By monotonicity there exists $j\in \N$ such that $\CB(g_i)\geq \lambda$ for all $i\geq j$, and using monotonicity again
we get $g\equiv\pgl_{i\geq j} g_i$ by two applications of \cref{Pgluingasupperbound}.
Now for every $i\geq j$, since $g_i \in \sC_{[\lambda,\lambda+n]}$ we can choose a way of writing $g_i\equiv\gl_{f\in F}n_{i,f} f$  and set $G_i=\{f\in F \mid n_{i,f}>0\}$.
We define $G=\bigcup_{i\geq j}G_i$ and show that $g\equiv\pgl G$. Note that $G$ cannot be empty, since this would imply $g_i= \emptyset$ for all $i\geq j$ and so $\sup_i \CB(g_i)=0$, a contradiction.  

By definition of $G$, $g_i\in \FinGl{G}$ for all $i\geq j$. Moreover, for all $f\in G$ we have $f\in G_i$ for some $i\geq j$, hence $f\leq g_i$ and so $f\leq g_{k}$ for all $k\geq i$ by monotonicity of $(g_i)_{i\geq j}$. Using \cref{FinitegenerationandPgluing}, we obtain $g\equiv \pgl G$.
\end{proof}

Here is the first common point between $\sC_2$ and $\sC_{\lambda+1}$ for all $\lambda$ limit that we announced.

\begin{corollary}\label{cor:CenteredSucessor}
Let $\lambda<\omega_1$ be either equal to 1 or infinite limit.
Then, up to continuous equivalence, there are two centered functions in $\sC_{\lambda+1}$: $\Minimalfct{\lambda+1}$ and $\pgl\Maximalfct{\lambda}$. Moreover, $\Minimalfct{\lambda+1}<\pgl\Maximalfct{\lambda}$ holds.
\end{corollary}

\begin{proof}
We can use \cref{finitenessofcenteredfunctions} since the hypothesis is valid thanks to \cref{LocallyConstantFunctions} for $\lambda=1$, and to \cref{JSLgeneralstructure} for $\lambda$ limit.

In case $\lambda=1$, any centered function in $\sC_2$ is equivalent to $\pgl G$ where $G$ is either $\set{\Minimalfct{1}}$, $\set{\Maximalfct{1}}$ or $\set{\Minimalfct{1},\Maximalfct{1}}$.
Therefore we get $\Minimalfct{2}$ and $\pgl\set{\Minimalfct{1},\Maximalfct{1}}\equiv\pgl\Maximalfct{1}$ by \cref{FinitegenerationandPgluing}, as wanted.

In case $\lambda$ is limit the result is straightforward since the only possible $G$ is $\set{\Maximalfct{\lambda}}$. 

For the last part, note that $\Minimalfct{\lambda+1}\leq \pgl\Maximalfct{\lambda}$ and suppose towards a contradiction that equivalence holds. If $\lambda=1$, then using \cref{Rigidityofthecocenter} we get  $\Maximalfct{1}=\id_\N \leq n \id_1=n \Minimalfct{1}$ for some $n\in\N$, a contradiction. If $\lambda$ is limit, let $\Minimalfct{\lambda+1}=\pgl_n\Minimalfct{\alpha_n+1}$ for $(\alpha_n)_n$ cofinal in $\lambda$. By \cref{Rigidityofthecocenter} again, it follows that $\Maximalfct{\lambda}\leq \gl_{n<M} \Minimalfct{\alpha_n+1}$ for some $M\in \N$, but then $\CB(\Maximalfct{\lambda})=\lambda\leq \sup_{n<M}(\alpha_n +1)<\lambda$, a contradiction again. 
\end{proof}


\section{Simple functions at successors of limit levels}

\index[IdxNotion]{simple function}\index[IdxNotion]{function!simple}

\begin{proposition}\label{Simpleiffcoincidenceofcocenters}
Let $f$ be in $\sC$, assume that $f=\bigsqcup_{i\in\N} f_i$ for some sequence $(f_i)_{i\in\N}$ of centered functions and set $I=\set{n\in\N}[\CB(f_n)=\sup_i\CB(f_i)]$.
\begin{enumerate}
\item $\CB(f)$ is successor if and only if $I\neq\emptyset$,
\item The $\CB$-degree of $f$ is the cardinality of the set of \cocenter{}s of the functions $f_i$ for $i$ in $I$.
\end{enumerate}
In particular, $f$ is simple if and only if $I\neq\emptyset$ and for all $n\in I$ the \cocenter{}s of the functions $f_n$ coincide with the distinguished point of $f$.
\end{proposition}

\begin{proof}
Observe first that by \cref{CBrankofclopenunion} we have $\CB(f)=\sup_i\CB(f_i)$.

To see the first item, suppose that $\CB(f)=\alpha+1$ for some $\alpha<\omega_1$. By \cref{CBbasics0}~\cref{CBbasicsfromJSL2} we have $\CB_\alpha(f)=\bigsqcup_n\CB_\alpha(f_n)$.
Since $\CB_\alpha(f)$ is non-empty, so must be $\CB_\alpha(f_n)$ for some $n$, which means that $\CB(f_n)=\alpha+1$ and $n\in I\neq\emptyset$.
If now $I\neq\emptyset$, then for any $n\in I$ we have $\CB(f_n)=\sup_i\CB(f_i)=\CB(f)$ and by \cref{CenteredasPgluing} $\CB(f_n)$ is successor, hence $\CB(f)$ is too.

\smallskip

For the second point, if $\CB(f)$ is limit then our convention is that the $\CB$-degree of $f$ is $0$, which is also the cardinality of the set of \cocenter{}s of functions in $I$, since $I=\emptyset$ by the first point. Suppose now that $\CB(f)=\alpha+1$ for some $\alpha<\omega_1$ and so $I\neq\emptyset$.
For all $n\in I$ the function $f_n$, being centered, is simple of distinguished point its \cocenter{} $y_n$ by \cref{CenteredasPgluing}.
Since by \cref{CBbasics0}~\cref{CBbasicsfromJSL2} we have $\CB_\alpha(f)=\bigsqcup_i\CB_\alpha(f_i)$ and $\CB_{\alpha}(f_i)=\emptyset$ if $i\notin I$, we actually have
$\CB_\alpha(f)=\bigsqcup_{n\in I}\CB_\alpha(f_n)$, and $f(\CB_\alpha(f))=f(\bigsqcup_{n\in I}\CB_\alpha(f_n))=\set{y_n}[n\in I]$, which proves our point.
\end{proof}

\index[IdxNotion]{Precise Structure!simple functions at successor of limit}
\begin{theorem}\label{simplefunctionslambda+1damuddafuckaz}
Let $\lambda$ be limit or 1. Assume that continuous reducibility is \bqo{} on $\sC_{<\lambda}$.
Any simple function $f\in\sC_{\lambda+1}$ is continuously equivalent to one of $\Minimalfct{\lambda+1},\Minimalfct{\lambda+1}\gl\Maximalfct{\lambda}$ or $\pgl\Maximalfct{\lambda}$.
\end{theorem}

\begin{proof}
Since we assume that $\leq$ is \bqo{} on $\sC_{<\lambda}$, by the General Structure \cref{JSLgeneralstructure} it is also \bqo{} on $\sC_{\leq\lambda}$. So by
\cref{LocalCenterednessFromBQO} we can write $f=\bigsqcup_{i\in I} f_i:A_i\to B$ with $(f_i)_{i\in I}$ a sequence of centered functions. As $f$ is simple, $f\leq \pgl\Maximalfct{\lambda}$ by \cref{Maxfunctions}, so  if $f_i\equiv\pgl\Maximalfct{\lambda}$ for some $i\in I$, then $f\equiv\pgl\Maximalfct{\lambda}$.

Since $\pgl\Maximalfct{\lambda}$ and $\Minimalfct{\lambda+1}$ are the only two centered functions in $\sC_{\lambda+1}$ by \cref{cor:CenteredSucessor},
we can suppose that for all $n\in I$ if $\CB(f_n)>\lambda$ then $f_n\equiv\Minimalfct{\lambda+1}$.
By \cref{Simpleiffcoincidenceofcocenters} we have $f_i\equiv \Minimalfct{\lambda+1}$ for at least one $i\in I$, and
the distinguished point $\bar{y}$ of $f$ is the \cocenter{} of $f_i$ for any $f_i\equiv\Minimalfct{\lambda+1}$.
In the rest of this proof, all rays are taken with respect to $\bar{y}$. 
By simplicity $\CB_\lambda(f)\subseteq f^{-1}(\set{\bar{y}})$, so $\CB(\ray{f}{n})\leq\lambda$ for all $n\in\N$.
Also we note that $\CB(\ray{f_i}{j})<\lambda$ for all $i,j\in \N$, {which follows from \cref{Rigidityofthecocenter} when $f_i\equiv\Minimalfct{\lambda+1}$.}

If $\CB(\ray{f}{n})<\lambda$ for all $n\in \N$, then $f \leq \pgl_n \ray{f}{n} \leq \Minimalfct{\lambda+1}$ by \cref{Pgluingofraysasupperbound,ConsequencesGeneralStructureThm} and so $f \equiv \Minimalfct{\lambda+1}$.

\smallskip

So we can suppose that $\CB(\ray{f}{n})=\lambda$ for some $n\in\N$. Fix $W=\ray{B}{n}$, we have $\bar{y}\notin W$, $\Maximalfct{\lambda}\leq f\corestr{W}$
by \cref{JSLgeneralstructure}, and $\Minimalfct{\lambda+1} \leq f\corestr{B\setminus W}$ by centeredness of $\Minimalfct{\lambda+1}$.
Since $W$ is clopen we have $f\corestr{B\setminus W}\gl f\corestr{W}\equiv f$ by \cref{UsefulcriterionforequivFinGl}, so $\Minimalfct{\lambda+1}\gl\Maximalfct{\lambda}\leq f$.

We show that $f\leq\Minimalfct{\lambda+1}\gl\Maximalfct{\lambda}$ also holds in that case.
To do so we are going to find a clopen partition $A=C_0\sqcup C_1$ satisfying
 $f\restr{C_0}\leq \Maximalfct{\lambda}$ and $f\restr{C_1}\leq \Minimalfct{\lambda+1}$, which by \cref{Gluingasupperbound}
 gives \(f\leq\Minimalfct{\lambda+1}\gl\Maximalfct{\lambda}\).

We split $f$ ``along the diagonal''. For all $j\in \N$, let $g_j=\bigsqcup_{i\leq j} \ray{f_i}{j}$ and $h_{j}=\bigsqcup_{i> j} \ray{f_i}{j}$.
Consider the clopen sets $C^{j}_i=\dom (\ray{f_i}{j})=A_i\cap f^{-1}(\ray{B}{j})$ for all $i,j$. Let $C_{0}=\bigcup\{C^{j}_i\mid i > j\}$ and its complement $C_1=A\setminus C_0$. Since $\CB(\ray{f}{j}_i)<\lambda$ for all $i,j$ we have $\CB(h_j)\leq\lambda$,
so $f\restr{C_0}\leq\Maximalfct{\lambda}$ by \cref{Maxfunctions}.
For all $j$, we also have $g_j=\bigsqcup_{i\leq j}\ray{f_i}{j}$ and so $\CB(g_j)<\lambda$. Since the functions $g_j$ correspond to the rays of $f\restr{C_1}$ at $\bar{y}$, a combination of \cref{Pgluingofraysasupperbound,ConsequencesGeneralStructureThm} then yields
$f\restr{C_1}\leq\pgl_j g_j \leq\Minimalfct{\lambda+1}$.

Note that $C_0$ is open as a union of clopen sets. To see that $C_1$ is open too, note that $C_1=\bigcup_i (A_i\setminus C_0)$ and that $A_i\setminus C_0=\bigcup_{j\leq i} C^{j}_i$ is clopen for all $i$, which concludes the proof.
\end{proof}

Moreover, we show in \cref{OptimalityatSuccessorofLimit} that in fact $\Minimalfct{\lambda+1}<\Minimalfct{\lambda+1}\gl\Maximalfct{\lambda}< \pgl\Maximalfct{\lambda}$. Finally, note that
the Decomposition \cref{JSLdecompositionlemma} yields the following corollary.

\index[IdxNotion]{finitely generated set of functions}
\begin{corollary}\label{finitedegreedamuddafuckaz}
For $\lambda$ limit or 1, if continuous reducibility is \bqo{} on $\sC_{<\lambda}$, then
the set of functions in $\sC_{\lambda+1}$ that have finite degree is finitely generated by the set $\set{\Maximalfct{\lambda},\Minimalfct{\lambda+1},\pgl\Maximalfct{\lambda}}$.
\end{corollary}


\chapter{Precise structure} \label{PreciseStructureFinal}

We have now reached the point where we need to define a finite set of generators, for that we require a third and last operation.

\section{The wedge operation}\label{subsectionWedge}

Finite gluing and pointed gluing allow us to generate all functions with compact domains up to continuous equivalence by \cref{Compactdomains}.
The infinite gluing allows us to build functions which map non converging sequences to non converging sequences.
With the pointed gluing, we can build functions which send converging sequences to converging sequences.

The \emph{wedge operation}, that we now introduce, allows us to construct functions exhibiting two new ``continuous'' behaviours on sequences that appear to be essential for functions with non compact domains.
First, several sequences converging to distinct limits can be mapped to converging sequences that share the same limit.
Second, a non converging sequence can be mapped to a converging sequence.
While gluing and pointed gluing naturally arise from operations on spaces, this new operation is fundamentally an operation on functions, with no obvious counterpart for spaces.  

\index[IdxNotion]{operation!wedge}\index[IdxNotion]{wedge operation}
\begin{definition}
Given a natural number $k\in \N$ and functions $(f_i)_{i\leq k+1}$ in $\functionsonbaire$,
we define a function $f:=\fctwedge (f_0,\ldots,f_{k}\mid f_{k+1})$ called the \emph{wedge} of $(f_0,\ldots,f_{k}\mid f_{k+1})$ as follows.
\index[IdxSymb]{v@{$\fctwedge(f_0,\ldots f_k\mid f_{k+1})$}}

Let $A_i= \pgl \dom f_i$, for $i\leq k$, $A_{k+1+i}=\dom f_{k+1}$ for all $i\in\N$, and $B_j= \gl_{i\leq k+1}\im f_i$ for all $j\in \N$.
The domain of $\fctwedge (f_0,\ldots,f_{k}\mid f_{k+1})$ is $\gl_{i} A_i$ and its range is $\pgl_j B_j$.
Then we set for all $i,j$ in $\N$:
\begin{align*}
f ((i)\conc \iw{0})&=\iw{0}\mbox{ if }i\leq k \\
 f((i)\conc(0)^j\conc(1)\conc x)&=(0)^j\conc(1)\conc(i)\conc f_i(x)\mbox{ if }i\leq k \\
  f((k+1+i)\conc x)&=(0)^i\conc(1)\conc(k+1)\conc f_{k+1}(x)
\end{align*}
\end{definition}

\begin{figure}
\centering
\index[IdxFig]{The Wedge operation}
\begin{tikzpicture} [thick, scale=0.35]
\def \HOne{2.5}
\def \VOne{1.5}
\def \HTwo{5} 
\def \LHOne{8}
\def \LVOne{7}
\def \LVTwo{1}
\def \LHTwo{1}
\def\Angle{0}

\def\Second{20}

\foreach \r in {0,1}
\draw (\Angle:\HOne*\r)-- ++(90+\Angle :\LVOne);

\draw (0,-1) -- (\Angle:\HOne*0);

\foreach \r  in {1,...,4}
\draw (0,-1) ..controls +({0.5\r},0) and +(0,-0.8).. (\Angle:\HOne*\r);

\node at ($({\HOne*5},{\VOne*0})+({\LHTwo},1)$) {$\cdots$};

\foreach \r  in {0,1}
\foreach \s in {0,1,2}
{
\pgfmathparse{20*\s}
\edef\c{\pgfmathresult}
\draw[fill=black!\c!white] ($({\HOne*\r},{\VOne*\s})+({\LHTwo},1)$) ellipse (0.8 and 0.5);   
\draw ($({\HOne*\r},{\VOne*\s})+(0,0.5)$) -- ($({\HOne*\r},{\VOne*\s})+({\LHTwo},0.5)$);

\node at ($({\HOne*\r},{\LVOne})+({\LHTwo},0)$)  {$\vdots$};
}

\foreach \r  in {2,3,4}
{
\pgfmathparse{\r-2}
\edef \s{\pgfmathresult}
\pgfmathparse{20*\s}
\edef\c{\pgfmathresult}
\edef \s{0}

\draw[fill=black!\c!white] ($({\HOne*\r},{\VOne*\s})+({\LHTwo},1)$) ellipse (0.8 and 0.5);   
\draw  (\Angle:\HOne*\r)-- ($({\HOne*\r},{\VOne*\s})+(0,0.5)$) -- ($({\HOne*\r},{\VOne*\s})+({\LHTwo},0.5)$);

}

\def \anglesecond{25}
\draw (\Second,-1)-- ++ (\anglesecond :16); 

\foreach \r  in {0,1,...,2}
\foreach \s in {0,1,...,2}
{
\pgfmathparse{\HTwo*\r}
\edef\x{\pgfmathresult}
\pgfmathparse{180-35*\s}
\edef\y{\pgfmathresult}
\pgfmathparse{20*\r}
\edef\c{\pgfmathresult}
\draw ($(\Second,-1)+(\anglesecond :\x)$)-- ++(\y:2.5);
\draw[fill=black!\c!white] ($(\Second,-1)+(\anglesecond :\x)+(\y:2.5)+(0,0.7)$) ellipse (0.5 and 0.7);
}
\node at ($(\Second,-1)+(\anglesecond :{3*\HTwo})+(110:1)+(0,0.7)$) {$\cdots$};

\draw[->] ($({\HOne*0},{\VOne*2})+({\LHTwo},1)$) ..controls +(4,4) and +(-4,4).. ($(\Second,-1)+(\anglesecond :{2*\HTwo})+(180:2.5)+(0,0.7)$) node[midway,above] {$f_0$};
\draw[->] ($({\HOne*1},{\VOne*1})+({\LHTwo},1)$) ..controls +(4,3) and +(-4,2).. ($(\Second,-1)+(\anglesecond :{\HTwo})+(145:2.5)+(0,0.7)$) node[midway,above] {$f_1$};
\draw[->] ($({\HOne*2},{\VOne*0})+({\LHTwo},1)$) ..controls +(3,2) and +(-4,2).. ($(\Second,-1)+(\anglesecond :0)+(110:2.5)+(0,0.7)$) node[midway,above] {$g$};

\end{tikzpicture}
\caption{The wedge operation $\fctwedge(f_0,f_1\mid g)$.} 
\label{fig:Wedge}
\end{figure}

\begin{remark}\label{rem:infinitewedge}
This finitary operation is closely related to the following infinitary operation. 
Given an infinite matrix of functions $f_{i,j}:A_{i,j}\to B_{i,j}$, ${(i,j)\in\N\times \N}$, we can define $\fctwedge(f_{i,j}):\gl_j \pgl_i A_{i,j} \to \pgl_i \gl_j B_{i,j}$ by
\begin{align*}
(j)\conc \iw{0}&\longmapsto\iw{0}\\
 (j)\conc(0)^i\conc(1)\conc x&\longmapsto(0)^i\conc(1)\conc(j)\conc f_{i,j}(x)
\end{align*}
Alternatively, we can first form the function $w=\gl_j \pgl_{i} f_{i,j}$ whose codomain is given by $\gl_j \pgl_{i} B_{i,j}$.
Naturally, each $\pgl_i B_{i,j}$ is a pointed space with basepoint $p_j=\iw{0}$ and their \emph{wedge sum}\footnote{See for instance \cite[p.153]{MR0957919}.} is defined through
the topological quotient $\pi:\gl_j  \pgl_i B_{i,j} \to \bigvee_j \pgl_i B_{i,j}$ with respect to the equivalence relation $\{((j)\conc p_j,(j')\conc p_{j'})\mid j,j'\in \N\}$.
Then, we can define $\fctwedge(f_{i,j})$ by $\pi w$, which explains our choice of name for this operation.

Observe that our finitary operation $\fctwedge (f_0,\ldots,f_{k}\mid f_{k+1})$ can be viewed as the following particular case of the infinitary wedge. Each of first $k+1$ functions $f_0,\ldots, f_k$ are repeated infinitely in each of the first $k+1$ columns, $(f_{i,j})_j=\iw{f_i}$ for $i=0,\ldots, k$, hence we informally refer to them as ``the verticals''.
The function $f_{k+2}$ is repeated infinitely on the diagonal of the remainder of the matrix, say $f_{i,i}=f_{k+2}$ for all $i>k$, so we informally call it ``the diagonal''. All other entries are filled with the empty function.
\end{remark}

Note that for any function $f$ we have $\pgl f \equiv \fctwedge (f| \emptyset)$.

\index[IdxNotion]{Cantor--Bendixson!rank!of a wedge}\index[IdxNotion]{wedge operation!Cantor--Bendixson rank}
\begin{fact}\label{BasicfactsWedge}
Let $k\in\N$, $(f_i)_{i\leq k+1}$ be a sequence in $\functionsonbaire$, and set $f:=\fctwedge(f_0\ldots f_k\mid f_{k+1})$.
\begin{enumerate}
\item The wedge operation preserves countability and Polishness of domain and range, surjectivity, finite-to-oneness, and continuity of functions.
\item If $k>0$, then $f$ is not injective.
\item If $f_i\in\sC$ for all $i\leq k+1$ then so does $f$, and 
\[\CB(f)=\max\big(\set{\CB(f_i)+1}[i\leq k]\cup\set{\CB(f_{k+1})}\big).\]
\end{enumerate}
\end{fact}

As with our other operations on functions, we now to look for upper and lower bound criteria.
We can first adapt the ``upper bound'' sufficient criterion \cref{Pgluingasupperbound} to get an upper bound criterion for the wedge operation.

\index[IdxNotion]{upper bound condition!wedge}
\index[IdxNotion]{wedge operation!upper bound condition}

\begin{proposition}[Wedge as upper bound]\label{Wedgeasupperbound}
Let $f:A\to B$ in $\functionsonbaire$ be continuous and $(f_i)_{i\leq k+1}\subseteq \functionsonbaire$ for $k\in \N$.
Suppose that there exist $y\in B$ and a clopen partition $(A_i)_{i\in\N}$ of $A$ with the following properties:
\begin{enumerate}
\item For all $i\leq k$, $(\ray{(f\restr{A_i})}{y,j})_{j\in \N}$ is reducible by pieces to $\iw{f_i}$,
\item $(f\restr{A_i})_{i>k}$ is reducible by pieces to $\iw{f_{k+1}}$,
\item $f(A_i)\rao y$.
\end{enumerate}
Then $f\leq\fctwedge (f_0,\ldots,f_{k}\mid f_{k+1})$.
\end{proposition}
\begin{proof}
Use the hypotheses to fix a family $(I_n)_n$ of pairwise disjoint finite subsets of $\N$ and reductions $(\sigma^D_n,\tau^D_n)$ from $f\restr{A_n}$ to $\gl_{i\in I_n} f_{k+1}$ for all $n>k$,
and to get continuous reductions $(\sigma^V_i,\tau^V_i)$ from $f\restr{A_i}$ to $\pgl f_i$ with $\tau^V_i(\iw{0})=y$ using \cref{Pgluingasupperbound}, for all $i\leq k$. 

Define $\sigma$ as follows: 
\begin{align*}
\sigma(x)=&(i)\conc\sigma^V_i(x) && \text{if $x\in A_i$ for some $i\leq k$,}\\
\sigma(x)=&\sigma^D_i(x) && \text{if $x\in A_i$ for some $i>k$.}
\end{align*}

Define $\tau$ as follows:
\begin{align*}
\tau(\iw{0})=&y\\
\tau((0)^j\conc(1)\conc(i)\conc x)=&\tau^V_i((0)^j\conc(1)\conc x)&&\text{if $i\leq k$ and $j\in \N$,}\\
\tau((0)^{i}\conc(1)\conc(k+1)\conc x)=&\tau^D_{n}((i)\conc x) && \text{if $i\in I_{n}$.}
\end{align*} 
By construction $(\sigma, \tau)$ is a reduction from $f$ to $\fctwedge (f_0,\ldots,f_{k}\mid f_{k+1})$. The function $\sigma$ is continuous as a disjoint union of continuous maps by \cref{lem:ContUnion}.
To see that $\tau$ is continuous, note that $\tau$ is continuous on $U=\dom \tau \setminus \{\iw{0}\}$ again by \cref{lem:ContUnion}.
To use \cref{prop:sufficientcondforcont}, we claim that for every sequence $(x_n)_n$ in $U$ that converges to $\iw{0}$ we have $\tau(x_n)\to y$.
If $(x_n)$ is included in $U_i=\bigcup_j N_{(0)^j\conc(1)\conc(i)}$ for some $i\leq k$, then continuity of $\tau_i^{V}$ ensures that $\tau(x_n)\to \tau^V_i(\iw{0})=y$.
If now $(x_n)_n$ is included in $U_{k+1}=\bigcup_{i} N_{(0)^i\conc(1)\conc(k+1)}$ then $\tau(x_n)\to y$ because we assumed that $f(A_i)\to y$.
Since the sets $U_0,\ldots U_{k},U_{k+1}$ form a finite partition of $U$, the claim follows.
\end{proof}

As a corollary, we obtain that, up to continuous equivalence, any function of the form $\fctwedge (f_0,\ldots,f_{k}\mid g)$ is equivalent to one where the verticals $\set{f_i}[i\leq k]$ form an antichain for continuous reducibility\footnote{We go further in \cref{WedgeReducedForm} and prove the existence of a reduced form for the wedge of continuous functions.}. To state this corollary, we use the \emph{domination} order on sets of functions, which will also be used in the sequel. For $F,G$ set of functions we say that $F$ is \emph{dominated} by $G$  if for all $f\in F$ there is $g\in G$ with $f\leq g$. Note that if $G$ is finite, then $G$ is equivalent for domination to any maximal antichain of maximal elements of $G$.
When no confusion is possible, we simply denote by $F\leq G$\index[IdxSymb]{2@{$F\leq G$}} the domination relation and write $f\leq G$\index[IdxSymb]{2@{$f\leq G$}} instead of $\set{f}\leq G$.

\begin{corollary}\label{cor:wedgeSets}
Assume that $f_0,\ldots,f_k, h_0,\ldots, h_l, g \in \functionsonbaire$ are continuous. If $\set{f_i}[i\leq k]$ is equivalent to $\set{h_j}[j\leq l]$ for domination, then
\[\fctwedge(f_0,\ldots,f_k\mid g)\equiv\fctwedge(h_{0},\ldots,h_{l}\mid g).\]
\end{corollary}

\begin{proof}
To see that the left hand side wedge, denoted by $f$, reduces to right hand side wedge,
for each $i\leq k$ let $\phi(i)=\min \set{j\leq l}[f_i\leq h_j]$ and set $A_j=\bigcup \set{N_{(i)}}[\phi(i)=j]$ for $j\leq l$.
Notice that the sets $A_j$, $j\leq l$ partition $\dom f$ and for $y=\iw{0}$ we have $\ray{(f\restr{A_j})}{y,n}\equiv \gl_{i\in\phi^{-1}(j)} f_i \leq |\phi^{-1}(j)| h_j$ for all $j\leq l$. 
\end{proof}

\ynote*{wedge of a finite set}{In the sequel, we sometimes write $\fctwedge(F\mid g)$ for a finite set $F$ of functions, since by the above result any enumeration of $F$ yields a continuously equivalent function.}\index[IdxSymb]{V@{$\fctwedge(F\mid g)$}}

Finding \enquote{natural} sufficient conditions for when a wedge is a lower bound for a function $f$ has proven to be quite tricky,
in particular to make sure that the map $\tau$ is well-defined. So far, we have not been able to find a condition resembling that of \cref{Pgluingaslowerbound}.
We can however generalize \cref{Pgluingaslowerbound2} which enjoys much stronger hypotheses.
In these specific cases we can \enquote{disjointify} enough pieces of the reduction to make sure that their union is well-defined.

\index[IdxNotion]{lower bound condition!wedge}
\index[IdxNotion]{wedge operation!lower bound condition}

\index[IdxNotion]{Disjointification Lemma}
\begin{lemma}[Disjointification Lemma]\label{DisjointificationLemma}
Let $f:A\to B$ in $\functionsonbaire$ be continuous and $(f_i)_{i\leq k+1}\subseteq \functionsonbaire$ for $k\in \N$.

Suppose that there exist $y\in\im f$ and $(x_i)_{i\leq k}$ in $f^{-1}(\{y\})$ such that 
\begin{enumerate}
\item for every $i\leq k$, for every open $U\ni x_i$ there exists $(\sigma,\tau)$ reducing $f_i$ to $f$
with $\im(\sigma)\subseteq U$ and $y\notin \overline{\im(f\sigma)}$, and \label{disjoint_condition1}
\item for every open $V\ni y$, there exists $(\sigma,\tau)$ reducing $f_{k+1}$ to $f$ with $\im(f\sigma)\subseteq V$ and $y\notin\overline{\im(f\sigma)}$.\label{disjoint_condition2}
\end{enumerate}
Then $\fctwedge (f_0,\ldots,f_{k}\mid f_{k+1})\leq f$.
\end{lemma}

\begin{proof} We simultaneously define by induction on $n$, a sequence $(V_n)_n\subseteq B$ of open neighborhoods of $y$ and sequences $(U_{i,n})_n\subseteq A$ of open neighborhoods of $x_i$ for each $i\leq k$ satisfying the following conditions:
\begin{enumerate}[label=\arabic*)]
\item $V_n\rao y$ in $B$,
\item $U_{i,n}\rao x_i$ in $A$, and
\item  for all $n\in\N$, for all $i\leq k+1$, there is a continuous reduction $(\sigma_{i,n},\tau_{i,n})$ from $f_i$ to $f\corestr{V_n\setminus V_{n+1}}$,
satisfying $\im(\sigma_{i,n})\subseteq U_{i,n}$ in case $i\leq k$.
\end{enumerate}

Set $V_0=B$ and $U_{i,0}=A$ for all $i\leq k$.
Fix $n\in\N$ and suppose that all relevant sets are built for $m\leq n$.
For all $i\leq k$, use \cref{disjoint_condition1} on $U_{i,n}\cap f^{-1}(V_n)$ (that is an open set containing $x_i$) to find reductions $(\sigma_i,\tau_i)$ from $f_i$ to $f\corestr{V_n}$.
Use \cref{disjoint_condition2} on $V_n$ to find a reduction $(\sigma_{k+1},\tau_{k+1})$ from $f_{k+1}$ to $f\corestr{V_n}$.
Together, \cref{disjoint_condition1,disjoint_condition2} guarantee that $y$ is not in the closed set $\bigcup_{i\leq k+1}\overline{\im(f\sigma_i)}$, so
we can find a natural number $N$ such that $V_{n+1}:=\nbhd{y}{N}$ is disjoint from $\bigcup_{i\leq k+1}\overline{\im(f\sigma_i)}$.
But by continuity of $f$ at $x_i$, for all $i\leq k$ since $f(x_i)=y\in V_{n+1}$ there exists a natural number $N_i\geq n$ such that $U_{i,n+1}=\nbhd{x_i}{N_i}\subseteq U_{i,n}\cap f^{-1}(V_{n+1})$.
This finishes the inductive construction.

Fix now a bijection $\iota:\N\rao(k+1)\times\N$ and for all $m\in\N$ write $\iota(m)=(\iota_0(m),\iota_1(m))$. For all $m\in\N$,
pick a continuous reduction $(\sigma_{m},\tau_{m})$ from $f_{\iota_0(m)}$ to $f\corestr{V_{m}\setminus V_{m+1}}$ as guaranteed by the above construction.

Name $\varphi$ the inverse of $\iota$ and define $(\sigma,\tau)$ as follows:
\begin{align*}
\sigma((i)\conc\iw{0})&=x_i, \\
\sigma((i)\conc(0)^n\conc(1)\conc x)&=\sigma_{\varphi(i,n)}(x), &&\text{for all $i\leq k$ and all $n\in\N$},\\
\sigma((k+1+n)\conc x)&=\sigma_{\varphi(k+1,n)}(x), &&\text{for all $n\in\N$}.
\end{align*}
Also, we let $\tau(y)=\iw{0}$ and if $z\in \dom \tau_m \subseteq V_m\setminus V_{m+1}$
we set 
\[
\tau(z)=(0)^{\iota_1(m)}\conc (1)\conc (\iota_0(m))\conc \tau_m(z).
\]

We claim that the pair $(\sigma,\tau)$ is a reduction from $g:=\fctwedge (f_0,\ldots,f_k\mid f_{k+1})$ to $f$. By construction we have $g=\tau f\sigma$.
To see that $\sigma$ is continuous, we use \cref{prop:sufficientcondforcont} with $U=\dom g \setminus \set{(i)\conc \iw{0}}[i\leq k]$.
As $\sigma$ is continuous on $U$ by \cref{lem:ContUnion}, we show the second condition of \cref{prop:sufficientcondforcont}.
Take $(x_n)_n$ in $U$ converging to some $(i)\conc \iw{0}$. Therefore there exists $M$ such that for all $n\geq M$ we can write $x_n=(i)\conc(0)^{k_n}\conc(1) \conc x'_n$,
with $k_n\to \infty$. Therefore $\sigma(x_n)=\sigma_{\varphi(i,k_n)}(x'_n)\in U_{i,k_n}$ for all $n\geq M$ and since $k_n\to \infty$ and $U_{i,k_n}\to x_i$ we get $\sigma(x_n)\to x_i=\sigma((i)\conc\iw{0})=\sigma(x)$, as desired.

Continuity of $\tau$ similarly follows from \cref{prop:sufficientcondforcont}, this time applied to the open set $\dom \tau \setminus \set{y}=\bigsqcup_m \dom \tau_m$ by remarking that $\iota_1(m_k)\to \infty$ when $m_k\to \infty$. 
\end{proof}

\section{Finite generation at successors of limits}

As a warm-up for the general case, we can now apply the new operation to complete the analysis of $\sC_{\lambda+1}$ when $\lambda$ is 1 or limit.
This will allow us to see all the ingredients of the proof of the general case in a simplified setting.

We need following elaboration on \cref{InfiniteEmbedOmega}. 

\begin{lemma}\label{InfiniteEmbedOmegaStronger}
Let $n\in \N$ and $X_0,\ldots, X_n$ be infinite subsets of a metrizable space $B$.
Then there are pairwise disjoint infinite sets $Y_i\subseteq X_i$ for all $i\leq n$ such that $\bigcup_{i\leq n} Y_i$ is discrete.
\end{lemma}
\begin{proof}
We prove the statement by induction on $n$. If $n=0$, this is simply \cref{InfiniteEmbedOmega}.
So assume it holds for all collections of size $n$ and let $X_0,\ldots, X_{n+1}$ be infinite subsets of $B$. We distinguish two cases.

First case: there exist distinct $i,j\leq n+1$ such that $X_{i}\cap X_{j}$ is infinite.
Reindexing if necessary, assume $j=n+1$ and define $\tilde{X}_{i}=X_{i}\cap X_{n+1}$ and $\tilde{X}_{k}=X_k$ otherwise.
We can apply the induction hypothesis to $\tilde{X}_0,\ldots \tilde{X}_n$ to obtain pairwise disjoint infinite sets $\tilde{Y}_i\subseteq \tilde{X}_{i}$, for $i\leq n$, whose union is discrete.
To obtain the desired sets for the original collection, it suffices to set $Y_j=\tilde{Y}_j$ for $j\neq i$ and further partition $\tilde{Y}_{i}$ into two infinite sets $Y_i$ and $Y_{n+1}$.

Second case: for all $i,j\leq n+1$ the set $X_i\cap X_j$ is finite. We can make them disjoint: set $\tilde{X}_i=X_i\setminus\bigcup_{j\neq i}X_j$ for all $i \leq n$. 
Apply the induction hypothesis to $\tilde{X}_i$ for $i\leq n$ to get infinite pairwise disjoint sets $Y_i\subseteq \tilde{X}_i$, $i\leq n$ with a discrete union.
Now pick a discrete $Y_{n+1}\subseteq \tilde{X}_{n+1}$ by \cref{InfiniteEmbedOmega} and note that the sets $Y_i$ for $i\leq n+1$ are infinite pairwise disjoint discrete sets. 
If $Y=\bigcup_{i\leq n+1} Y_i$ is discrete as well, we are done.
Otherwise, there exists $i\leq n$ such that either $Y_i\cap \overline{Y_{n+1}}\neq \emptyset$ or $\overline{Y_i}\cap Y_{n+1}\neq \emptyset$.  
Suppose that the former holds, choose $y\in Y_i\cap \overline{Y_{n+1}}$ and an open neighborhood $V$ of $y$ such that $V\cap Y_i= \{y\}$ (since $Y_i$ is discrete).
Note that $Y_{n+1}\cap V$ is infinite and so we can shrink these two sets as follows $Y_{n+1}:=Y_{n+1}\cap V$ and $Y_i:=Y_i\setminus \set{y}$.
Both $Y_i$ and $Y_{n+1}$ are infinite and now $Y_i\cap \overline{Y_{n+1}}=\emptyset$.
The latter case is similar and since $n$ is finite, we can apply this procedure repeatedly until we obtain the desired family $(Y_i)_{i\leq n+1}$.
\end{proof}

As an application, we show that if for infinitely many points $y\in \im f$ we can reduce a single function $g$ to corestrictions of $f$ on arbitrary small neighborhood of $y$, then we have $\omega g \leq f$. We will use the stronger statement which involves a finite set $G$ in place of $G$ in \cref{sectionDoubleSucc}.
Given a (finite) set $G$ of functions in $\functionsonbaire$, we write $\omega G$\index[IdxSymb]{w@{$\omega G$}} for $\omega(\gl G)$.

\begin{lemma}\label{Intertwinereductions}
Let $f:A\to B$ between metrizable spaces and $G\subseteq\functionsonbaire$ be finite.
Suppose that for all $g\in G$ there are infinitely many points $y\in B$ such that for all neighbourhood $V$ of $y$ we have $g\leq f\corestr{V}$.
Then $\omega G \leq f$.
\end{lemma}
\begin{proof}
If $G$ is empty, then $\omega G=\emptyset$ which reduces to any function, so suppose that $G\neq\emptyset$. 
For each $g\in G$, let $X_g$ be the set of points $y\in B$ such that for all neighborhood $V\ni y$ we have $g\leq f\corestr{V}$. 
Since the sets $X_g$ are infinite subsets of the metrizable space $B$ and $G$ is finite, we can apply \cref{InfiniteEmbedOmegaStronger} to get an injective map $G\times \N\to B$, $(g,n)\mapsto y^g_n$ with discrete image such that  $y^g_n\in X_g$ for all $g\in G$ and all $n\in \N$. As $B$ is metrizable, we can find pairwise disjoint open sets $V^g_n$ of $B$ such that $y^g_n\in V^g_n$ for all $g\in G$ and $n\in\N$. Therefore $\omega G \leq f$ holds by \cref{Gluingaslowerbound}.
\end{proof}

\begin{lemma}\label{Intertwinereductionsforomegacentered}
Let $f\in\functionsonbaire$ be continuous and $G\subseteq\functionsonbaire$ be a finite set of centered functions.
If $\omega g \leq f$ for all $g\in G$, then $\omega G \leq f$. 
Moreover, if $f=\bigsqcup_{i=0}^{n} f_i$ for some $n\in\N$, then
\begin{enumerate}
\item if $g$ is centered and $\omega g\leq f$, then $\omega g\leq f_i$ for some $i\leq n$,
\item if $\lambda$ is limit and $\Maximalfct{\lambda}\leq f$, then $\Maximalfct{\lambda} \leq f_i$ for some $i\leq n$.
\end{enumerate}
\end{lemma}
\begin{proof}
Take a center $c_g$ for $g\in G$ and a continuous reduction $(\sigma,\tau)$ from $\omega g$ to $f$. The set $\set{f\sigma((n)\conc c_g)}[n\in\N]$ is infinite and by \cref{Centerinvariance} we have $g\leq f\restr{V}$ whenever $V$ is a neighbourhood of $\sigma((n)\conc c_g)$. So by continuity of $f$, for every $n$ and every neighborhood $U$ of $f\sigma((n)\conc c_g)$ it follows that $g\leq f\corestr{U}$. This implies that $\omega G\leq f$ by \cref{Intertwinereductions}.

For the second part, take a center $c_g$ for $g$ and a continuous reduction $(\sigma,\tau)$ from $\omega g$ to $f$. As $f=\bigsqcup_{i\leq n}f_i$, there exists $i\leq n$ such that $\dom f_i$ contains $\sigma((n)\conc c_g)$ for infinitely many $n$. We conclude as before. Finally for $\lambda$ limit, using \cref{Gluingcohomomorphism} and \cref{JSLgeneralstructure} notice that for every sequence $(\beta_k)_{k\in \N}$ cofinal in $\lambda$, we have $\Maximalfct{\lambda}\equiv \gl_{k\in \N} \Minimalfct{\beta_k}$. Choose an increasing sequence $(\beta_k)_{k\in \N}$ cofinal in $\lambda$ and recall that $\iw{0}$ is a center for each $\Minimalfct{\beta_k}$. If $(\sigma,\tau)$ witness $\gl_{k\in\N} \Minimalfct{\beta_k}\leq \bigsqcup_{i\leq n}f_i$, then there exists $i\leq n$ such that $X=\{k\in\N \mid \sigma((n)\conc \iw{0}) \in \dom f_i\}$ is infinite. Note that the subsequence $(\beta_k)_{k\in X}$ is cofinal in $\lambda$, so $\Maximalfct{\lambda}\equiv \gl_{k\in X} \Minimalfct{\beta_k}\leq f_i$, as before. 
\end{proof}

We are now ready to show that $\sC_{\lambda+1}$ is finitely generated, assuming that $\sC_{<\lambda}$ is \bqo{} under continuous reducibility.  We first isolate the only situation where the wedge operation is necessary.

\begin{lemma}\label{Diagonalforlambda+1}
Suppose that $f=\bigsqcup_{n\in\N} f_n:A_n\to B$ for simple functions $f_n\in\sC_{\lambda+1}$ with pairwise distinct distinguished points $y_n$. Assume that the following hold: 
\begin{enumerate}
\item $f_0\equiv \pgl \Maximalfct{\lambda}$,
\item $f_n\leq \Minimalfct{\lambda+1}\glbin \Maximalfct{\lambda}$ whenever $n>0$, and
\item $(y_n)_{n>0}$ converges to $y_0$. 
\end{enumerate}
Then for all clopen neighborhood $U$ of $y_0$ we have $f\leq \fctwedge(\Maximalfct{\lambda}\mid \Minimalfct{\lambda+1})\leq f\corestr{U}$.
\end{lemma}
\begin{proof}
To prove that $\fctwedge(\Maximalfct{\lambda}\mid\Minimalfct{\lambda+1})\leq f\corestr{U}$, we use the Disjointification \cref{DisjointificationLemma} with $y=y_0$. Note that $\pgl\Maximalfct{\lambda}\equiv f_0$ is centered with \cocenter{} $y_0$ and so we can choose a reduction $(\sigma,\tau)$ witnessing $\pgl\Maximalfct{\lambda}\leq f_0\corestr{U}$. We choose $x=\sigma(\iw{0})$. To verify the first condition, let $W$ be any clopen neighborhood of $x$ in $\dom f$, let $W'=W\cap  f^{-1}(U)$. We have $\pgl\Maximalfct{\lambda}\geq f_0\restr{W'}$ and since $x\in W'$ we have $\pgl\Maximalfct{\lambda}\leq f_0\restr{W'}$ by \cref{Centerinvariance}. So picking a witness for the latter reduction, we get the desired reduction $\Maximalfct{\lambda}\leq f_0\corestr{U}$ whose $\sigma$ satisfy $\im \sigma \subseteq W$ and $y_0\notin \overline{\im f\sigma}$ by \cref{Rigidityofthecocenter}. 
For the second condition, let $V$ be a neighborhood of $y_0$ in $U$ and note that there exists $n>0$ with $y_n\in V$, since $(y_n)_{n>0}$ converges to $y_0$. Then choosing a clopen neighborhood $W$ of $y_n$ in $V$, note that $\CB(f\corestr{W})=\lambda+1$ so $\Minimalfct{\lambda+1}\leq f\corestr{W}$ by \cref{Minfunctions}, which grants us with the desired reduction.

Next, we show that $f\leq \fctwedge(\Maximalfct{\lambda}\mid\Minimalfct{\lambda+1})$ using \cref{Wedgeasupperbound}.
Let $B_k=N_{y_0\restr{k}}\cap B$, $k\in \N$, denote the clopen neighborhood basis at $y_0$. 
For every $n>0$, as $y_n\neq y_0$, let $k_n$ be the largest $k\in\N$ such that $y_n\in B_{k}$. Note that $f_n\corestr{B_{k_n}}$ has rank $\lambda+1$, since $y_n\in B_{k_n}$, while $f_n\corestr{B\setminus B_{k_n}}$ has rank at most $\lambda$. Since $f_n\corestr{B_{k_n}} \leq f_n \leq \Minimalfct{\lambda+1} \gl \Maximalfct{\lambda}$, we can pick a clopen subset $A^1_n$ of $f_n^{-1}(B_{k_n})$ such that $f\restr{A^1_n}\leq\Minimalfct{\lambda+1}$ and $A^0_n=A_n \setminus A^1_n$ satisfy $\CB(f\restr{A^0_n})\leq \lambda$ by \cref{Gluingasupperbound}. Note that $f(A^1_n)\subseteq B_{k_n}$ and $k_n\to \infty$ as $n\to \infty$ since $(y_n)_{n>0}$ converges to $y_0$, so $f(A^1_n)$ converges to $y_0$.

Setting $A^0=\bigcup_{n>0}A^0_n$ we obtain a clopen set of $\dom f$ with $f\restr{A^0}\leq\Maximalfct{\lambda}$ by \cref{CBrankofclopenunion}. We define a clopen partition of $\dom f$ by setting $\tilde{A}_0= A_0 \cup A^0$ and $\tilde{A}_n=A^1_n$ for $n>0$. Note that the rays of $f\restr{\tilde{A}_0}$ at $y_0$ all have $\CB$-rank $\leq \lambda$ and so they are all reducible to $\Maximalfct{\lambda}$. Moreover $f\restr{\tilde{A_n}}\leq \Minimalfct{\lambda+1}$ for all $n>0$ and since $f(\tilde{A}_n)_{n>0}$ converges to $y_0$, as desired.
\end{proof}

\index[IdxNotion]{Precise Structure!successor of limit}
\begin{theorem}\label{FGatsuccessoroflimit}
Let $\lambda$ be limit or 1. Suppose that continuous reducibility is \bqo{} on $\sC_{<\lambda}$.
Then $\sC_{\lambda+1}$ is generated by the finite set 
\[\set{\Maximalfct{\lambda},\Minimalfct{\lambda+1},\pgl\Maximalfct{\lambda},\omega\Minimalfct{\lambda+1},
\fctwedge(\Maximalfct{\lambda}\mid\Minimalfct{\lambda+1}),\Maximalfct{\lambda+1}}.\]
\end{theorem}

\begin{proof}
By \cref{finitedegreedamuddafuckaz}, we can assume that $f:A\to B$ in $\sC_{\lambda+1}$ has infinite $\CB$-degree.
By the Decomposition \cref{JSLdecompositionlemma}, $f$ is locally simple of $\CB$-rank $\lambda+1$. {Note that a disjoint union of simple functions of $\CB$-rank $\lambda+1$ with same distinguished point is again simple of $\CB$-rank $\lambda+1$.} So we can write $f=\bigsqcup_{n\in\N}f_n$ with each $f_n:A_n\to B$ simple of $\CB$-rank $\lambda+1$ with distinguished point $y_n$ so that the points $y_n$, $n\in \N$, are pairwise distinct.  

Since $\Minimalfct{\lambda+1}\leq f_n\corestr{V}$ for all $n\in\N$ and neighbourhood $V\ni y_n$ by \cref{Minfunctions}, we obtain $\omega\Minimalfct{\lambda+1}\leq f$ by \cref{Intertwinereductionsforomegacentered}.

We partition $\N=N_0\sqcup N_1$ with $N_1=\set{n\in\N}[f_n\equiv\pgl\Maximalfct{\lambda}]$ and set $Y_i=\set{y_n}[n\in N_i]$, $i=0,1$. We distinguish several cases:

\begin{descThm}
\item[$N_1$ is infinite] By centeredness of $\pgl\Maximalfct{\lambda}$ we can use \cref{Intertwinereductionsforomegacentered} again and get $\Maximalfct{\lambda+1}\leq f$.
As we know that $f\leq\Maximalfct{\lambda+1}$ by \cref{Maxfunctions}, we obtain $f\equiv\Maximalfct{\lambda+1}$ in this case.

\item[$N_1$ is empty] Then by \cref{simplefunctionslambda+1damuddafuckaz}, for all $i$
we actually have either $f_i\equiv\Minimalfct{\lambda+1}$ or $f_i\equiv\Minimalfct{\lambda+1}\gl\Maximalfct{\lambda}$, but in both cases $f_i\leq\Minimalfct{\lambda+1}\gl\Maximalfct{\lambda}\leq 2 \Minimalfct{\lambda+1}$. Therefore, we obtain using \cref{Gluingasupperbound}:
\[\omega\Minimalfct{\lambda+1}\leq f\leq\gl_if_i\leq\omega(\Minimalfct{\lambda+1}\glbin\Maximalfct{\lambda})\leq\omega(2\Minimalfct{\lambda+1})\equiv\omega\Minimalfct{\lambda+1},\]
so in that case $f\equiv\omega\Minimalfct{\lambda+1}$ and we are also done.

\item[$N_1$ is finite and non empty] Suppose first that for some clopen $U\supseteq Y_1$, the set of $n\in N_0$ such that $y_n\notin U$ is infinite,
then $\omega\Minimalfct{\lambda+1}\leq f\corestr{B\setminus U}$ by \cref{Intertwinereductionsforomegacentered} and $|N_1|\cdot\pgl\Maximalfct{\lambda}\leq f\corestr{U}$ by centeredness and \cref{Gluingaslowerbound}, so
\[f\leq\gl_nf_n\leq (\gl_{n\in N_1}f_{n})\glbin(\gl_{n\in N_0}f_n)\leq|N_1|\cdot\pgl\Maximalfct{\lambda}\glbin\omega\Minimalfct{\lambda+1}\leq f\corestr{U}\glbin f\corestr{B\setminus U}\leq f,\]
and in this case $f\equiv|N_1|\pgl\Maximalfct{\lambda}\gl\omega\Minimalfct{\lambda+1}$, so we are done.

Otherwise, any clopen $U\supseteq Y_1$ contains all but finitely many of the points $y_n$, $n\in N_0$.  
In this case, choose pairwise disjoint clopen sets $U_i\ni y_i$, $i\in N_1$, such that for all $i\in N_1$ the set $P_i=\set{n\in \N }[y_n\in U_i]$ is either infinite or equal to $\set{i}$. 
Since $Y_1$ is included in the clopen set $\bigcup_{i\in N_1}U_i$, it follows that $R=\N\setminus \bigcup_{i\in N_1} P_i$ is finite. By adding $R$ to some infinite $P_i$, we get that $(P_i)_{i\in I}$ is a finite partition of $\N$ and we define $f^{i}=\bigsqcup_{n\in P_i} f_n$.

If $P_i=\set{i}$, then 
\[f^{i}\leq \pgl\Maximalfct{\lambda} \leq f\corestr{U_i},\]
since $f^{i}=f_i\equiv \pgl\Maximalfct{\lambda}$ and $\pgl\Maximalfct{\lambda}$ is centered.

If $P_i$ is infinite, note that the sequence $(y_n)_{n\in P_i\setminus \{i\}}$ converges to $y_i$ and so $f^{i}=\bigsqcup_{n\in P_i} f_n$ falls in the case presented in \cref{Diagonalforlambda+1}. Therefore we obtain $f^{i}\leq \fctwedge(\Maximalfct{\lambda}\mid\Minimalfct{\lambda+1}) \leq f\corestr{U_i}$ in this case.

Now it follows from \cref{Gluingasupperbound} and \cref{Gluingaslowerbound}, that
\[\gl_{i\in I}f\corestr{U_i}\leq f\leq\gl_{i\in I}f^i,\]
so in this last case the function is a finite gluing of $ \fctwedge(\Maximalfct{\lambda}\mid\Minimalfct{\lambda+1})$ and $\pgl\Maximalfct{\lambda}$. \qedhere
\end{descThm}
\end{proof}

\section{The generators}\label{subsectionGenerators}

Given a finite set $F$ of functions, we introduce the following notations. We write $ \mathcal{P}^+(F)$\index[IdxSymb]{P@{$\mathcal{P}^+(F)$}} for the set of non-empty subsets of a set $F$. We define $\pgl\mathcal{P}^+(F):=\set{\pgl F'}[\emptyset\neq F'\subseteq F]$\index[IdxSymb]{ptgl@{$\pgl\mathcal{P}^+(F)$}} and $\omega\set{F}:=\set{\omega f}[f\in F]$.\index[IdxSymb]{w@{$\omega\set{F}$}}
Finally, for $\mathcal{F}\in \mathcal{P}^+ (\mathcal{P}^+(F))$ and $D$ a finite set of functions, we write $\fctwedge(\mathcal{F}\mid D)=\fctwedge(\{\gl G\mid G\in \mathcal{F}\} \mid \gl D)$, thereby leveraging \cref{cor:wedgeSets} and the remark that follows.\index[IdxSymb]{V@{$\fctwedge(\mathcal{F}\mid D)$}}

We now define for every $\alpha<\omega_1$ three finite sets: a set of centered functions $\centered{\alpha}$, a set $\omegacentered{\alpha}$ of \emph{$\omega$-functions}, \index[IdxNotion]{o@$\omega$-functions}
and the generators $\generator{\alpha}$. \index[IdxSymb]{W@{$\omegacentered{\alpha}$}} \index[IdxSymb]{C@{$\centered{\alpha}$}} \index[IdxSymb]{G@{$\generator{\alpha}$}}

\index[IdxNotion]{generators}
For $\lambda<\omega_1$ limit or null, we set
\[
\centered{\lambda}=\emptyset,\qquad \omegacentered{\lambda}=\generator{\lambda}=\begin{cases}\emptyset&\text{if $\lambda=0$}\\ \set{\Maximalfct{\lambda}}&\text{$\lambda$ is limit,}\end{cases}
\]
and then we define $\centered{\lambda+n+1}$, $\omegacentered{\lambda+n+1}$, $\generator{\lambda+n+1}$ simultaneously by induction on $n\in\N$:
\begin{align*}
\centered{\lambda+n+1}=&\set{\Minimalfct{\lambda+1}}\cup\centered{\lambda+n}\cup\pgl\mathcal P^+\big(\centered{\lambda+n}\cup\omegacentered{\lambda+n}\big),\\
\omegacentered{\lambda+n+1}=&\omegacentered{\lambda+n}\cup\omega\set{\centered{\lambda+n+1}},\\
\generator{\lambda+n+1}=&\centered{\lambda+n+1}\cup\omegacentered{\lambda+n+1}\cup\set{\fctwedge(\mathcal S\mid D)}[\mathcal S\in\mathcal P^+(\mathcal P^+(\generator{\lambda+n})),\ D\subseteq\centered{\lambda+n+1}],
\end{align*}
where members of the last family in the definition $\generator{\lambda+n+1}$ are referred to as \emph{wedge generators}.

Note that by our convention $\generator{0}=\emptyset$ generates the only function of null $\CB$-rank, namely the empty function.

\begin{fact}\label{AlreadyKnownGenerators}
\begin{enumerate}
\item We have $\generator{1}=\{\Minimalfct{1}, \omega \Minimalfct{1}\}$,
\item For $\lambda$ limit, the set $\generator{\lambda+1}$ is equal up to elementwise continuous equivalence to 
\[\set{\Maximalfct{\lambda},\Minimalfct{\lambda+1},\pgl\Maximalfct{\lambda},\omega\Minimalfct{\lambda+1},
\fctwedge(\Maximalfct{\lambda}\mid\Minimalfct{\lambda+1}),\Maximalfct{\lambda+1}}.\]
\end{enumerate}
\end{fact}
\begin{proof}
Note that $\centered{0}=\emptyset$, $\generator{0}=\emptyset$, and $\centered{1}=\{\Minimalfct{1}\}$, so $\generator{1}=\{\Minimalfct{1},\omega \Minimalfct{1}\}$.
For the second item, if $\lambda<\omega_1$ is limit, then $\centered{\lambda+1}=\set{\Minimalfct{\lambda+1},\pgl\Maximalfct\lambda}$, $\omegacentered{\lambda+1}=\set{\Maximalfct\lambda,\omega\Minimalfct{\lambda+1},\Maximalfct{\lambda+1}}$ (using $\Maximalfct{\lambda+1}=\omega\pgl\Maximalfct\lambda$), and the wedge clause contributes the four $\fctwedge(\Maximalfct\lambda\mid D)$, $D\subseteq\centered{\lambda+1}$, of which $D=\emptyset$ gives $\pgl\Maximalfct\lambda$ and $\pgl\Maximalfct\lambda\in D$ gives $\Maximalfct{\lambda+1}$. So: $\generator{\lambda+1}$ has exactly these six elements up to $\equiv$.
\end{proof}

The following lemma justifies the terminology $\omega$-functions.

\begin{lemma}\label{lem:omegagenerators}
For all $h\in\omegacentered{\alpha}$ we have $\omega h\equiv h$. Moreover $\omegacentered{n}=\omega\set{\centered{n}}$ and $\omegacentered{\lambda+n}=\set{\Maximalfct{\lambda}}\cup\omega\set{\centered{\lambda+n}}$.
\end{lemma}
\begin{proof}
Observe that $\omega (\omega f)\equiv \omega f$ for any function $f$, which yields the second claim by induction.
Moreover, note that for $\lambda$ limit, the function $\omega\, \Maximalfct{\lambda}$ has $\CB$-rank $\lambda$ by \cref{BasicsOnGluing} and so it is continuously equivalent to $\Maximalfct{\lambda}$ by \cref{Maxfunctions}. Together, this proves the first claim.
\end{proof}

Moreover, note that by \cref{Intertwinereductionsforomegacentered} the $\omega$-functions enjoy a primeness property: if $h$ is an $\omega$-function and $h\leq \sqcup_{i=0}^n f_i$ for some $n\in \N$, then $h\leq f_i$ for some $i\leq n$.

Next, we record some basic facts on the generators sets. We use the notation $\sC_{[\lambda,\alpha]}$ for $\sC_{\leq \alpha}\cap\sC_{\geq\lambda}$. \index[IdxSymb]{C@{$\sC_{[\lambda,\alpha]}$}}

\begin{fact}\label{BasicsOnGenerators}
Let $\alpha$ be a countable ordinal.
\begin{enumerate}
\item We have $\centered{\alpha}\subseteq\centered{\alpha+1}$, $\omegacentered{\alpha}\subseteq\omegacentered{\alpha+1}$ and $\generator{\alpha}\subseteq\generator{\alpha+1}$. \label{monotonocityGenerators}
\item If $\lambda$ is the limit part of $\alpha$, then $\centered{\alpha}$, $\omegacentered{\alpha}$ and $\generator{\alpha}$ are included in $\sC_{[\lambda,\alpha]}$.
\item The sets $\centered{\alpha}$, $\omegacentered{\alpha}$ and $\generator{\alpha}$ are finite. 
\item Every function in $\centered{\alpha}$ is centered.
\end{enumerate}
\end{fact}

\begin{proof}
The first item follows by induction from the definitions.
Item (2) follows from \cref{CBrankofPgluingofregularsequence1,BasicsOnGluing,BasicfactsWedge}, note that $\Maximalfct{\lambda}\in \omegacentered{\alpha}\subseteq \generator{\alpha}$ has rank $\lambda$. Item (3) follows by induction, all three induction steps take (iterated) power sets of finite sets. For the last item, recall that $\Minimalfct{\lambda+1}$ and $\pgl G$ are centered as pointed gluing of regular sequences.
\end{proof}

Our goal is to prove:

\index[IdxNotion]{Precise Structure!Theorem}
\begin{theorem}[Precise Structure]\label{PreciseStructureThm}
For all $\alpha<\omega_1$, every function in $\sC_\alpha$ is continuously equivalent to a finite gluing of functions in $\generator{\alpha}$.
\end{theorem}

For $\alpha<\omega_1$, we call \emph{finite generation at level $\alpha$} the statement:
\begin{descThm}
\item[$\FG(\alpha)$] Every function in $\sC_\alpha$ is equivalent to a finite gluing of functions in $\generator{\alpha}$.\index[IdxSymb]{FG@{$\FG(\alpha)$, $\FG(<\alpha)$, $\FG(\leq\alpha)$}}
\end{descThm}
The notations $\FG(<\alpha)$ and $\FG(\leq\alpha)$ stand for the conjunction of $\FG(\beta)$ for all $\beta<\alpha$ and for all $\beta\leq\alpha$, respectively.
Let us collect the results we have already established about the statements $\FG(\alpha)$.

\index[IdxNotion]{finitely generated set of functions}
\begin{proposition}\label{FGconsequences} We have $\FG(0)$, $\FG(1)$ and $\FG(\lambda)$ for every $\lambda$ limit. 
Moreover for every countable ordinal $\alpha=\lambda+n$, with $\lambda$ limit or null and $n<\omega$, we have
\begin{enumerate}
\item $\FG(<\lambda)$ implies $\FG(\lambda+1)$.
\item $\FG(\leq\alpha)$ implies that every function in $\sC_{[\lambda,\alpha]}$ is a finite gluing of functions in $\generator{\alpha}$.
\item $\FG(<\alpha)$ implies that continuous reducibility is a \bqo{} on $\sC_{< \alpha}$.
\item If $\FG(<\alpha)$ holds then every function in $\sC_\alpha$ is a disjoint union of centered functions.\label{disjointUnionCenteredFromFG}
\end{enumerate}
\end{proposition}
\begin{proof}
Note that \cref{JSLgeneralstructure} trivially implies $\FG(\lambda)$ for $\lambda$ limit or null. Using \cref{AlreadyKnownGenerators}, \cref{LocallyConstantFunctions} yields $\FG(1)$. 
\begin{enumerate}
\item Note that $\FG(< \lambda)$ implies that $\sC_{\beta}$ is \bqo{} for every $\beta<\lambda$ by \cref{SecondstepforBQOthm}, which in turn implies $\sC_{<\lambda}$ is \bqo{} by \cref{FGgivesBQO_2}. Therefore $\FG(\lambda+1)$ follows from \cref{FGatsuccessoroflimit,AlreadyKnownGenerators}. 
\item Recall that $\generator{\lambda+k}\subseteq \generator{\lambda+k+1}$ for all $k<\omega$, so $G_{\alpha}=\bigcup_{k\leq n} G_{\lambda+k}$.
\item This follows from \cref{SecondstepforBQOthm,FGgivesBQO_2}.
\item This follows from the previous item combined with \cref{LocalCenterednessFromBQO} and \cref{ultrametricdecomposition}.   \qedhere
\end{enumerate}
\end{proof}

Here are some useful relations between our generator sets. 

\begin{lemma}\label{TransportAndBounding}
Let $\alpha=\lambda+n+1<\omega_1$ with $\lambda$ limit or null and $n\in \N$.
\begin{enumerate}
\item Every $g\in\generator{\lambda+n}$ with $g\neq\Maximalfct{\lambda}$ satisfies $E\leq g\leq\gl E$
for some non-empty $E\subseteq\centered{\lambda+n}\cup \omega \set{\centered{\lambda+n}}$.\label{item:pieces}
\item For every non-empty $F\subseteq\generator{\lambda+n}$, there exists a non-empty $F'\subseteq \centered{\lambda+n}\cup \omega\set{\centered{\lambda+n}}$ such that $\pgl F\equiv \pgl F'$, and so $\pgl F'\in \centered{\alpha}$. \label{item:transport}
\item For all wedge generator $f\in\generator{\alpha}$ there are $F,H\subseteq\centered{\alpha}$ such
that $\set{\omega H}\cup F\leq f$ and $f\leq(\gl F)\glbin\omega H$.\label{BasicsOnGenerators_boundingwedge}
\end{enumerate}
\end{lemma}

\begin{proof}
We prove the three items simultaneously by induction on $n$, the first one using the third at
smaller levels.

\cref{item:pieces} Since $g\neq\Maximalfct{\lambda}$, \cref{lem:omegagenerators} ensures that either $g\in\centered{\lambda+n}\cup \omega\set{\centered{\lambda+n}}$ or $g$ is a wedge generator. In the former case, take $E=\set{g}$.
In the latter case, $n>0$ and \cref{BasicsOnGenerators_boundingwedge} at level
$\lambda+n$ yields $F,H\subseteq\centered{\lambda+n}$ with $\set{\omega H}\cup F\leq g$ and
$g\leq(\gl F)\glbin\omega H$. Then $E=F\cup\omega\set{H}$ works, since $\omega H\equiv\gl\omega\set{H}$.

\cref{item:transport} If $F=\set{\Maximalfct{\lambda}}$, then $\pgl F=\pgl\Maximalfct{\lambda}$ belongs
to $\centered{\lambda+1}\subseteq\centered{\alpha}$ by \cref{BasicsOnGenerators}. Otherwise every
$g\in F\setminus\set{\Maximalfct{\lambda}}$ has $\CB(g)>\lambda$, so $\Maximalfct{\lambda}\leq g$ by
\cref{JSLgeneralstructure} and we may assume $\Maximalfct{\lambda}\notin F$ by
\cref{Pgluingasupperbound}. For $g\in F$ let $E_g$ be given by \cref{item:pieces} and set
$F'=\bigcup_{g\in F}E_g$. Then $\iw{F}$ is equivalent by pieces to $\iw{F'}$, so $\pgl F\equiv\pgl F'$
by \cref{Pgluingasupperbound}. As $F'$ is a non-empty subset of
$\centered{\lambda+n}\cup\omegacentered{\lambda+n}$, we have $\pgl F'\in\centered{\alpha}$ by definition.

\cref{BasicsOnGenerators_boundingwedge} Let $f=\fctwedge(\mathcal S\mid D)$ for some
$\mathcal S\in\mathcal P^+(\mathcal P^+(\generator{\lambda+n}))$ and $D\subseteq\centered{\alpha}$. By
\cref{item:transport}, choose for each $G\in\mathcal S$ some $f_G\in\centered{\alpha}$ with
$f_G\equiv\pgl G$, and set $F=\set{f_G}[G\in\mathcal S]$ and $H=D$. Then $f\equiv \fctwedge(F\mid D)$ by \cref{cor:wedgeSets} and we have $F\leq f$, $\omega H\leq f$
and $f\leq(\gl F)\glbin\omega H$ by \cref{Gluingasupperbound_cor}.\qedhere
\end{proof}

This allows us to prove that $\FinGl{\generator{\alpha}}$ is actually closed under infinite gluing too.

\begin{lemma}\label{OmegaOfGenerators}
For countable ordinal $\alpha$, the following holds.
\begin{enumerate}
\item For all $g\in \generator{\alpha}$, $\omega g$ is equivalent to $\gl G$ for some $G\subseteq \omegacentered{\alpha}$.
\item If $f_n\in\FinGl{\generator{\alpha}}$ for all $n\in\N$, then $\gl_{n\in\N}f_n$ is equivalent to a function in $\FinGl{\generator{\alpha}}$.\label{trickforfinalproof}
\end{enumerate}
 
\end{lemma}

\begin{proof}
Item (1): In case $g$ is a wedge generator, take $F,H\subseteq \centered{\alpha}$ given by \cref{TransportAndBounding}~\cref{BasicsOnGenerators_boundingwedge}. By \cref{Gluingasupperbound,Gluingaslowerbound} we get $\omega g\equiv \omega (F\cup H)$, so $G=\{\omega f \mid f\in F\cup H\}\subseteq \omegacentered{\alpha}$ is as desired.

Item (2): Up to continuous equivalence we can suppose that all $f_n$ are actually in $\generator{\alpha}$.
Note that by finiteness of $\generator{\alpha}$ there is $N\in\N$ and a partition $I_0,\ldots,I_N$ of $\N$ such that for all $n\leq N$ and $i,j\in I_n$
we have $f_i=f_j=g_n$ for some $g_n\in\generator{\alpha}$.
Now if $I_n$ is infinite, we have $\gl_{i\in I_n}f_i\equiv \omega g_n$, which is equivalent to a finite gluing of functions in $\omegacentered{\alpha}$ by the previous item.
\end{proof}

Finally, under $\FG(<\alpha)$, every scattered centered function is equivalent to one in $\centered{\alpha}$. 

\begin{corollary}\label{CenteredAreGenerated}
Let $\alpha$ be countable ordinal with limit part $\lambda$ and assume $\FG(<\alpha)$. Then every centered function in $\sC_{[\lambda,\alpha]}$ is equivalent to some
function in $\centered{\alpha}$, namely either to $\Minimalfct{\lambda+1}$ or to $\pgl G$ for some
non-empty $G\subseteq\centered{\alpha-1}\cup\omega\set{\centered{\alpha-1}}$.
\end{corollary}
\begin{proof}
The assumption $\FG(<\alpha)$ ensures that $\generator{\alpha-1}$
finitely generates $\sC_{[\lambda,\alpha-1]}$, so any centered $f\in\sC_{[\lambda,\alpha]}$
is equivalent to $\Minimalfct{\lambda+1}$ or to $\pgl F$ for some non-empty finite
$F\subseteq\generator{\alpha-1}$ by \cref{finitenessofcenteredfunctions}. We conclude by
\cref{TransportAndBounding}~\cref{item:transport}.
\end{proof}


\chapter{Finite generation at double successors}\label{sectionDoubleSucc}

In view of \cref{FGconsequences}, it remains to show that $\FG(<\alpha+2)$ implies $\FG(\alpha+2)$ in order to complete the proof by induction of \cref{PreciseStructureThm}.
To do so, writing $\alpha=\lambda+n$ with $\lambda$ limit and $n\in\N$, we can can rely on a decomposition in centered functions, which are either in $\sC_{<\lambda}$ or in $\centered{\alpha+2}$, up to equivalence.
The situation is more complex than for successor of limits, however we can actually assume that they are all in $\centered{\alpha+2}$ thanks to \cref{ConsequencesGeneralStructureThm}~\cref{ConsequencesGeneralStructureThm2} (see \cref{lem:gobblingLessThanLambda}).
For this general case, we show that partitions in centered functions can be made particularly well-behaved. We then analyze two gradually more complex situations before we define solvable functions. We finally show that solvable functions are equivalent to finite gluings of our generators and how to reduce the general case to that of solvable functions.

Note that this strategy departs from the one we used at successors of limit (\cref{FGatsuccessoroflimit}) which relied on two structural properties of continuous reducibility specific to these levels. 
First, by \cref{cor:CenteredSucessor} there are up to equivalence only two centered functions: the minimum and the maximum among simple functions, which makes them comparable.
Second, the only wedge generator which is not equivalent to a function in $\centered{\lambda+1}\cup \omegacentered{\lambda+1}$ is $\fctwedge(\Maximalfct{\lambda}\mid\Minimalfct{\lambda+1})$, which is actually of infinite $\CB$-degree.

\section{Fine partitions in centered functions}\label{DoubleSucc_domain}

We now introduce some terminology and notations for partitions in centered functions.

\ynote*{$\omegacentered{}$-regular functions}{
\begin{definition}
Let us say that $f\in \sC_\alpha$ is \emph{$\omegacentered{}$-regular}\index[IdxNotion]{function!$\omegacentered{}$-regular at a point}\index[IdxNotion]{w@{$\omegacentered{}$-regular function}} at $y\in B$ if for all $h\in \omegacentered{\alpha}$ the set $\set{j\in \N}[ h\leq \ray{f}{y,j}]$ is either empty or infinite. 
\end{definition}
}

\begin{lemma}\label{centeredImpliesWregular}
Any centered function\index[IdxNotion]{function!centered} $f\in \sC$ is $\omegacentered{}$-regular at its \cocenter{} $y$.
\end{lemma}
\begin{proof}
Let $f \in \sC$ with $\CB(f)=\alpha$. Note first that for all $i,m\in \N$ we have $\ray{f}{y,i}\leq \gl_{m=j}^M\ray{f}{y,j}$ for some $M>m$ by \cref{Rigidityofthecocenter}. By \cref{lem:omegagenerators}, $w\in \omegacentered{\alpha}$ is either $\omega g$ for $g\in\centered{\alpha}$ or $\Maximalfct{\lambda}$ for $\lambda$ the limit part of $\alpha$. In both cases, \cref{Intertwinereductionsforomegacentered} ensures that $w\leq \ray{f}{y,i}$ implies $w\leq \ray{f}{y,j}$ for arbitrarily large $j$.
\end{proof}

Given topological spaces $A$ and $B$, a \emph{$c$-partition}\index[IdxNotion]{partition!$c$-partition} of a function $f:A\to B$ in $\sC$ is a clopen partition $\partit\subseteq \mathbf{\Delta}^0_1(A)$ of $A$
such that for all $P\in \partit$ the restriction $f\restr{P}$ is centered with \cocenter{} denoted by $y_P$.
\index[IdxSymb]{P@{$\partit$}}\index[IdxSymb]{y@{$y_P$}}

If $\partit$ is a $c$-partition for $f$, we denote by $Y_\partit=\set{y_P}[ P\in \partit]\subseteq B$ the (countable) set of \cocenter{}s of functions $f\restr{P}$ for $P\in \partit$. \ynote*{introduce block}{Using $\partit$, we decompose $f$ into \emph{blocks} $f_{(g,y)}=\bigsqcup_{P\in \partit_{g,y}}f\restr{P}$ where $g$ is a centered function, $y\in Y_\partit$ a cocenter and $\partit_{(g,y)}=\set{P\in \partit}[f\restr{P}\equiv g \text{ and } y_P=y]$.}
\index[IdxSymb]{Y@{$Y_\partit$}}\index[IdxSymb]{P@{$\partit_{g,y}$}}\index[IdxSymb]{f@{$f_{(g,y)}$}}

\begin{definition}
\index[IdxNotion]{lump}
Given a $c$-partition $\partit$ of $f$, a \emph{$\partit$-lump} is a pair $(g,y)$ with $y\in Y_\partit$ and $g$ centered such that $f_{(g,y)}$ is not $\omegacentered{}$-regular at $y$.
The \emph{rank} of the $\partit$-lump is the rank $\CB(g)$ of $g$, which is equal to $\CB(f_{(g,y)})$.
\end{definition}

Note that $\partit_{g,y}$ is infinite whenever $(g,y)$ is a lump. Indeed, when $\partit_{g,y}$ is finite, $f_{(g,y)}$ is $\omegacentered{}$-regular at $y$. As each ray of $f_{(g,y)}$ at $y$ is the disjoint union of the corresponding rays of $f\restr{P}$ for $P\in\partit_{g,y}$, \cref{Intertwinereductionsforomegacentered} gives 
\[\set{i\in\N}[w\leq\ray{(f_{(g,y)})}{y,i}]=\bigcup_{P\in\partit_{g,y}}\set{i\in\N}[w\leq\ray{f\restr{P}}{y,i}]\]
for every $w\in\omegacentered{\CB(g)}$. Each set in this finite union is empty or infinite by \cref{centeredImpliesWregular}, so $f_{(g,y)}$ is $\omegacentered{}$-regular at $y$.

The goal of this section is to show the existence of the following particularly nice partitions which are free from lumps and from parts on which the function has rank below the limit. 
\begin{definition}
Let $f\in \sC_{\alpha+1}$ with $\lambda$ the limit part of $\alpha$. We say that a $c$-partition $\partit$ of $f$ is \emph{fine} if there are no $\partit$-lumps and $\CB(f\restr{P})>\lambda$ for all $P\in \partit$.\index[IdxNotion]{fine $c$-partition}\index[IdxNotion]{partition!fine $c$-partition}
\end{definition}

\subsection{Dissolving lumps}

Recall that a partition $\partit'$ is \emph{finer}\index[IdxNotion]{partition!finer} than a partition $\partit$ if every $P'\in\partit'$ is included in some $P\in\partit$.
We first show that any $c$-partition presenting a lump can be refined to \emph{dissolve} that lump.

\begin{lemma}\label{lemma:RefiningBy1}
Let $\alpha<\omega_1$ and assume $\FG(<\alpha)$. Let $f\in \sC_\alpha$ and $\partit$ a $c$-partition for $f$.

If $(g,y)$ is a $\partit$-lump of rank $\beta\leq\alpha$, then there exists a finer $c$-partition $\partit'$ such that $(g,y)$ is not a $\partit'$-lump,
 $\partit\setminus\partit_{(g,y)}\subseteq\partit'$, and every $\partit'$-lump is either a $\partit$-lump or has rank $<\beta$.
\end{lemma}
\begin{proof}
Suppose that $(g,y)$ is a $\partit$-lump of rank $\beta$, namely that $h=f_{(g,y)}$ is not $\omegacentered{}$-regular at $y$.
Note that $\CB(h)=\CB(g)=\beta$ by \cref{CBrankofclopenunion}.
There exists $w\in \omegacentered{\beta}$ such that $J_w=\{j\in\N\mid w\leq \ray{h}{j}\}$ is finite but non-empty.
Since $\omegacentered{\beta}$ is finite, there exists $J\in \N$ such that for all $w\in \omegacentered{\beta}$ either $J_w\subseteq J$ or $J_w$ is infinite. Let $U=N_{y\restr{J+1}}$ and note that $h\corestr{U}$ is $\omegacentered{}$-regular at $y$.

Now for all $P\in \partit_{(g,y)}$, let $P'=P\cap f^{-1}(U)$ and $A_P=P\setminus P'$. Note that while $f\restr{P'}\equiv g$ by centeredness, we have $\CB(f\restr{A_P})<\CB(g)=\beta\leq\alpha$ by \cref{CenteredasPgluing}.
So by $\FG(<\alpha)$ we can use \cref{FGconsequences}~\cref{disjointUnionCenteredFromFG} to get $\partit_P$ a $c$-partition of $f\restr{A_P}$ with functions of $\CB$-rank $<\beta$.
Our new $c$-partition for $f$ is given by 
\[
\partit'=(\partit\setminus \partit_{g,y})\cup \{P' \mid P\in \partit_{g,y}\} \cup \bigcup \{\partit_P \mid P\in \partit_{g,y}\}\,.
\]
Note that $\partit'_{g,y}=\{P'\mid P\in \partit_{g,y}\}$ and $\bigcup \partit'_{g,y}=f^{-1}(U)\cap \bigcup \partit_{g,y}$, so $(g,y)$ is not a $\partit'$-lump. 
Moreover for every $Q\in \partit'\setminus \partit$ if $Q\notin \partit'_{g,y}$ then $\CB(f\restr{Q})<\CB(g)=\beta$ holds, which proves the statement.
\end{proof}

\subsection{Gobbling up small functions}

Writing $\alpha=\lambda+n$ with $\lambda$ limit and $n$ finite, note that a $c$-partition of $f\in\sC_\alpha$ may contain a clopen $P$ such that $\CB(f\restr{P})<\lambda$.
When $\alpha$ is a double successor, that is when $n>1$, the parts of $\CB$-rank bigger than $\lambda+1$ can gobble up those of rank $<\lambda$.

\begin{lemma}\label{lem:gobblingLessThanLambda}
Let $\lambda<\omega_1$ be limit and $f\in\sC$.
Assume that $f=f_0\sqcup f_1$ with $f_0$ centered, $\pgl \Maximalfct{\lambda}\leq f_0$ and $f_1\leq \Maximalfct{\lambda}$. Then any center for $f_0$ is a center for $f$, so $f$ is centered and $f\equiv f_0$.
\end{lemma}
\begin{proof}
Let $x$ be a center for $f_0$ and let $U\ni x$ be clopen in $\dom (f_0)$. Setting $f_U=f\restr{U}$, we have $f_{U}\leq f_0$ and we show that $f\leq f_{U}$.

Since $x$ is a center for $f_0$, we have $f_0\leq f_{U}$ and in particular $\pgl \Maximalfct{\lambda}\leq f_{U}$ as witnessed by some reduction $(\sigma,\tau)$.
Let $\tilde{y}=f\sigma(\iw{0})$ and $y=f(x)$. If $\tilde{y}\neq y$, choose a clopen neighborhood $V\subseteq \im (f)$ of $\tilde{y}$ not containing $y$.
Otherwise, set $V=\tau^{-1}(N_{(1)})$, notice that $V$ is a clopen set with $y\notin V$, and such that $ \Maximalfct{\lambda}\leq f_{U}\corestr{V}$ as witnessed by restricting $(\sigma,\tau)$.
Now let $W\ni y$ be clopen and disjoint from $V$. By centeredness of $f_0$, we have $f_0\leq f_{U}\corestr{W}$. Moreover $f_1\leq\Maximalfct{\lambda}\leq f_{U}\corestr{V}$. Hence by \cref{Gluingasupperbound,Gluingaslowerbound}, we get
\[
f\leq f_0 \glbin f_1 \leq  f_{U}\corestr{W} \glbin  f_{U}\corestr{V} \leq f_{U}.\qedhere
\]
\end{proof}

\subsection{Existence of fine $c$-partitions.}\label{DefFine}
We now show that $\FG(<\alpha)$ yields the existence of fine $c$-partitions when $\alpha$ is a double successor.
To do so, we take any $c$-partition and start by dissolving all lumps of large rank before gobbling up any small parts that it may contain.

\index[IdxNotion]{finitely generated set of functions}
\begin{proposition}\label{ExistenceFinePartitions}
Let $\alpha$ be a countable ordinal and assume $\FG(<\alpha+2)$.
Then every function in $\sC_{\alpha+2}$ admits a fine $c$-partition.
\end{proposition}
\begin{proof}
Take $f\in \sC_{\alpha+2}$ and write $\alpha+2=\lambda+n+2$ with $\lambda$ limit or null and $n\in\N$. We start by defining a sequence $(\partit_i)_{i\leq n+2}$ of $c$-partitions of $f$ such that all $\partit_i$-lumps have rank $\leq\lambda+n+2-i$.
By $\FG(<\alpha+2)$, $f$ admits a  $c$-partition $\partit_0$ by \cref{FGconsequences}~\cref{disjointUnionCenteredFromFG}.

Suppose by induction that $\partit_i$ is defined for some $i\leq n+2$ and fix an enumeration $(g_j,y_j)_{j\in J}$ for $J$ an initial segment of $\N$ of all $\partit_i$-lumps of rank $\lambda+n+2-i$.
We define by induction a sequence $(\partit'_j)_{j\in J}$ of $c$-partitions, starting with $\partit'_{-1}=\partit_i$. For all $j\in J$, let $\partit'_{j+1}$ be a $c$-partition refining $\partit'_{j}$ and dissolving the $\partit_i$-lump $(g_j,y_j)$ by an application of \cref{lemma:RefiningBy1}.
The limit inferior $\partit_{i+1}=\bigcup_{k\in J} \bigcap_{k\leq j \in J} \partit'_j$ is the desired $c$-partition.
Note that all $\partit_{n+2}$-lumps have rank $\leq\lambda$. If $\lambda=0$, then $\partit_{n+2}$ has no lumps and contains no empty part, so $\partit_{n+2}$ is already fine. We therefore assume that $\lambda$ is limit. Since $\CB(f)\geq \lambda+2$ we have $\pgl \Maximalfct{\lambda}\leq f$ by \cref{ConsequencesGeneralStructureThm}~\cref{ConsequencesGeneralStructureThm2}. Therefore $\pgl \Maximalfct{\lambda}\leq f\restr{P}$ for some $P\in \partit_{n+2}$ by \cref{Centerinvariance}~\ref{Centerinvariance3}.
So letting $\tilde{\partit}=\{Q\in \partit_{n+2} \mid f\restr{Q}\leq \Maximalfct{\lambda}\}$ and $R=\bigcup \tilde{\partit}$,
then $\CB(f\restr{P\cup R})\geq \CB(\pgl \Maximalfct{\lambda})>\lambda$ and by \cref{lem:gobblingLessThanLambda}
$f\restr{P\cup R}$ is centered with cocenter $y_P$ and $f\restr{P\cup R}\equiv f\restr{P}$.
Therefore the $c$-partition $\partit$ obtained from $\partit_{n+2}$ by replacing $P$ and elements of $\tilde{\partit}$ by $\tilde{P}=P\cup R$ is the desired fine $c$-partition. \ynote*{Existence of fine partitions}{To see that $\partit$ is fine, it only remains to show that $(g,y)$ is not a $\partit$-lump where $g = f\restr{P}$ and $y=y_P$. To this end let $V=\bigcup \partit_{(g,y)}$, $h=f_{(g,y)}=f\restr{V}$ and fix some $w\in \omegacentered{\CB(g)}$. If $w$ reduces to some ray $\ray{h}{i}$ of $h$ at $y$, then either $w\leq \ray{h}{i}\restr{\tilde{P}}$ or $w\leq \ray{h}{i}\restr{V\setminus \tilde{P}}$ by \cref{Intertwinereductionsforomegacentered}. In the latter case, $w$ reduces to infinitely many rays $\ray{h}{i}$ because $(g,y)$ is not a $\partit_{n+2}$-lump. In the former case, $w$ reduces to $\ray{h}{i}\restr{\tilde{P}}$ for infinitely many $i\in \N$ by \cref{centeredImpliesWregular}.}  
\end{proof}

\section{Pseudo-centered functions}\label{subsec:pseudocentered}

Assuming $\FG(<\alpha+2)$ any function in $\sC_{\alpha+2}$ admits a fine $c$-partition $\partit$ by \cref{ExistenceFinePartitions}.
Hence every $P\in \partit$ is associated with both a centered function in $\centered{\alpha+2}$ and a point $y_P$ in the codomain of $f$.
The difficulty in proving that $f$ is equivalent to a finite gluing of generators in $\generator{\alpha+2}$
resides in the potentially complex combinatorial and topological structure of the combination of the mappings $\partit\to \centered{\alpha+2}$ and $\partit\to \im f$.
Before tackling the general case, we deal with ubiquitous situations that are easier to handle.
The first one consists of a disjoint union of the same centered function with only one \cocenter{},
in other words (using the notations of \cref{DoubleSucc_domain}) a function $f=f_{(g,y)}$ for some centered $g$ and $y\in\im(f)$.

\index[IdxNotion]{pseudo-centered function}
\begin{definition}
We say that $f\in\sC$ together with a fine $c$-partition $\partit$ is \emph{pseudo-centered} at $y$ if $Y_\partit$ is a singleton $\set{y}$ and for all $P,P'\in \partit$ we have $f\restr{P}\equiv f\restr{P'}$.
\end{definition}

The following result generalizes \cref{simplefunctionslambda+1damuddafuckaz} to functions whose $\CB$-rank is a double successor.
The strategy of proof is an elaboration on the successor of limit case.

\index[IdxNotion]{Precise Structure!pseudo-centered functions}
\index[IdxNotion]{Vertical Theorem}
\index[IdxNotion]{generators}
\index[IdxNotion]{finitely generated set of functions}
\ynote*{Vertical Theorem}{
\begin{theorem}[Vertical theorem]\label{VerticalTheorem}
Let $\alpha<\omega_1$ and assume $\FG(\leq\alpha+1)$.
Let $f:A\to B$ in $\sC_{\alpha+2}$ be pseudo-centered at $y$.
There exist $g\in\centered{\alpha+2}$, a finite $H\subseteq \omegacentered{\alpha+1}$ and a clopen partition $A=A^0\sqcup A^1$ such that 
\begin{enumerate}
\item $f\restr{A^0}\leq \gl H$ and $\gl H\leq g$,
\item $f\restr{A^1}\equiv g$, the cocenter of $f\restr{A^1}$ is $y$, and for every clopen
      $V\ni y$, $$\ g\leq f\restr{A^1}\corestr{V}\leq f\corestr{V},$$
\item for every clopen neighbourhood $U$ of $y$ there is a clopen $W\subseteq U$ with
      $y\notin W$ and $\gl H\leq f\corestr{W}$.
\end{enumerate}
In particular $f\leq g\glbin g$.
\end{theorem}}
\begin{proof}
Let $\alpha=\lambda+n$ with $\lambda$ limit or null and $n\in \N$ and fix a fine $c$-partition $\partit$ such that $Y_\partit=\set{y}$ with $f\restr{P}\equiv f\restr{P'}$ for all $P,P'\in \partit$.
In this proof, all rays are taken with respect to $y$, so that $\ray{f}{j}$ always denotes the rays of $f$ at $y$. Note that $\CB(\ray{f}{j})\leq\alpha+1$ for all $j\in\N$ by an application of \cref{CenteredasPgluing} and \cref{CBrankofclopenunion}. Choose any $P\in\partit$ and note that $\alpha+2=\CB(f)=\CB(f\restr{P})>\lambda+1$, so $f\restr{P}>\Minimalfct{\lambda+1}$ and $f\not\equiv\Minimalfct{\lambda+1}$. By $\FG(\leq\alpha+1)$, we can use \cref{CenteredAreGenerated} to see that $f\restr{P}\equiv \pgl G$ for some non-empty set $G\subseteq \centered{\alpha+1}\cup \omega\set{\centered{\alpha+1}}$ and set $g:=\pgl G\in\centered{\alpha+2}$. 
Note that for all clopen set $V\ni y$, $g \leq f\corestr{V}$, because for any $P\in \partit$ we have $g \leq f\restr{P} \leq f\restr{P}\corestr{V}$.

If for all $j\in\N$ we have $\ray{f}{j}\leq\FinGl{G}$, then $f\leq\pgl_j\ray{f}{j}\leq \pgl G = g$ by \cref{Pgluingasupperbound}.
Observe that this happens automatically if $\partit$ is finite.
In this case, we can set $H=\emptyset$, $A^0=\emptyset$ and $A^1=A$. In particular, $f$ is centered and the \cocenter{} of $f$ is $y$ by \cref{Centerinvariance}~\cref{Centerinvariance2}

\smallskip

Otherwise, $\partit$ must be infinite and we consider the set 
\[
H=\set{h\in \omegacentered{\alpha+1}}[h\nleq\FinGl{G} \text{ and } \exists j\in\N \; h \leq \ray{f}{j}]
\]
We let $w=\gl H$ and we claim that for all clopen $U\ni y$ we have $w\leq f\corestr{W}$ for some clopen $W$ with $y\notin W\subseteq U$.

To see this, let $h\in H$ and note that $H\subseteq \omegacentered{\alpha+2}$ by \cref{BasicsOnGenerators}~\cref{monotonocityGenerators}. Since $(g,y)$ is not a $\partit$-lump, we must have $h\leq\ray{f}{j}$ for infinitely many $j\in\N$.
Suppose that $U\ni y$ is clopen, and therefore contains all but finitely many rays $\ray{B}{j}$.
Note that $H$ is a subset of the finite set $\omegacentered{\alpha+1}$
and choose an injective map $J:H\to \N$ such that
$h\leq \ray{f}{J(h)}=f\corestr{\ray{B}{J(h)}}$ and $\ray{B}{J(h)}\subseteq U$ for all $h\in H$. Then $W=\bigcup_{h\in H}\ray{B}{J(h)}$ is clopen as a finite union of clopen sets, $y\notin W$, $W\subseteq U$ and $w \leq f\corestr{W}$ by \cref{Gluingaslowerbound}.
\smallskip

Next we show that $w\leq \pgl G$.  Note that $\ray{f}{j}=\bigsqcup_{P\in \partit} \ray{f\restr{P}}{j}$ and $\ray{f\restr{P}}{j} \leq \FinGl G$ by \cref{Rigidityofthecocenter} and so $\ray{f}{j}\leq \omega G$ for all $j$.
Hence by definition of $H$, {$h\leq \omega G$ for $h\in H$ and thus $w\leq \omega G$}. Since $\omega G\leq \pgl G$ by \cref{GluinglowerthanPgluing}, it follows that  $w\leq \pgl G$.

It remains to find a clopen partition $A=A^0\sqcup A^1$ such that $f\restr{A^0}\leq w$ and $f\restr{A^1}\leq g$.
Fix an enumeration $(A_i)_{i\in\N}$ of $\partit$, let $f_i=f\restr{A_i}$, and for all $j\in\N$ set $f^{[j]}=\bigsqcup_{i>j} \ray{f_i}{j}$. Note that $\CB(f^{[j]})\leq \CB(\ray{f}{j})\leq \alpha+1$.  We prove the following:

\begin{figure}
\centering
\index[IdxFig]{Proof of the Vertical Theorem}
\input{verticalthm_diagonalsplit_small.tex}
\caption{Proof of the Vertical Theorem.}
\end{figure}
\begin{claim}
For all $j\in \N$, there exists a clopen partition $(A^0_j, A^1_j)$ of $\dom(f^{[j]})$
such that:  $f^{[j]}\restr{A^0_j}\leq w$ and $f^{[j]}\restr{A^1_j}\leq \FinGl{G}$. 
\end{claim}
\begin{subproof}
Let $j\in \N$, we show that $f^{[j]}\leq w \glbin c$ for some $c\in\FinGl{G}$.
By \cref{Gluingasupperbound}, this grants us with the desired clopen partition $(A^0_j, A^1_j)$ of $\dom(f^{[j]})$. 

Since $\CB(f^{[j]})\leq \alpha+1$, either $\CB(f^{[j]})< \lambda$ or by $\FG(\leq\alpha+1)$ it is equivalent to a finite gluing of functions in $\generator{\alpha+1}$.
In the former case $f^{[j]}\leq G$ by \cref{JSLgeneralstructure} and we are done.
In the latter case, we can write $f^{[j]}\equiv\gl_{i\leq k}b_i$ for some $k\in\N$ and $b_i\in\generator{\alpha+1}$.
So in order to prove the claim, as $\generator{\alpha+1}$ is finite, it suffices to show that for every generator $b\in \generator{\alpha+1}$ if $b\leq f^{[j]}$, then $b \leq w \glbin c$ for some $c\in\FinGl{G}$.
So let $b\in \generator{\alpha+1}$ with $b\leq f^{[j]}$, we distinguish three cases according to the definition of $ \generator{\alpha+1}$:
\begin{enumerate}
\item Suppose that $b\in \centered{\alpha+1}$, so $b$ is centered. Since $b\leq f^{[j]}=\bigsqcup_{i>j} \ray{f_i}{j}$ 
and $\ray{f_i}{j}\leq \FinGl G$ for all $i,j\in\N$, we obtain $b\leq G$ by \cref{Rigidityofthecocenter,Centerinvariance}.

\item Suppose that $b\in\omegacentered{\alpha+1}$. By definition of $H$, we have either $b\leq\FinGl{G}$ or, since $b\leq f^{[j]}\leq \ray{f}{j}$, $b\in H$ and so $b\leq w$.

\item Finally\footnote{Here we see the specificity of double successors!},
it can be that $b$ is a wedge generator, so use 	\cref{TransportAndBounding}~\cref{BasicsOnGenerators_boundingwedge} to get $F_0,F_1\subseteq\centered{\alpha+1}$  such that $\{\omega F_1\}\cup F_0\leq b$ and $b\leq(\gl F_0)\glbin \omega F_1$.

Since $F_0\leq f^{[j]}$, we have $F_0\leq G$ by the first case and so $\gl F_0$ reduces to $\FinGl{G}$. 
Since $\omega F_1\leq f^{[j]}$, for every $f'\in F_1$, note that $\omega f'\in \omegacentered{\alpha+1}$, so either $\omega f' \leq\FinGl{G}$ or $\omega f'\in H$ (hence $\omega f'\leq w$) by definition of $H$. 
Altogether, we obtain the existence of an element $c$ of $\FinGl{G}$ such that
\[
b\leq (\gl F_0)\glbin \omega F_1 \leq w \glbin c, 
\]
as desired.\qedhere
\end{enumerate}
\end{subproof}

We let $A^0=\bigcup_{j\in\N} A^0_j$ and $A^1=A\setminus A^0$;
$A^0$ is open as union of the sets $A^0_j$ which are open, we prove $A^1$ is also open. 
Whenever $j\geq i$, we have $A_i\cap \dom(f^{[j]})=\emptyset$ and so for all $i\in \N$ the set $A_i\cap A^0=A_i \cap \bigcup_{j<i}A^0_j$ is clopen, hence so is $A_i\setminus A^0$.
Therefore $A^1=A\setminus A^0=\bigcup_i A_i\setminus A^0$ is open.

On the one hand, using \cref{lem:omegagenerators}:
\[
f\restr{A^0}\leq \gl_{j} f\restr{A^0_j} \leq \omega w \equiv w.
\]

On the other hand, for all $j\in\N$ we have by \cref{Gluingasupperbound}
\[
\ray{f\restr{A^1}}{j}=f\restr{A^1_j} \sqcup  \bigsqcup_{i\leq j}\ray{f_i}{j}
\leq f\restr{A^1_j}\glbin \gl_{i\leq j} \ray{f_i}{j} \leq \FinGl G,
\]
where the last reduction follows from the claim and the fact that $\ray{f_i}{j}\leq \FinGl G$ for all $i,j\in\N$. Therefore $(\ray{f\restr{A^1}}{j})_j$ is piecewise reducible to $\iw{\gl G}$, so by \cref{Pgluingasupperbound} $f\restr{A^1}\leq\pgl G=g$. 
\ynote*{Proof of Vertical Theorem}{
Finally, note that a whole part of the fine partition $\partit$ (namely $A_0$) survives in $A^1$, which yields
$g\leq f\restr{A^1}$ and \cocenter{} of $f\restr{A^1}$ is $y$. Indeed, in the diagonal split $f^{[j]}=\bigsqcup_{i>j}\ray{f_i}{j}$ the part $A_0$ of index $0$ never lies in $\dom(f^{[j]})$ (the index condition is $i>j$), so
$A_0\cap A^0=A_0\cap\bigcup_j A^0_j=\emptyset$ and hence $A_0\subseteq A^1$.
Since $f\restr{A_0}\equiv g$ we get $g\leq f\restr{A^1}$ by restriction.
Together with $f\restr{A^1}\leq g$ this gives $f\restr{A^1}\equiv g$. Therefore, $f\restr{A^1}$ is centered and the \cocenter{} of $f\restr{A^1}$ must coincide with that $f\restr{A_0}$, namely $y$, by \cref{Centerinvariance}~\cref{Centerinvariance2}. In particular, we get $g\leq f\restr{A^1}\corestr{V}\leq f\corestr{V}$ for all clopen $V\ni y$, which concludes the proof.}
\end{proof}


\begin{lemma}\label{BlockIsPseudoCentered}
Let $f\in\sC_{\alpha+2}$ with a fine $c$-partition $\partit$, let $g$ be centered and $y\in Y_\partit$
with $\partit_{(g,y)}\neq\emptyset$. Then $f_{(g,y)}$ together with $\partit_{(g,y)}$ is
pseudo-centered at $y$.
\end{lemma}
\begin{proof}
Let $\alpha<\omega_1$ with limit part $\lambda$. By definition, $\partit_{(g,y)}$ is a $c$-partition of $f_{(g,y)}$ whose parts are all equivalent to $g$ and
with \cocenter{} $y$, so $Y_{\partit_{(g,y)}}=\set{y}$, and $\CB(f\restr{Q})>\lambda$ for every
$Q\in\partit_{(g,y)}$ by fineness of $\partit$. It only remains to see that $\partit_{(g,y)}$ has no
lump. A $\partit_{(g,y)}$-lump is of the form $(g',y)$ with $g'\equiv g$, since all parts are
equivalent to $g$ and have \cocenter{} $y$; the corresponding block is then $f_{(g,y)}$ itself,
which is $\omegacentered{}$-regular at $y$ because $(g,y)$ is not a $\partit$-lump.
\end{proof}

\begin{proposition}\label{BlockReducesToTwoCopies}
Let $\alpha<\omega_1$ with limit part $\lambda$ and assume $\FG(\leq\alpha+1)$.
Let $f:A \to B$ in $\sC_{\alpha+2}$ with a fine $c$-partition $\partit$ and $y\in Y_\partit$. Then for every centered function $g$ the following hold:
\begin{enumerate}
\item $f_{(g,y)}\leq 2g$,
\item there exists $h$ 
(which depends only on $g$) with $\omega h\leq g$ such that for all clopen set $V\ni y$, 
 $f_{(g,y)}\corestr{B\setminus V}\leq \omega h$.
\end{enumerate}
\end{proposition}

\begin{proof}
If $\partit_{(g,y)}=\emptyset$, then $f_{(g,y)}$ is the empty function and both claims hold
trivially, the second one with $h=\emptyset$. So assume
$\partit_{(g,y)}\neq\emptyset$ and pick $P\in\partit_{(g,y)}$. By fineness $\CB(g)=\CB(f\restr{P})>\lambda$, and $\CB(g)$ is a
successor since $g$ is centered (\cref{CenteredasPgluing}), so $\CB(g)\geq\lambda+1$.
Moreover $\CB(g)\leq\CB(f)=\alpha+2$, so writing $\alpha=\lambda+n$ we get $\CB(g)=\lambda+k$ for
some $1\leq k\leq n+2$. Hence either $\CB(g)=\lambda+1$ or $\CB(g)$ is a double successor, and we
distinguish these two cases.

\emph{Case $\CB(g)$ is a double successor.} By \cref{BlockIsPseudoCentered} $f_{(g,y)}$ together with $\partit_{(g,y)}$ is
pseudo-centered at $y$. Since $\CB(g)$ is a double successor with $\CB(g)\leq\alpha+2$, the
hypothesis $\FG(\leq\alpha+1)$ ensures that the Vertical
\cref{VerticalTheorem} applies at $\CB(g)$ and in particular $f_{(g,y)}\leq 2g$.

\emph{Case $\CB(g)=\lambda+1$.} Then $f_{(g,y)}$ is a disjoint union of centered functions of rank
$\lambda+1$ whose \cocenter{}s are all equal to $y$, so it is simple of rank $\lambda+1$ with
distinguished point $y$ by \cref{Simpleiffcoincidenceofcocenters}.

If $\lambda=0$, then $\CB(f_{(g,y)})=1$ and simplicity means that $f_{(g,y)}$ is constant with
value $y$; it therefore reduces to any of its parts, so $f_{(g,y)}\leq g$. The second claim trivially holds with $h=\emptyset$. 

If $\lambda$ is limit, then $\FG(\leq\alpha+1)$ implies that $\leq$ is \bqo{} on $\sC_{<\lambda}$
by \cref{FGconsequences}, so by \cref{simplefunctionslambda+1damuddafuckaz} the function
$f_{(g,y)}$ is equivalent to one of $\Minimalfct{\lambda+1}$,
$\Minimalfct{\lambda+1}\glbin\Maximalfct{\lambda}$ or $\pgl\Maximalfct{\lambda}$. We have
$\Minimalfct{\lambda+1}\leq g$ by \cref{Minfunctions} and $\Maximalfct{\lambda}\leq g$ by
\cref{JSLgeneralstructure}, which settles the first two cases. In the last case,
$\pgl\Maximalfct{\lambda}$ is centered and reduces to $f_{(g,y)}$, whose domain is covered by the
open parts $Q\in\partit_{(g,y)}$; so $\pgl\Maximalfct{\lambda}\leq f\restr{Q}\equiv g$ for some
$Q\in\partit_{(g,y)}$ by \cref{Centerinvariance}~\cref{Centerinvariance3}, and therefore
$f_{(g,y)}\leq\pgl\Maximalfct{\lambda}\leq g$.

For the second part, note first that by \cref{CenteredAreGenerated} either $g\equiv \Minimalfct{\lambda+1}$ or $g\equiv\pgl G$ for some finite non-empty set
$G\subseteq\centered{\alpha+1}\cup\omega\set{\centered{\alpha+1}}$ and we distinguish these two cases.

If $g\equiv\Minimalfct{\lambda+1}$, then $y$ is the distinguished point of $f\restr{P}$ for all $P\in \partit_{(g,y)}$. For every clopen $V\ni y$, we therefore have $\CB\bigl((f\restr{P})\corestr{B\setminus V}\bigr)\leq\lambda$, and so
$(f\restr{P})\corestr{B\setminus V}\leq\Maximalfct{\lambda}$ by \cref{JSLgeneralstructure}. Therefore $f_{(g,y)}\corestr{B\setminus V}\leq \bigsqcup_{P\in \partit_{(g,y)}} (f\restr{P})\corestr{B\setminus V}\leq \omega \Maximalfct{\lambda}\equiv \Maximalfct{\lambda}  \leq \Minimalfct{\lambda+1}$ by \cref{Gluingasupperbound,lem:omegagenerators,JSLgeneralstructure}. So we can take $h=\Maximalfct{\lambda}$ in this case.

If $g\equiv\pgl G$, then
$(f\restr{P})\corestr{B\setminus V}\leq\FinGl{G}$ by \cref{ResidualCorestrictionOfCentered} for every $P\in \partit_{(g,y)}$ since $B\setminus V$ is a clopen set that excludes the \cocenter{} $y$ of $f\restr{P}$; therefore
$f_{(g,y)}\corestr{B\setminus V}\leq\omega(\gl G)\leq\pgl G\equiv g$ by
\cref{Gluingasupperbound,GluinglowerthanPgluing}. So $h=\gl G$ is as desired, which concludes the proof of the second part.
\end{proof}

\section{Strongly solvable functions}
For a fine $c$-partition of a function $f$, we denote by $\partitn{y}=\set{P\in\partit}[y_P\neq y]$\index[IdxSymb]{P@{$\partitn{y}$}} the members of the partition whose \cocenter{}s are distinct from $y$.
Now that we dealt with the case of a unique \cocenter{} $y$, we next analyze the situation where
cocenters $\{y_P\mid P\in\partitn{y}\}$ converge to $y$ and the $\set{f\restr{P}}[P\in\partitn{y}]$ also enjoys a nice combinatorial property.
\begin{definition}
We say that the function $f:A\to B$ together with a fine $c$-partition $\partit$ is \emph{strongly solvable} at $y\in Y_\partit$ if for all clopen neighborhood $V$ of $y$\index[IdxNotion]{strongly solvable function}\index[IdxNotion]{function!strongly solvable}
\begin{enumerate}
\item the set $\{y_P\mid P\in\partit \text{ and } y_P\notin V\}$ is finite, and
\item for all $P\in\partitn{y}$ there exists $Q\in \partitn{y}$ with $y_{Q}\in V$ and $f\restr{P}\leq f\restr{Q}$.
\end{enumerate}
 \end{definition}

Throughout this chapter we repeatedly need to name the equivalence classes of a set of centered
functions --- typically the restrictions $f\restr{P}$ to the parts of a fine $c$-partition --- by
functions taken from the finite set $\centered{\alpha+2}$. This is the purpose of the following
terminology.

\begin{definition}\label{def:representatives}
Let $\lambda$ be the limit part of a countable ordinal $\alpha$ and $\mathcal S\subseteq \sC_{[\lambda,\alpha+2]}$ be a set of centered functions.
A \emph{set of representatives} for $\mathcal{S}$ is a finite set $R\subseteq\centered{\alpha+2}$ of
pairwise nonequivalent functions such that
\index[IdxNotion]{set of representatives}\index[IdxNotion]{partition!set of representatives}
\begin{enumerate}
\item for every $f\in\mathcal{S}$ there is a (necessarily unique) $r\in R$ with $f\equiv r$, and
\item every $r\in R$ satisfies $r\equiv f$ for some $f\in \mathcal{S}$.
\end{enumerate}
\end{definition}

Note that $\FG(\leq\alpha+1)$ ensures the existence of a set of representatives for every set $\mathcal S\subseteq \sC_{[\lambda,\alpha+2]}$ of centered functions: by \cref{CenteredAreGenerated} every $f\in\mathcal{S}$ is equivalent to some element of $\centered{\alpha+2}$, and picking one element of $\centered{\alpha+2}$ in each of the finitely many classes so obtained yields such a set $R$, which is finite since $\centered{\alpha+2}$ is. Observe also that $\CB(r)=\CB(f)$ whenever $f\in\mathcal{S}$, $r\in R$ and $f\equiv r$. We apply this to sets of the form $\set{f\restr{P}}[P\in\mathcal{Q}]$, where $\mathcal{Q}$ is a subfamily of a fine $c$-partition $\partit$ of some $f\in\sC_{\alpha+2}$: each $f\restr{P}$ is then centered and satisfies $\lambda<\CB(f\restr{P})\leq\alpha+2$ by fineness, as required.


\index[IdxNotion]{Precise Structure!strongly solvable functions}
\index[IdxNotion]{Diagonal Theorem}
\index[IdxNotion]{finitely generated set of functions}
\index[IdxNotion]{generators}
\begin{theorem}[Diagonal Theorem]\label{DiagonalTheorem}
Assume $\FG(\leq\alpha+1)$ for $\alpha<\omega_1$.
Let $f:A \to B$ in $\sC_{\alpha+2}$ be strongly solvable {at $y$}.
Then there exists $g\in \FinGl{\generator{\alpha+2}}$ such that $f\leq g\leq f\corestr{U}$ {for all clopen set $U\ni y$}. 
\end{theorem}

\begin{proof}
Let $\partit$ be a fine $c$-partition for $f$ and $\bar{y}\in Y_\partit$ witness that $f$ is strongly solvable {at $\bar{y}$}.
Let $\partit^{\bar{y}}=\partit\setminus \partitn{\bar{y}}$.
Note that if $\partitn{\bar{y}}$ is non-empty, then $Y^{\neq\bar{y}}=\{y_P\mid P\in \partit\}\setminus\{\bar{y}\}$ is an infinite discrete subset of $B$ by strong solvability of $f$. 
We write $\alpha=\lambda+n$ with $\lambda$ limit or null and $n\in\N$ and distinguish two cases.

\medskip

\textit{First case}: $\CB(f\restr{P})=\lambda+1$ for all $P\in \partit^{\bar{y}}$. Since $\CB(f)=\alpha+2$, we can pick $Q\in\partitn{\bar{y}}$ with $\CB(f\restr{Q})=\alpha+2$ by \cref{CBrankofclopenunion}, the supremum of the ranks of the parts being attained since $\alpha+2$ is a successor.
Let $D$ be a set of representatives for $\set{f\restr{P}}[P\in\partitn{\bar{y}}]$ and let $h\in D$ be the representative of $f\restr{Q}$, so that $\CB(h)=\alpha+2$.
We show that in that case we can set $g:=\omega D$; note that $\omega D=\omega(\gl D)\equiv\gl_{d\in D}\omega d$ and $\omega\set{\centered{\alpha+2}}\subseteq\omegacentered{\alpha+2}$, so that $g \in\FinGl{\omegacentered{\alpha+2}}\subseteq\FinGl{\generator{\alpha+2}}$.
By \cref{Maxfunctions,ConsequencesGeneralStructureThm} we have $f\restr{P}\leq\pgl\Maximalfct{\lambda}\leq h$ for all $P\in\partit^{\bar{y}}$,
so using \cref{Gluingasupperbound,Gluingasupperbound_cor} we get:
\[
f \leq \gl_{P\in \partit} f\restr{P}\leq \gl_{P\in\partitn{\bar{y}}} f\restr{P}\glbin\gl_{P\in \partit^{\bar{y}}} f\restr{P}\leq \omega D\glbin \omega \pgl\Maximalfct{\lambda}\leq \omega D=g.
\]
Now fix any clopen set $U\ni \bar{y}$. Since $f$ is strongly solvable at $\bar{y}$, for all $d\in D$ the set $X_d=\set{y_P}[P\in\partitn{\bar{y}} \text{ and }y_P\in U \text{ and }d\leq f\restr{P}]$ is infinite: given a finite subset of $X_d$, we separate $\bar{y}$ from it by a clopen $V\ni\bar{y}$ and apply the second clause of strong solvability to $V\cap U$ and to a part $P_0\in\partitn{\bar{y}}$ realizing $d$, which produces a new element of $X_d$.
Moreover, if $y_P\in X_d$ then $d\leq f\restr{P}\leq (f\restr{P})\corestr{V}\leq (f\corestr{U})\corestr{V}$ for every neighbourhood $V\ni y_P$, since $f\restr{P}$ is centered with \cocenter{} $y_P\in U$. So by \cref{Intertwinereductions} applied to $f\corestr{U}$ we have $\omega D\leq f\corestr{U}$, which concludes the proof in this case.

\medskip

\textit{Second case}: $\CB(f\restr{P})> \lambda+1$ for some $P\in \partit^{\bar{y}}$. 
Let $\partit_M=\{P\in \partit^{\bar{y}}\mid \CB(f\restr{P})> \lambda+1\}$ and let $M$ be a set of representatives for $\set{f\restr{P}}[P\in \partit_M]$.
For each $g\in M$, as $\CB(g)>\lambda+1$, there exists a non-empty $M_g\subseteq  \centered{\alpha+1}\cup \omega\{\centered{\alpha+1}\}$ such that $g\equiv\pgl M_g$. Let $A_g=\bigcup \partit_{(g,\bar{y})}$
; note that the sets $A_g$, $g\in M$, are pairwise disjoint since the elements of $M$ are pairwise nonequivalent, they cover $\bigcup \partit_M$ and $f\restr{A_g}=f_{(g,\bar y)}$.

For all $g\in M$ the function $f\restr{A_g}=f_{(g,\bar y)}$ together with $\partit_{(g,\bar{y})}$ is pseudo-centered at $\bar{y}$ by \cref{BlockIsPseudoCentered}, so we can apply the Vertical \cref{VerticalTheorem} at $\CB(g)\in [\lambda+2, \alpha+2]$: there exists $H_g\subseteq \omegacentered{\alpha+1}\subseteq\omegacentered{\alpha+2}\subseteq\generator{\alpha+2}$ and a clopen partition $A_g=A^0_g\sqcup A^1_g$ such that:
\begin{enumerate}
\item $f\restr{A^0_g}\leq \gl H_g \leq g $,
\item $f\restr{A^1_g}\equiv g$, the \cocenter{} of $f\restr{A^1_g}$ is $\bar{y}$, and $g \leq f\restr{A^1_g}\corestr{V}$ for all clopen $V\ni \bar{y}$,
\item for all clopen neighborhood $\tilde{U}\ni \bar{y}$ there exists a clopen $W\subseteq \tilde{U}$ with $\bar{y}\notin W$ and $\gl H_g \leq f\corestr{W}$.
\end{enumerate}
Set $w=\gl_{g\in M} (\gl H_g)$. If $\partitn{\bar{y}}$ is empty, set $D=\emptyset$. Otherwise, as in the previous case, let $D$ be a set of representatives for $\set{f\restr{P}}[P\in \partitn{\bar{y}}]$. Set $\mathcal{M}=\set{M_g}[g\in M]$, note that $M_g\subseteq\centered{\alpha+1}\cup\omega\set{\centered{\alpha+1}}\subseteq\generator{\alpha+1}$, and $D\subseteq\centered{\alpha+2}$, so $\fctwedge(\mathcal M\mid D)$ is a wedge generator of $\generator{\alpha+2}$.  We claim that for all clopen set $U\ni \bar{y}$ we have
\[
f\leq w \glbin \fctwedge(\mathcal M\mid D) \leq f\corestr{U},
\]
which concludes the proof, once we prove the existence of the two reductions.

\begin{figure}
\centering
\index[IdxFig]{Proof of the Diagonal Theorem}
\definecolor{bgreen}{RGB}{126,211,33} 
\definecolor{bblue}{RGB}{153,204,255}
\begin{tikzpicture}[scale=0.58,
  every node/.style={font=\small},
  vert/.style={draw,thick,fill=green!10},
  diag/.style={draw,thick,fill=blue!8},
  resb/.style={draw,thick,dashed,fill=black!10},
  wbox/.style={draw,thick,fill=orange!15},
  rowline/.style={gray!60,dotted}]


\draw[wbox] (-2.0,0) rectangle (-.3,2.15);
\node at (-1.15,0.57) {$w$};
\node at (-1.15,-0.75) {$A^0$};

\foreach \i/\lab in {0/{g_1}, 1/{g_2}, 2/{g_3}}{
  \foreach \yb/\hh in {0/1.15, 1.35/0.9, 2.45/0.68, 3.33/0.5}
    \draw[vert] ({1.9*\i},\yb) rectangle ({1.9*\i+1.7},\yb+\hh);
  \draw[vert] ({1.9*\i},4.03) rectangle ({1.9*\i+1.7},4.39);
  \node[rotate=90,font=\footnotesize] at ({1.9*\i+0.85},4.21) {$\cdots$};
  \filldraw[bgreen, thin] ({1.9*\i},4.11) arc[start angle=-90, end angle=0, radius=0.28]
    -- ({1.9*\i},4.39) -- cycle;
  \node at ({1.9*\i+0.85},-0.75) {$A^1_{\lab}$};}
\node[black!95] at (2.85,5.2) {verticals};

\foreach \i/\n/\yb/\hh in {3/0/0/1.15, 4/1/1.35/0.9, 5/2/2.45/0.68}{
  \draw[resb] ({1.9*\i},0) rectangle ({1.9*\i+1.7},3.83);
  \draw[diag] ({1.9*\i},\yb) rectangle ({1.9*\i+1.7},\yb+\hh);
  \node[font=\footnotesize] at ({1.9*\i+0.85},\yb+\hh/2) {$A^D_{\n}$};
  \node at ({1.9*\i+0.85},-0.75) {$A_{Y_{\n}}$};}
\node at (11.6,3.5) {\rotatebox{80}{$\ddots$}};
\node[black!95] at (9.45,5.2) {diagonal};
\node[anchor=north] at (10.3,1.) {$A^D_{\mathrm{res}}$};
\begin{scope}[shift={(15.2,4.75)},xscale=0.031,yscale=-0.0460,shift={(-413.26,-118.89)}]
\draw[thick] (452.13,174.92) .. controls (473.39,186.41) and (473.58,205.88) .. (452.56,218.41)
  .. controls (431.54,230.95) and (397.26,231.79) .. (376,220.3)
  .. controls (354.74,208.81) and (354.55,189.33) .. (375.57,176.8)
  .. controls (400.95,161.68) and (413.51,141.7) .. (413.26,116.89)
  .. controls (413.51,141.7) and (426.46,161.04) .. (452.13,174.92) -- cycle;
\end{scope}
\begin{scope}[shift={(15.2,4.75)},xscale=0.030,yscale=-0.040,shift={(-413.26,-116.89)}]
\draw[thick, fill=bblue!30] (452.13,174.92) .. controls (473.39,186.41) and (473.58,205.88) .. (452.56,218.41)
  .. controls (431.54,230.95) and (397.26,231.79) .. (376,220.3)
  .. controls (354.74,208.81) and (354.55,189.33) .. (375.57,176.8)
  .. controls (400.95,161.68) and (413.51,141.7) .. (413.26,116.89)
  .. controls (413.51,141.7) and (426.46,161.04) .. (452.13,174.92) -- cycle;
\end{scope}

\draw (13.60,1.35) -- (16.84,1.35);
\draw (14.20,2.45) -- (16.31,2.45);
\draw (14.85,3.33) -- (15.56,3.33);
\node[font=\footnotesize] at (15.2,3.85) {$\vdots$};
\fill[bgreen,draw=black] (15.2,4.75) circle (3pt);
\node at (15.6,4.9) {$\bar{y}$};
\fill[gray,draw=black] (14.7,0.1) circle (3pt);

\node[anchor=west] at (16.95,0.65) {$\ray{B}{\bar{y},0}$};
\fill[gray,draw=black] (14.5,1.) circle (3pt);
\fill[gray,draw=black] (15,0.5) circle (3pt);

\node[anchor=west] at (16.75,1.85) {$\ray{B}{\bar{y},1}$};
\fill[gray,draw=black] (14.4,2.1) circle (3pt);
\fill[gray,draw=black] (14.8,1.6) circle (3pt);
\fill[gray,draw=black] (15.1,2.2) circle (3pt);
\node[anchor=west] at (16.0,2.85) {$\ray{B}{\bar{y},2}$};
\fill[gray,draw=black] (15.3,2.7) circle (3pt);
\node at (15.2, 0.05) {$U$};
\node at (15.2,-.75) {$B$};

\draw[wbox] (15.9,1.22) ellipse (0.45 and 0.52);
\node[font=\footnotesize] at (15.9,1.22) {$W$};

\foreach \yr/\xe in {0.57/13.4, 1.80/13.4, 2.79/14.}
  \draw[->,dashed,black] (12.4,\yr) -- (\xe,\yr);
\node[black] at (12.9,3.2) {$f$};
\end{tikzpicture}
\caption{The proof of the Diagonal Theorem, second case.}
\label{fig:diagonaltheorem}
\end{figure}

\smallskip

\textit{Right reduction}: $ w \glbin\fctwedge(\mathcal M\mid D) \leq f\corestr{U}$.
Since $M$ is finite, we can use the last clause of the Vertical Theorem iteratively to obtain \emph{pairwise disjoint} clopen sets $W_g\subseteq U$, $g\in M$, with $\bar{y}\notin W_g$ and $\gl H_g \leq f\corestr{W_g}$: having chosen $W_{g_1},\ldots,W_{g_i}$, the set $\tilde{U}=U\setminus(W_{g_1}\cup\dots\cup W_{g_i})$ is again a clopen neighborhood of $\bar{y}$, to which we apply the clause for the next element of $M$. For $W=\bigsqcup_{g\in M}W_g$, we have $\bar{y}\notin W$ and $w\leq f\corestr{W}$ by \cref{Gluingaslowerbound}.

For each $g\in M$, we choose some $P_g\in\partit_M$ such that $f\restr{P_g}\equiv g$ and choose a center $x_g$ for $f\restr{P_g}$.
Since $W$ is clopen and $\bar{y}\notin W$, we can choose a clopen $V\ni \bar{y}$ disjoint from $W$ and included in $U$.
We claim that $\fctwedge(\mathcal M\mid D) \leq f\corestr{V}$ and we prove this by verifying
the two conditions in the premise of \cref{DisjointificationLemma} for $\bar{y}$ and the points $x_g$.  

For the first condition, we need to show that for all $g\in M$, for every open $U'\ni x_g$ and all $h\in M_g$
there exists $(\sigma,\tau)$ reducing $h$ to $f$ with $\im(\sigma)\subseteq U'$ and $\overline{\im(f\sigma)}\notni \bar{y}$.
{Since $x_g\in P_g\cap U'$ is a center for $f\restr{P_g}\equiv g$, there is a reduction $(\sigma,\tau)$ from $g$ to $f\restr{P_g\cap U'}$ by \cref{Centerinvariance}~\cref{Centerinvariance1}.
But the \cocenter{} of $f\restr{P_g\cap U'}$ is $\bar{y}$ so by \cref{Rigidityofthecocenter} we obtain that the desired reduction by restricting $(\sigma,\tau)$.  
}
 
To verify the second condition, we need to show that
for every open $U'\ni \bar{y}$, there exists $(\sigma,\tau)$ reducing $\gl D$ to $f$ with $\im(f\sigma)\subseteq U'$ and $\overline{\im(f\sigma)}\notni \bar{y}$. \ynote*{Diagonal Theorem, second case, right reduction}{
Shrinking $U'$, we may assume that $U'$ is clopen. We first choose, for every $d\in D$, a part $Q_d\in\partitn{\bar{y}}$ with $d\leq f\restr{Q_d}$ and $y_{Q_d}\in U'$, in such a way that the points $y_{Q_d}$, $d\in D$, are pairwise distinct.
This is possible: given $d\in D$ and a part $P\in\partitn{\bar{y}}$ realizing $d$, we separate $\bar{y}$ from the finitely many cocenters already chosen by a clopen neighborhood $V''\ni\bar{y}$ and apply the second clause of strong solvability to $V''\cap U'$ and $P$, which yields $Q_d\in\partitn{\bar{y}}$ with $y_{Q_d}\in V''\cap U'$ and $d\leq f\restr{P}\leq f\restr{Q_d}$.
Since the points $y_{Q_d}$ are pairwise distinct and different from $\bar{y}$, we can pick pairwise disjoint clopen neighbourhoods $V'_d\subseteq U'$ of $y_{Q_d}$ with $\bar{y}\notin V'_d$.
As $f\restr{Q_d}$ is centered with \cocenter{} $y_{Q_d}$, we have $d\leq f\restr{Q_d}\leq (f\restr{Q_d})\corestr{V'_d}\leq f\corestr{V'_d}$ for every $d\in D$, so $\gl D\leq f\corestr{V'}$ by \cref{Gluingaslowerbound}, where $V'=\bigsqcup_{d\in D}V'_d$ is clopen, included in $U'$ and does not contain $\bar{y}$. The resulting reduction $(\sigma,\tau)$ therefore satisfies $\im(f\sigma)\subseteq V'\subseteq U'$ and $\overline{\im(f\sigma)}\subseteq V'\notni\bar{y}$, as desired.}

Altogether, since $W$ and $V$ are clopen subsets of $U$,  we obtain by \cref{Gluingasupperbound, Gluingaslowerbound2}:
\[
w \glbin\fctwedge(\mathcal M\mid D) \leq f\corestr{W}\glbin f\corestr{V} \leq f\corestr{U}.
\]

\smallskip
\textit{Left reduction:} $f\leq w \glbin\fctwedge(\mathcal M\mid D)$.
Let $A^0=\bigsqcup_{g\in M} A^0_g$, which is indeed a disjoint union since the sets $A_g$ are pairwise disjoint and $A^0_g\subseteq A_g$; note that $A^0$ is clopen (since $M$ is finite) and $f\restr{A^0}\leq w$ by \cref{Gluingasupperbound}.
We show that for $A^1=A\setminus A^0$ we have $f\restr{A^1}\leq\fctwedge(\mathcal M\mid D)$ by verifying the three clauses of \cref{Wedgeasupperbound} for $\bar y$.

Note that $\bigcup \partit^{\bar y}=\bigcup \partit_M \sqcup  A^1_{\mathrm{low}}$, where $A^1_{\mathrm{low}}=\bigcup \{P\in\partit^{\bar y}\mid \CB(f\restr{P})=\lambda+1\}$. If $A^1_{\mathrm{low}}$ is non empty, notice that $f\restr{A^1_{\mathrm{low}}}$ is a simple function of rank $\lambda+1$ with distinguished point $\bar{y}$ by \cref{Simpleiffcoincidenceofcocenters}. So for all $j\in\N$ we have $\CB(\ray{(f\restr{A^1_{\mathrm{low}}})}{\bar{y},j})\leq\lambda$ and by \cref{JSLgeneralstructure} we get $\ray{(f\restr{A^1_{\mathrm{low}}})}{\bar{y},j}\leq \Maximalfct{\lambda}$.
Recall that any $g\in M$ has $\CB$-rank at least $\lambda+2$, so for some $h\in M_g$ we have $\CB(h)\geq\lambda+1$
and in turn $\Maximalfct{\lambda}\leq h$ by \cref{JSLgeneralstructure} again.
This means that we can safely replace $A^1_g$ by $A^1_g \sqcup A^1_{\mathrm{low}}$ for some $g\in M$, since $(\ray{(f\restr{A^1_g})}{\bar{y},j})_j$ still reduces by pieces to $\iw{\gl M_g}$, hence $f\restr{A^1_g \sqcup A^1_{\mathrm{low}}}\equiv g$ with cocenter $\bar y$.

If $\partitn{\bar{y}}$ is empty and so $D=\emptyset$, then $(A^1_g)_{g\in M}$ partitions $A^1$ and so $f\restr{A^1}\leq \fctwedge(\mathcal M\mid \emptyset)$ by \cref{Wedgeasupperbound}.
\smallskip

Otherwise, we have $A^1=(\bigsqcup_{g\in M} A^1_g)\sqcup A^D$ --- after the absorption of
$A^1_{\mathrm{low}}$ described above --- with
$A^D=\bigcup\partitn{\bar{y}}$, so we need to partition $A^D$ into a sequence of
clopen sets that is reducible by pieces to $\iw{D}$ and whose images converge to $\bar{y}$.
\ynote*{Diagonal Theorem, second case, left reduction, $\partitn{\bar{y}}\neq\emptyset$.}
{For each $n\in\N$, let $Y_n:=\set{y\in Y^{\neq\bar{y}}}[y\in\ray{B}{\bar{y},n}]$ and note that these sets are
finite by strong solvability and that they partition $Y^{\neq\bar{y}}$, since the rays $(\ray{B}{\bar{y},n})_n$
partition $B\setminus\set{\bar{y}}$. For every $y\in Y_\partit$, let
$A_y=\bigcup\set{P\in\partit}[y_P=y]$, and for every $n\in\N$ let $A_{Y_n}=\bigcup_{y\in Y_n}A_y$; the
sets $A_{Y_n}$ are open, and together with the open set $A_{\bar{y}}$ they form a partition of the
domain of $f$, so they are clopen. Set $A^D_n:=A_{Y_n}\cap f^{-1}(\ray{B}{\bar{y},n})$ and note that
these sets are clopen and that $f(A^D_n)\subseteq\ray{B}{\bar{y},n}$, so $f(A^D_n)$ converges to
$\bar{y}$ as $n\to \infty$. There remains one block to account for, since the sets $A^D_n$ do not exhaust $A^D$: let $R^D_n= A_{Y_n}\setminus f^{-1}(\ray{B}{\bar{y},n})$ and define 
\[
A^D_{\mathrm{res}}=A^D\setminus\bigcup_{n\in\N}A^D_n=\bigcup_{n\in\N} R^D_n,
\]
which is open by the second equality and closed as the complement in the clopen $A^D$ of the open
set $\bigcup_n A^D_n$, hence clopen.

We apply \cref{BlockReducesToTwoCopies} to $f_{(d,y)}$ for $d\in D$, $n\in \N$ and $y\in Y_n$. First, we have $f_{(d,y)}\leq 2d$. Second, there exists $h_d$ 
such that $\omega h_d\leq d$ and 
$f_{(d,y)}\restr{R^D_n}=(f_{(d,y)})\corestr{B \setminus \ray{B}{\bar{y},n}}\leq \omega h_d$, since $\ray{B}{\bar y,n}$ is a clopen neighbourhood of $y$ for $y\in Y_n$. 

Now as $f\restr{A_{Y_n}}=\bigsqcup_{d\in D}\bigsqcup_{y\in Y_n} f_{(d,y)}$, it follows from two applications of  \cref{Gluingasupperbound} that $f\restr{A_{Y_n}}\leq 2\lvert Y_n\rvert\; \gl D$, hence $f\restr{A^D_n}\leq f\restr{A_{Y_n}} \leq\FinGl{D}$.
Moreover, by \cref{Gluingasupperbound} we get
\[
f\restr{A^D_{\mathrm{res}}}=\bigsqcup_{d\in D}\bigsqcup_{n\in \N}\bigsqcup_{y\in Y_n} f_{(d,y)}\restr{R^D_n}\leq \gl_{d\in D} \omega h_d \leq \gl D.
\]

Altogether, the family consisting of the columns $(A^1_g)_{g\in M}$ (with $A^1_{\mathrm{low}}$
absorbed into one of them) followed by the diagonal blocks
$A^D_{\mathrm{res}},A^D_0,A^D_1,\dots$ is a clopen
partition of $A^1$ satisfying the three hypotheses of \cref{Wedgeasupperbound} at the point
$\bar{y}$: 
\begin{description}
\item[hypothesis (1)] For all $g\in M$ we have $f\restr{A^1_g}\equiv g\equiv\pgl M_g$ and the \cocenter{} of $f\restr{A^1_g}$ is $\bar{y}$.
By \cref{Rigidityofthecocenter} this implies that the sequence of rays $(\ray{(f\restr{A^1_g})}{\bar{y},j})_j$ is reducible by finite pieces to $\iw{\gl M_g}$. 
\item[hypothesis (2)] is the
reducibility by pieces of the diagonal blocks to $\iw{\gl D}$, which holds since $f\restr{A^D_{\mathrm{res}}}\leq \gl D$ and  $f\restr{A^D_n} \leq\FinGl{D}$ as just proved above, and
\item[hypothesis (3)] holds
because $f(A^D_n)\to \bar y$ as already observed, which is sufficient. 
\end{description}
Therefore $f\restr{A^1}\leq\fctwedge(\mathcal{M}\mid D)$, where we used
\cref{cor:wedgeSets} to pass from the tuple $(\gl M_g)_{g\in M}$ to the family $\mathcal{M}$. Since
$A=A^0\sqcup A^1$ is a clopen partition of $A$ and $f\restr{A^0}\leq w$, we conclude
$f\leq w\glbin\fctwedge(\mathcal{M}\mid D)$ by \cref{Gluingasupperbound}, which is the left
reduction. 

Together with the right reduction and with the observation that
$w\glbin\fctwedge(\mathcal{M}\mid D)$ lies in $\FinGl{\generator{\alpha+2}}$, this settles the
second case and completes the proof.
}
\end{proof}

\section{Solvable functions}
We finally introduce the weaker notion of \emph{solvable functions} by dropping the first requirement in the definition of strongly solvable functions.
We show that we can reduce the general case to the case of solvable functions and then we show how solvable functions can be expressed in terms of generators.  

\begin{definition}
We say that the function $f:A\to B$ is \emph{solvable} {with} a fine $c$-partition $\partit$ if for some $y\in Y_\partit$, for all $P\in\partitn{y}$ and all clopen $V\ni y$ there exists $Q\in \partitn{y}$ with $y_{Q}\in V$ and $f\restr{P}\leq f\restr{Q}$. \index[IdxNotion]{solvable function}\index[IdxNotion]{function!solvable}

\end{definition}
Note that if $f$ is solvable with $\partit$, then either $Y_\partit$ is a singleton or it is infinite.
If $\partit$ is a fine $c$-partition for a function $f:A\to B$ and $U$ is a clopen subset of $B$,
we define the corestriction $\partit\corestr{U}=\{P\in \partit \mid y_P\in U\}$ and we let $A^U_\partit=\bigcup \partit\corestr{U}$.\index[IdxSymb]{P@{$\partit\corestr{U}$}}\index[IdxSymb]{A@{$A^U_\partit$}}
Note that $\partit\corestr{U}$ is a fine $c$-partition of $f\restr{A^U_\partit}$, because any $\partit\corestr{U}$-lump is a $\partit$-lump.

\begin{theorem}\label{SolvableDecomposition}
For $\alpha<\omega_1$ assume $\FG(<\alpha+2)$ and let $\partit$ be a fine $c$-partition of $f:A \to B$ in $\sC_{\alpha+2}$.
Then there exists a countable family $\mathcal{U}$ of pairwise disjoint clopen subsets of $B$ such that:
\begin{enumerate}
\item $Y_\partit \subseteq \bigcup \mathcal{U}$, and 
\item for all $U\in \mathcal{U}$, the function $f\restr{A^U_\partit}:A^U_\partit\to B$ is solvable with $\partit\corestr{U}$.
\end{enumerate}
\end{theorem}
\begin{proof}
We start by observing that for every $y\in Y_\partit$ and every small enough neighborhood $U\ni y$ the function $f\restr{A^U_\partit}:A^U_\partit\to B$ is solvable in $\partit\corestr{U}$. 
To see this, fix a decreasing basis of clopen neighborhoods $(U_n)_{n\in \N}$ at $y$ and define:
\[
D_n=\set{g\in \centered{\alpha+2}}[\exists P\in\partitn{y} \text{ and } y_P\in U_n \text{ and } f\restr{P}\equiv g]
\]
Note that $m<n$ implies $D_m\supseteq D_n$, therefore as $\centered{\alpha+2}$ is finite this sequence must be eventually constant. So let $M\in\N$ be such that $D_M=D_n$ for all $n\geq M$ and set $U=U_n$ for any $n\geq M$. 
Clearly $\partit\corestr{U}$ is a fine $c$-partition of $f\restr{A^U_\partit}$. Suppose that $P\in \partit\corestr{U}$ with $y_P\neq y$ and let $V$ be a clopen neighborhood of $y$. As we assume $\FG(<\alpha+2)$, there exists $g\in \centered{\alpha+2}$ with $g\equiv f\restr{P}$ by \cref{CenteredAreGenerated}, so $g\in D_M$. Moreover there exists $n>M$ with $U_n\subseteq V$, hence as $D_n=D_M$, it follows that there exists $Q\in\partit\corestr{U}$ with $y_Q\in U_n\subseteq V$ and $f\restr{Q}\equiv f\restr{P}$.

Now fix an enumeration $(y_n)_n$ of $Y_\partit$. We define by induction a sequence $(U_k)_k$ such that the sets $U_k$ are pairwise disjoint clopen sets of $B$,
they cover $Y_\partit$, and for all $k$ the restriction $f\restr{A^{U_k}_\partit}:A^{U_k}_\partit\to B$ is solvable in $\partit\corestr{U_k}$.
Let $n_0=0$ and use the observation to choose a neighborhood $U_0$ of $y_0$ such that $f\restr{A^{U_0}_\partit}:A^{U_0}_\partit\to B$ is solvable in $\partit\corestr{U_0}$. Assume that $(U_l)_{l<k}$ is defined with $U_l$ clopen pairwise disjoint and $f\restr{A^{U_l}_\partit}:A^{U_l}_\partit\to B$ is solvable in $\partit\corestr{U_l}$ for all $l<k$. If $Y_\partit \subseteq \bigcup_{l<k}U_l$ we are done, otherwise let $n_k=\min \{n\mid y_n\notin  \bigcup_{l<k}U_l\}$
and using the observation again choose a small enough neighbourhood $U_k$ of $y_{n_k}$ disjoint from $\bigcup_{l<k}U_l$ and such that $f\restr{A^{U_k}_\partit}:A^{U_k}_\partit\to B$ is solvable in $\partit\corestr{U_k}$. The family $\mathcal U=\set{U_k}$ is as required.
\end{proof}

\begin{remark}
If $f$ is strongly solvable at $y$ then we saw that $f\leq f\corestr{U}$ for every neighborhood $U$ of $y$.
In contrast, when $f$ is solvable with $\partit$ we will show that $f\leq f\corestr{U}$ whenever $U$ is a clopen set such that $Y_\partit\subseteq U$.
Note that it would be possible to strengthen the notion of solvable functions to a notion of \emph{solvable function at $y$} which further requires that
for any sequence $(P_n)_{n\in \N}$ in $\partitn{y}$ with pairwise distinct $y_{P_n}$ and every clopen set $V\ni y$
there exists a sequence $(Q_n)_{n\in \N}$ in $\partitn{y}$ with pairwise distinct $y_{Q_n}$ in $V$ such that $f\restr{P_n}\leq f\restr{Q_n}$.
This would ensure that $f\leq f\corestr{V}$ for any clopen neighborhood of $y$.
While this may seem desirable and a result analogous to \cref{SolvableDecomposition} can be established,
it seems to introduce a complexity which did not appear necessary for our proof.
\end{remark}

We now want to show the following statement for all $\alpha<\omega_1$:
\begin{descThm}
 \item[$\text{S}(\alpha)$] $\FG(\leq \alpha)$ \index[IdxSymb]{S@{$\text{S}(\alpha)$}}implies that for every $f\in \sC_{\alpha+1}$ if $f$ solvable with $\partit$, then there exists $g\in \FinGl{\generator{\alpha+1}}$ with $f\leq g$ and $g\leq f\corestr{U}$ for every clopen $U\supseteq Y_\partit$.
  \end{descThm}
In \cref{FGatdoublesuccessors}, we prove that this is enough to obtain $\FG(\alpha+2)$ from $\FG(<\alpha+2)$, thus completing the inductive proof of \cref{PreciseStructureThm}.

This latter proof notably relies on a combined application of \cref{ExistenceFinePartitions,SolvableDecomposition}
to functions in $\sC_{\alpha+2}$ (for $\alpha=\lambda+n$ with $\lambda$ limit or null) to first get a fine $c$-partition, which is then partitioned to get solvable functions. 
This process may yield some solvable functions in $\sC_{\lambda+1}$, and
while we have established $\FG(\lambda+1)$ this does not directly imply $\text{S}(\lambda)$\footnote{Nor does $\text{S}(\lambda)$ directly imply $\FG(\lambda+1)$,
which makes \cref{FGatsuccessoroflimit} still necessary.}.
So we now prove $\text{S}(\lambda)$ for $\lambda$ limit or null, before showing $\text{S}(\alpha)$ in the successor case.


\begin{proposition}\label{solvablelambda+1}
Let $\lambda<\omega_1$ be limit or null and assume $\FG(\leq \lambda)$. Suppose that $f:A\to B$ in $\sC_{\lambda+1}$ is solvable with $\partit$.

Then there exists a finite gluing $g$ of functions in $\generator{\lambda+1}$ such that $f\leq g$ and $g\leq f\corestr{U}$ for every clopen $U\supseteq Y_\partit$.
\end{proposition}

\begin{proof}
Let $\partit$ be a fine $c$-partition of $f$ and $y\in Y_\partit$ witnessing that $f$ is solvable, and let $U\supseteq Y_\partit$ be clopen in $B$.
We first deal with the case $\lambda=0$.
In this case, $f$ is locally constant, so we know that $f\equiv |\im f| \;\Minimalfct{1}$ from \cref{LocallyConstantFunctions}.
Since $f$ is solvable with $\partit$, $f\restr{P}$ is centered, hence constant, for all $P\in \partit$, and so $\im f= Y_\partit\subseteq U$. So $f=f\corestr{U}$
and the result follows with $g$ equal to $\Minimalfct{1}$ or $\omega\Minimalfct{1}$\footnote{Note that since $f$ is solvable, it actually follows that $\im f$ is either infinite or equal to $\set{y}$,
hence $f\equiv \omega \Minimalfct{1}$ or $f\equiv \Minimalfct{1}$.}.

Now assume that $\lambda$ is limit, recall that $\FG(\leq \lambda)$ ensures that $\FG(\lambda+1)$ holds by \cref{FGconsequences}, and using 
that $\partit$ is fine,
for every $P\in\partit$ we have $f\restr{P}\equiv \Minimalfct{\lambda+1}$ or $f\restr{P}\equiv  \pgl\Maximalfct{\lambda}$ by \cref{cor:CenteredSucessor}. 

\smallskip

First case:  $y_P=y$ for all $P\in \partit$, so $f$ is simple by \cref{Simpleiffcoincidenceofcocenters} and so $f\leq \pgl \Maximalfct{\lambda}$ by \cref{Maxfunctions}.
If $f\restr{P}\equiv  \pgl\Maximalfct{\lambda}$ for some $P\in \partit$, then $\pgl \Maximalfct{\lambda}\leq (f\restr{P})\corestr{V}\leq f\corestr{V}$ for every clopen neighbourhood $V$ of $y$ by centeredness of $f\restr{P}$, so we are done. Otherwise $f\restr{P}\equiv\Minimalfct{\lambda+1}$ for all $P\in \partit$. In particular, $\pgl \Maximalfct{\lambda}\nleq f$, for otherwise $\pgl \Maximalfct{\lambda}\leq f\restr{P}$ for some $P\in\partit$ by \cref{Centerinvariance}~\cref{Centerinvariance3}. As $f$ is simple of rank $\lambda+1$, \cref{simplefunctionslambda+1damuddafuckaz} therefore implies that $f\leq \Minimalfct{\lambda+1}\glbin\Maximalfct{\lambda}$. Moreover $\Minimalfct{\lambda+1}\leq f\restr{V}$ for every clopen neighborhood $V$ of $y$. We distinguish two subcases.

If $\CB(\ray{f}{y,n})<\lambda$ for all $n\in \N$, then actually $f\leq \pgl_n \ray{f}{y,n}\leq \Minimalfct{\lambda+1}$ by \cref{Pgluingofraysasupperbound,ConsequencesGeneralStructureThm}, as desired. 

Otherwise, there exists $n\in \N$ such that $\CB(\ray{f}{y,n})=\lambda$ and so $\Maximalfct{\lambda}\leq \ray{f}{y,n}$ by the General Structure \cref{JSLgeneralstructure}. As $\partit$ is a fine partition, $f_{(\Minimalfct{\lambda+1},y)}=f$ is $\omegacentered{}$-regular at $y$, 
 so $\Maximalfct{\lambda}\leq \ray{f}{y,n}$ for infinitely many $n$ and we can choose a large enough $N\in \N$ such that $\Maximalfct{\lambda}\leq  \ray{f}{y,N}$ and $\ray{B}{y, N}\subseteq U$. Therefore for $V=U \setminus \ray{B}{y,N}$ we have  $\Minimalfct{\lambda+1}\glbin\Maximalfct{\lambda}\leq f\corestr{V} \glbin f\corestr{\ray{B}{y,N}}\leq f\corestr {U}$, which concludes the proof of the first case\footnote{Such a function $f\equiv \Minimalfct{\lambda+1}\glbin\Maximalfct{\lambda}$ can for example be constructed using the infinitary wedge (cf. \cref{rem:infinitewedge}) where each column consists of a sequence of functions $(f_{i,j})_{j\in\N}$ whose $\CB$-rank is cofinal in $\lambda$. While $\Minimalfct{\lambda+1}\glbin\Maximalfct{\lambda}$ does not admit a fine $c$-partition (because it is not centered, and the restriction to any clopen which excludes $\iw{0}$ has rank $\leq \lambda$), the function $f$ clearly does.}.

\smallskip

Second case: there exists $P\in \partit$ with $y_P\neq y$. This means that $Y_\partit$ is actually infinite since $f$ is solvable.
We partition $\partit=\partit_0\sqcup\partit_1$ with $\partit_1=\set{P\in\partit}[f\restr{P}\equiv\pgl\Maximalfct{\lambda}]$ and set $Y_i=\set{y_P}[P\in\partit_i]$, $i=0,1$.
If $Y_1$ is empty then $Y_\partit=Y_0$ is infinite and $f\leq \omega\Minimalfct{\lambda+1}$ by \cref{Gluingasupperbound_cor}. Whenever $U\supseteq Y_0$ is clopen, we get $\omega\Minimalfct{\lambda+1}\leq f\corestr{U}$ by \cref{Intertwinereductions}, 
 so we are done. If $Y_1$ is infinite, then $f \leq \omega \pgl \Maximalfct{\lambda}= \Maximalfct{\lambda+1}$ by \cref{Gluingasupperbound_cor} (or by  \cref{Maxfunctions}~(1)) and $\Maximalfct{\lambda+1} \leq f\corestr{U}$ by \cref{Intertwinereductions}, as desired. 

Suppose now that $Y_1$ is finite and nonempty. Then we claim that $Y_1=\set{y}$.  

Towards a contradiction, assume that $y_P\neq y$ for some $P\in \partit_1$. Then $f\restr{P}\equiv \pgl\Maximalfct{\lambda}$, but since $Y_1$ is finite we can find a small enough clopen neighbourhood $V$ of $y$ such that $V\cap (Y_1\setminus \{y\})=\emptyset$. But as $f$ is solvable, there exists $Q\in\partitn{y}$ with $y_Q\in V$ and $f\restr{P}\leq f\restr{Q}$. Since necessarily $f\restr{Q}\equiv \Minimalfct{\lambda+1}$, we get $\pgl\Maximalfct{\lambda}\leq \Minimalfct{\lambda+1}$, in contradiction with \cref{cor:CenteredSucessor}.
So we must have $Y_1=\set{y}$ and $Y_0$ is infinite. \ynote*{proof of Solvability Theorem.}{
We group parts by \cocenter{} $A_z=\bigcup \{P\mid y_P= z\}$ for $z\in Y_\partit$. Note that by the first case, we have $f\restr{A_y}\equiv \pgl \Maximalfct{\lambda}$, since $y\in Y_1$, while for $z\in Y_\partit\setminus \{y\}$ we get $f\restr{A_z}\leq\Minimalfct{\lambda+1}\glbin\Maximalfct{\lambda}$, since $y_P=z$ implies $f\restr{P}\equiv \Minimalfct{\lambda+1}$.

We distinguish two subcases.
\begin{enumerate}
\item Suppose that for every clopen $V\ni y$ the set $Y_0\setminus V$ is finite. Enumerating $Y_\partit\setminus \{y\}$ as $(z_n)_{n>0}$, we have $z_n\to y$ as $n\to \infty$. So setting $f_0=f\restr{A_y}$ and $f_n=f\restr{A_{z_n}}$, for $n>0$, $f=\bigsqcup_{n\in \N} f_n$ satisfies all the assumptions of \cref{Diagonalforlambda+1} and so $g=\fctwedge(\Maximalfct{\lambda}\mid \Minimalfct{\lambda+1})$ works.

%

\item Otherwise, $Y_0\cap W$ is infinite for some clopen $W\notni y$ and set $V=U\setminus W$. Splitting $A=A_{y}\sqcup A_{\neq y}$, with $A_{\neq y}=\bigcup \{P\mid y_P\neq y\}$, and setting $f=f_y\sqcup f_{\neq y}$, we have $f_y\leq \pgl \Maximalfct{\lambda}\leq f\corestr{V}$ and 
by \cref{Gluingasupperbound_cor,Intertwinereductions}, we get $f_{\neq y}\leq \omega\Minimalfct{\lambda+1}\leq f\corestr{W}$. Therefore
\[
f=f_y\sqcup f_{\neq y} \leq \pgl \Maximalfct{\lambda} \glbin \omega\Minimalfct{\lambda+1} \leq f\corestr{V}\glbin f\corestr{W}\leq f\corestr{U},
\]
as desired.\qedhere
\end{enumerate}
}
\end{proof}

\index[IdxNotion]{Precise Structure!solvable functions}
\index[IdxNotion]{finitely generated set of functions}
\index[IdxNotion]{generators}
\begin{theorem}\label{FiniteGenerationForSolvable}
Assume $\FG(\leq\alpha+1)$ for $\alpha<\omega_1$. Let $f:A \to B$ in $\sC_{\alpha+2}$ be solvable with $\partit$.
Then there exists $g\in \FinGl{\generator{\alpha+2}}$ such that $f\leq g$ and $g\leq f\corestr{U}$, so in particular $f\equiv g \equiv f\corestr{U}$, for every clopen $U\supseteq Y_\partit$.
\end{theorem}
\begin{proof}
Let $y\in Y_\partit$ witness that $f$ is solvable with $\partit$. 

As $\FG(\leq\alpha+1)$ holds, we can fix a set $G\subseteq \centered{\alpha+2}$ of representatives for $\set{f\restr{P}}[P\in \partitn{y}]$.
Let us say that $g\in G$ is \emph{reducible infinitely often off} $y$ if there exists a clopen neighborhood $W$ of $y$ in $B$
such that the set $W_g=\set{y_P}[P\in\partit \text{ and }y_P\notin W \text{ and }g\leq f\restr{P}]$ is infinite. We define
\[
H=\set{g\in G}[\text{$g$ is reducible infinitely often off $y$}]
\]
Note that $H$ is a $\leq$-downward closed subset of $G$. 
We further define $\partit_H=\{P\in\partitn{y}\mid \exists h\in H\,  f\restr{P}\equiv h\}$ and $\partit_M=\partit\setminus \partit_H$. 

\ynote*{proof of Finite generation for solvable functions}{
We consider the functions $f_M=\bigsqcup_{P\in \partit_M}f\restr{P}$ and $f_H=\bigsqcup_{P\in \partit_H}f\restr{P}$ and fix a clopen set $U\supseteq Y_\partit$. First note that since $H$ is finite we can choose a single clopen neighborhood $V$ of $y$ in $B$ such that the sets $V_g$ are infinite for all $g\in H$.
It follows from \cref{Intertwinereductions} that $\omega H\leq f\corestr{U\setminus V}$. 
Hence the following holds by \cref{Gluingasupperbound}:
\[
f_H \leq \omega H \leq f\corestr{U\setminus V},
\]
and note that $\omega H\in \FinGl{\generator{\alpha+2}}$ since $\omega h\in \omegacentered{\alpha+2}$ for all $h$.

Note that $\partit_M$ is a fine $c$-partition of $f_M$ since each block $(\partit_M)_{(g,z)}$ equals $\partit_{(g,z)}$ or is empty, because of the definition of $\partit_H$. We observe that $\lambda<\CB(f_M)\leq \alpha+2$, and distinguish two cases:
\begin{description}
\item[$\CB(f_M)=\lambda+1$]  We show that $f_M$ is absorbed by $f_H$, i.e. $f_H\equiv f_M\glbin f_H$. Recall that
$\Maximalfct{\lambda+1}=\omega\,\pgl \Maximalfct{\lambda}$. Since $\CB(f_M)=\lambda+1$, we have $f_M\leq\Maximalfct{\lambda+1}$ by \cref{Maxfunctions}~(1).
Since $\lambda+1<\alpha+2=\CB(f)$, some part of $\partit_H$ attains the top successor rank $\alpha+2$ — otherwise
$f=f_H\sqcup f_M$ would have rank $<\alpha+2$ — so there is $h\in H$ with
$\CB(h)=\alpha+2$, whence $ \pgl \Maximalfct{\lambda}\leq h$ by \cref{ConsequencesGeneralStructureThm}. Therefore
\[
  f_M \;\leq\; \omega\,\pgl \Maximalfct{\lambda} \;\leq\; \omega h \;\leq\; \omega H.
\]
Taking $g=\omega H$ we already get
\[
  f \;\leq\; f_H \glbin f_M \;\leq\; \omega H \glbin  \omega H \;\equiv\; \omega H  \;\leq\; f\corestr{U\setminus V} \;\leq\; f\corestr{U},
\]
as desired.

\item[$\CB(f_M)$ is a double successor] Let $\CB(f_M)=\beta+2$ with $\beta\leq\alpha$, we claim that $f_M$ together with $\partit_M$ is strongly solvable at $y$, which allows us to apply \cref{DiagonalTheorem} at rank
$\beta+2$. To check the first condition, suppose towards a contradiction that the set $Y_W=\set{y_P}[P\in\partit_M \text{ and } y_P\notin W]$ is infinite for some clopen set $W\ni y$. Now for every $z\in Y_W$ there exists $g\in G$ such that $g\equiv f\restr{P}$ for some $P\in \partit_M$ with $y_P=z$. So the pigeonhole principle ensures that some $g$ is chosen infinitely often, so $Y_g=\{ y_P\mid \text{$P\in \partit_M$, $y_P\notin W$ and $g\equiv f\restr{P}$}\}$ is infinite. Since $Y_g\subseteq W_g$, this means that $g\in H$ but then any part $P\in\partit$ with $f\restr{P}\equiv g$ must be in $\partit_H$ by definition, hence not in $\partit_M$. This implies that $Y_g$ is empty, a contradiction, which proves the first condition.

For the second condition, let $P\in\partit_M$ with $y_P\neq y$ and $W\ni y$ be clopen.
As $f$ is solvable with $\partit$ there exists $Q\in \partitn{y}$ such that $f\restr{P}\leq f\restr{Q}$ and $y_Q\in W$.
Note that if $Q\in \partit_H$, then so is $P$ since $H$ is $\leq$-downward closed. Since $P\in \partit_M$, it follows that $Q\in \partit_M$ too, which proves the second condition.

Therefore $f_M$ is strongly solvable at $y$ and by \cref{DiagonalTheorem} there exists $g_M\in\FinGl{\generator{\beta+2}}\subseteq \FinGl{\generator{\alpha+2}}$ such that
\[
f_M \leq g_M \leq f_M\corestr{W}\leq f\corestr{W}.
\]
for every clopen neighborhood $W$ of $y$ in $B$. 

Altogether, for $g:=  \omega H \glbin g_M$, using \cref{Gluingaslowerbound2} we have as desired: 
\[
f \leq f_H \glbin f_M \leq \omega H \glbin g_M \leq f\corestr{U\setminus V} \glbin f\corestr{V\cap U} \leq f\corestr{U}.\qedhere
\]
\end{description}
}
\end{proof}

\index[IdxNotion]{Precise Structure!double successors}
\index[IdxNotion]{finitely generated set of functions}
\index[IdxNotion]{generators}
\begin{theorem}\label{FGatdoublesuccessors}
For all $\alpha<\omega_1$, if $\FG(<\alpha+2)$ holds then so does $\FG(\leq\alpha+2)$.
\end{theorem}

\begin{proof}
Assume $\FG(<\alpha+2)$ and let $f:A\to B$ in $\sC_{\alpha+2}$, we prove that $f$ is equivalent to some function in $\FinGl{\generator{\alpha+2}}$.
By \cref{ExistenceFinePartitions} we can take a fine $c$-partition $\partit$ of $f$, and by \cref{SolvableDecomposition} there is a countable family $\mathcal{U}$
of pairwise disjoint clopen subsets of $B$ whose union covers $Y_\partit$ and such that $f_U=f\restr{A^U_\partit}$ is solvable with $\partit\corestr{U}$ for all $U\in \mathcal{U}$.
Now for all $U\in \mathcal{U}$,  as $Y_{\partit\corestr{U}}\subseteq U$,  \cref{solvablelambda+1,FiniteGenerationForSolvable} ensure there exists $g_U\in \FinGl{\generator{\alpha+2}}$
such that $f_U\leq g_U$ and $g_U\leq f_U\corestr{U}$, so in particular $f_U\equiv g_U$.
Putting everything together, by \cref{Gluingasupperbound_cor,Gluingaslowerbound2} we have
\[
f=\bigsqcup_{U\in\mathcal{U}}f_U\leq\gl_{U\in\mathcal{U}}f_U\leq\gl_{U\in\mathcal{U}}g_U\leq\gl_{U\in\mathcal{U}}{f_U}\corestr{U}\leq\gl_{U\in\mathcal{U}}f\corestr{U}\leq f,
\]
since the members of $\mathcal U$ are pairwise disjoint and cover $Y_\partit$, so each part lies in exactly one $A^U_\partit$. Therefore $f\equiv\gl_{U\in\mathcal{U}}g_U$, which is as desired if $\mathcal{U}$ is finite. If $\mathcal{U}$ is infinite, we conclude using \cref{OmegaOfGenerators}~\cref{trickforfinalproof}.
\end{proof}

We can finally conclude the proofs of all our main results:

\index[IdxNotion]{Precise Structure!final proof}
\index[IdxNotion]{better-quasi-order (\bqo{})}
\begin{proof}[Proof of \cref{MainTheorem1,PreciseStructureThm,Introthm:LevelsAreFinitelyGenerated,Introthm:BQO,Introthm:BQOonScat}]~

As stated in \cref{FGconsequences}, we have already established $\FG(0)$, $\FG(1)$, $\FG(\lambda)$ and $\FG(<\lambda)$ implies $\FG(\lambda+1)$, for every $\lambda$ limit.
\Cref{FGatdoublesuccessors} therefore completes the proof by induction of \cref{PreciseStructureThm}.
This in particular yields \cref{Introthm:LevelsAreFinitelyGenerated}, which in turn implies \cref{Introthm:BQO,Introthm:BQOonScat} by \cref{FGgivesBQO_2,SecondstepforBQOthm}.
This therefore concludes the proof of \cref{MainTheorem1}, since any \bqo{} is \wqo{}.
\end{proof}

%

\chapter{Formalization in Lean 4}\label{chapterFormalization}

The results of this memoir have been formalized in the Lean 4 proof assistant, on top of
its mathematical library Mathlib: the code is available publicly\footnote{\url{https://github.com/yannpequignot/lean-wqo-continuous-functions/}}. 

This chapter describes what exactly has been proved and how to read the formal statements. The formal proof of our main theorem relies only on a property weaker than \bqo{}: the \emph{$2$-\bqo{}} property which was formalized in Lean for this occasion. 

\section{What is formalized}

The development is written for Lean~4 (toolchain \texttt{leanprover/lean4:v4.28.0}) against
Mathlib release \texttt{v4.28.0}, and depends on no other library.
It consists of about
$50\,000$ lines spread over four Lean libraries: \texttt{BQO} (well- and
better-quasi-orders), \texttt{GeneralTopology} and \texttt{ZeroDimensionalSpaces}
(general topology absent from Mathlib, in particular the notion of a zero-dimensional
space and the universality of the Cantor space and of $\mathbb{Q}$), and
\texttt{WqoContinuousFunctions} (the contents of this memoir).

Two properties of the development are worth stating precisely, since they are what makes
the formalization meaningful.

First, \emph{completeness}: no statement is admitted. Lean records, for every theorem, the
axioms its proof depends on; an unfinished proof is recorded as a dependency on the
distinguished axiom \texttt{sorryAx}. The project contains a file whose sole purpose is to
check, at compile time, that the three main theorems and the principal intermediate results
depend on no axiom other than the three which Lean adds to its type theory in order to obtain
ordinary classical mathematics, so that in particular none of them depends on
\texttt{sorryAx}; a build fails if this ceases to hold.
\emph{Propositional extensionality} states that equivalent propositions are equal, and
\emph{soundness of quotients} that related elements have equal classes in a quotient type;
function extensionality is derived from the latter. The third axiom,
\texttt{Classical.choice}, is a global choice operator providing an element of any non-empty
type. The use of this Lean axiom should not be understood as evidence that our results require the axiom of choice.
In Lean, the law of excluded middle is itself derived from \texttt{Classical.choice}, by
Diaconescu's argument~\cite{diaconescu1975}, so that every classical proof depends on it; so
does most of Mathlib, including statements provable in $\mathsf{ZF}$ such as the countability
of the union of two countable sets. This dependency therefore only certifies that the proofs
use classical logic: which fragment of the axiom of choice is needed for the mathematics is a separate question, which
the formalization does not address. We refer to~\cite[Axioms and Computation]{TPIL4} for a
detailed account of these axioms.

Second, \emph{faithfulness}: the statements of the three main theorems mention only notions
of Mathlib together with two definitions of our own, namely zero\--dimen\-sionality and
continuous reducibility, both reproduced below. Reading and believing the formal statements
therefore requires no familiarity with the rest of the development.

The formalization was carried out with the assistance of AI tools, Harmonic's Aristotle and
Anthropic's Claude, which were used to produce Lean proofs of statements fixed in advance by
the present text: AI tools were not used to develop the mathematics. We considered several formalization alternatives, proceeded statement by
statement, monitoring every step: compiling, reading and documenting each contribution. Whenever
the informal argument turned out to be incomplete, it was revised by hand, thereby improving the
manuscript. 
Relying on AI tools to translate informal arguments into Lean is
safe for a simple reason: formal proofs, whatever their origin, are checked by
Lean's kernel, so no proof in the development is taken on trust. In contrast, the \emph{statements} are not checked, that is, the definitions and the formulations of the theorems, and we reproduce them below so any reader can verify for themselves.

\section{The formal statements}

Continuous reducibility is defined exactly as in the introduction: the
translating map $\tau$ is defined on the image of $g\sigma$ and takes values in the image of
$f$. In Lean, writing \texttt{C(X, Y)} for the type of continuous maps from \texttt{X} to
\texttt{Y}:

\begin{lstlisting}
def ContinuouslyReduces_range_based (f : X → Y) (g : X' → Y') : Prop :=
  ∃ σ : C(X, X'),
  ∃ τ : C(Set.range (g ∘ σ), Set.range f),
    ∀ x : X, τ ⟨g (σ x), Set.mem_range_self x⟩ = ⟨f x, Set.mem_range_self x⟩
\end{lstlisting}

Zero-dimensionality is likewise absent from Mathlib, which passes the clopen basis around as
an unbundled hypothesis; we make it a class, bundling the separation axiom under which we
always use it:

\begin{lstlisting}
class ZeroDimensionalSpace (X : Type*) [TopologicalSpace X] : Prop where
  t2 : T2Space X
  clopen_basis : ∃ B : Set (Set X), IsTopologicalBasis B ∧ ∀ s ∈ B, IsClopen s
\end{lstlisting}

Around it the development proves the folklore facts about zero-dimensional spaces that the informal text may take for granted but a formal one must supply: that a countable metrizable space is zero-dimensional, that a zero-dimensional space is totally disconnected, and that a zero-dimensional separable metrizable space embeds into the Cantor space, hence into the Baire space.

A quasi-order is a well-quasi-order if and only if every sequence $(f_n)_{n\in\N}$ admits $m<n$
with $f_m\leq f_n$, and the main theorems use this formulation explicitly. Mathlib provides it as
\texttt{WellQuasiOrdered}, but for a relation on a single type, whereas the functions we compare
have varying domains and codomains, so that each of them inhabits a different type: a sequence of
functions is formalized as a family $X,Y:\N\to\mathsf{Type}$ of spaces together with
$f_n:X_n\to Y_n$. Appealing to \texttt{WellQuasiOrdered} would therefore have required bundling
the functions of all domains and codomains into a single type carrying the quasi-order, and being
explicit about \wqo{}ness is simpler for the reader to verify.



\begin{lstlisting}
theorem MainTheorem3
    (X : ℕ → Type*) (Y : ℕ → Type*)
    [∀ n, TopologicalSpace (X n)] [∀ n, TopologicalSpace (Y n)]
    [∀ n, SeparableSpace (X n)] [∀ n, MetrizableSpace (X n)]
    [∀ n, ZeroDimensionalSpace (X n)]
    [∀ n, MetrizableSpace (Y n)]
    (f : ∀ n, X n → Y n) (hf : ∀ n, Continuous (f n))
    (hsc : ∀ n, ScatteredFun (f n)) :
    ∃ m n : ℕ, m < n ∧ ContinuouslyReduces_range_based (f m) (f n)
\end{lstlisting}

\begin{lstlisting}
theorem MainTheorem1
    (X : ℕ → Type*) (Y : ℕ → Type*)
    [∀ n, TopologicalSpace (X n)] [∀ n, TopologicalSpace (Y n)]
    [∀ n, PolishSpace (X n)] [∀ n, ZeroDimensionalSpace (X n)]
    [∀ n, SeparableSpace (Y n)] [∀ n, MetrizableSpace (Y n)]
    (f : ∀ n, X n → Y n) (hf : ∀ n, Continuous (f n)) :
    ∃ m n : ℕ, m < n ∧ ContinuouslyReduces_range_based (f m) (f n)
\end{lstlisting}

\begin{lstlisting}
theorem MainTheorem2
    (X : ℕ → Type*) (Y : ℕ → Type*)
    [∀ n, TopologicalSpace (X n)] [∀ n, TopologicalSpace (Y n)]
    [∀ n, SeparableSpace (X n)] [∀ n, MetrizableSpace (X n)]
    [∀ n, ZeroDimensionalSpace (X n)]
    [∀ n, MetrizableSpace (Y n)] [∀ n, Countable (Y n)]
    (f : ∀ n, X n → Y n) (hf : ∀ n, Continuous (f n)) :
    ∃ m n : ℕ, m < n ∧ ContinuouslyReduces_range_based (f m) (f n)
\end{lstlisting}

Two remarks on the reading of these statements. The hypothesis of
\cref{MainTheorem1} is stated with \emph{Polish} rather than analytic domains: several necessary results about analytic spaces have not yet been formalized in Lean, so we resorted to formalizing only the Perfect Function Property for continuous functions with Polish domain. The scattered case,
\cref{MainTheorem3}, is proved first and in the stronger form recorded in
\cref{sec:2bqo} below; the other two follow from it by the scattered/non-scattered
dichotomy.

\paragraph{Two design choices}

We mention here two deliberate departures from the text that affects how the formal
statements relate to the informal ones.

The first is about domains and co-domains of functions. Rather than working with arbitrary spaces throughout, the
development reduces once and for all to a type \texttt{ScatFun} of scattered continuous
functions whose domain is a subset of the Baire space and whose codomain is the Baire
space; every function considered in \cref{MainTheorem3} is
continuously equivalent to such a function (cf. \cref{RepresentationforFunctions,FirststepforBQOthm}). All of the structure theory --- pieces,
partitions, gluings, generators --- is developed for \texttt{ScatFun}, and the main
theorems are obtained by transporting along this equivalence.

The second concerns the reduction relation itself. Alongside the image-based definition
displayed above, which is the one appearing in the statements of the main theorems, the
code uses a variant in which $\tau$ is a total function required to be continuous
only on the image of $g\sigma$. The two definitions agree whenever the functions are non empty, and this variant is remarkably more convenient to compose in Lean 4. The passage between them is proved once and used at the boundary of the development.

\section{$2$-\bqo{} suffices}\label{sec:2bqo}

In this memoir, we use \bqo{}ness of continuous reducibility three times: for
finitely generated sets of functions (\cref{SecondstepforBQOthm}), for the passage from the
levels $\sC_\beta$ to their union $\sC_{<\alpha}$ (\cref{FGgivesBQO_2}), and in the proof of
\cref{LocalCenterednessFromBQO}. Only the last one asks for more than \wqo{}ness: it needs
that the set $(\sC_{<\alpha})^{\N}$ of $\N$-indexed sequences over $\sC_{<\alpha}$, ordered
by domination, is a \wqo{} --- a property that \wqo{}ness of $\sC_{<\alpha}$ alone does not
ensure. The formalization shows that the following weakening of \bqo{}, obtained by
allowing only barriers of dimension at most two, suffices in all three places, as the
footnote to \cref{LocalCenterednessFromBQO} already announced.

Call a \emph{pair-sequence} in a set $Q$ a family $(q_{m,n})_{m<n}$ of elements of $Q$
indexed by the pairs $m<n$ of natural numbers, and say that a quasi-order $\leq$ on $Q$ is
\emph{$2$-\bqo{}} if for every pair-sequence there are $m<n<l$ with
$q_{m,n}\leq q_{n,l}$; equivalently, no pair-sequence is bad. In Lean:

\begin{lstlisting}
def TwoBQO {α : Type*} (r : α → α → Prop) : Prop :=
  ∀ (f : PairSeq α), ∃ m n l : ℕ, ∃ (hmn : m < n) (hnl : n < l),
    r (f m n hmn) (f n l hnl)
\end{lstlisting}

Taking pair-sequences that do not depend on their second index shows at once that a
$2$-\bqo{} quasi-order is a \wqo{}, and a \bqo{} quasi-order is $2$-\bqo{} by definition;
the notion therefore sits between the two. The facts about it that the three uses above
require are the following, all proved in the \texttt{BQO} library.

\begin{enumerate}
\item A $2$-\bqo{} quasi-order is a \wqo{}.
\item Every well-order is $2$-\bqo{} --- a bad pair-sequence would produce the infinite
      descending sequence $n\mapsto q_{n,n+1}$ --- and so are the ordinals and the order
      $\leq^\bullet$ which indexes the sum of \cref{FGgivesBQO_2}.
\item \emph{Finite products and transfer.} The product order on $\N^{k}$ is $2$-\bqo{}, and
      $2$-\bqo{}ness is inherited along a co-homomorphism: this is what is needed in
      \cref{SecondstepforBQOthm}, through the co-homomorphism $f\mapsto s_f$.
\item \emph{Lexicographic sums.} If the index quasi-order is $2$-\bqo{} and each fibre is
      $2$-\bqo{}, then so is the lexicographic sum $\Sigma_{i}Q_i$, where
      $(i,x)\leq(j,y)$ if $i<j$, or $i=j$ and $x\leq y$ in $Q_i$; this gives
      \cref{FGgivesBQO_2}. The formalization provides it both for a well-ordered index and
      for an arbitrary $2$-\bqo{} index quasi-order.
\item \emph{Sequences.} If $\leq$ is $2$-\bqo{}, then the $\N$-indexed sequences over $Q$
      form a \wqo{} under domination: $(q_n)_n\sqsubseteq(q'_n)_n$ if there is a strictly
      increasing $e:\N\to\N$ with $q_n\leq q'_{e(n)}$ for all $n$. This is what the proof of
      \cref{LocalCenterednessFromBQO} consumes.
\end{enumerate}

The only genuinely combinatorial ingredient is Ramsey's theorem for triples, used to split a
pair-sequence into a perfect or a bad one along a subsequence, which is proved in the same
library. Accordingly, the formal counterpart of \cref{MainTheorem3} is proved through the
statement that continuous reducibility is $2$-\bqo{} on scattered functions, from which the
\wqo{}ness follows.

%


\chapter{Concluding remarks and Open questions}\label{sectionConclusion}

\section{Comments, questions and optimality}
%
%
%
%
%
Here are a few interesting corollaries 
 of the whole analysis.

\begin{theorem}\label{PreciseStructureCor1}
Every function in $\sC$ is locally centered.
\end{theorem}

\begin{proof}
By \cref{Introthm:LevelsAreFinitelyGenerated,SecondstepforBQOthm,LocalCenterednessFromBQO}.
\end{proof}

As a byproduct, we obtained a ``normal form'' theorem for functions in $\sC$.

\begin{theorem}\label{PreciseStructureCor2}
Every function in $\sC$ is continuously equivalent to a finite-to-one function whose domain and range are countable Polish spaces.
\end{theorem}

\begin{proof}
Using \cref{BasicsOnGluing}, observe first that having countable Polish domain and range and being finite-to-one are all properties preserved by finite gluing,
so by the Precise Structure \cref{PreciseStructureThm}, it is enough to prove the result for generators.
By \cref{BasicsOnGluing,BasicsOnPointedGluing,BasicfactsWedge}, all three operations -- gluing, pointed gluing and wedge -- preserve countability and Polishness of both domain and range,
so by induction all generators -- built using only these operations -- have countable Polish domain and range.

\ynote*{injectivity of the generators}{For finite-to-oneness, note that by
\cref{BasicsOnGluing,BasicsOnPointedGluing} again the gluing and pointed gluing operations
preserve injectivity. As the definition of $\centered{\alpha}$ and $\omegacentered{\alpha}$ uses
only these two operations, starting from $\Minimalfct{\lambda+1}$ and $\Maximalfct{\lambda}$, an
induction along that definition yields that for all $\alpha<\omega_1$ all functions in
$\centered{\alpha}\cup\omegacentered{\alpha}$ are injective; in particular so are
$\Maximalfct{\alpha}$, $\pgl\Maximalfct{\alpha}$ and $\Minimalfct{\alpha+1}$.}
The only operation which produces non-injective functions is the wedge, but by \cref{BasicfactsWedge} it produces finite-to-one generators, thus concluding the proof.
\end{proof}

It is now the right moment to open a line of questions regarding the optimality of our general strategy:
was proving the Precise Structure \cref{PreciseStructureThm} really necessary to get all its consequences?
For the first instance of this general question, observe that
\cref{PreciseStructureCor1,PreciseStructureCor2} are statements that follow from the Precise Structure \cref{PreciseStructureThm}, but they do not concern \emph{a priori} the generators.
This suggests the following angle:

\begin{question}
Is there a proof of \cref{PreciseStructureCor1,PreciseStructureCor2} that does not use the Precise Structure \cref{PreciseStructureThm},
or \cref{Introthm:LevelsAreFinitelyGenerated}, or even \cref{Introthm:BQOonScat}?
\end{question}

The second instance appears already in the introduction: \cref{Introthm:LevelsAreFinitelyGenerated} implies \cref{Introthm:BQOonScat} by \cref{SecondstepforBQOthm},
but they could still be somehow of the same strength.

\begin{question}
Can we deduce \cref{Introthm:LevelsAreFinitelyGenerated} from \cref{Introthm:BQOonScat}?
\end{question}

Even within our global strategy, one can wonder if all these operations and generators are really needed.
In that direction, the strongest sign is that -- up to continuous equivalence -- a generating set for $\sC$ cannot omit the generators we identified.

\ynote*{$\omegacentered{}$ among the generators}{Write
$\generator{}=\bigcup_{\alpha<\omega_1}\generator{\alpha}$,
$\centered{}=\bigcup_{\alpha<\omega_1}\centered{\alpha}$ and
$\omegacentered{}=\bigcup_{\alpha<\omega_1}\omegacentered{\alpha}$.\index[IdxSymb]{G@{$\generator{}$}}\index[IdxSymb]{C@{$\centered{}$}}\index[IdxSymb]{W@{$\omegacentered{}$}}
We already have all we need to show that elements of $\centered{}$ and $\omegacentered{}$ cannot be
omitted, but the situation for wedge generators is a bit trickier and needs more preparation.}

Recall from \cref{subsectionGenerators} that a wedge generator $\fctwedge(\mathcal M\mid D)$ has
for verticals the finite gluings $\gl F$ for $F\in\mathcal M$ and for diagonal $\gl D$; below we
write $f_i$ for its verticals and keep the diagonal in set form.

\index[IdxNotion]{wedge operation!reduced form}
\begin{lemma}\label{WedgeReducedForm}
Any wedge generator $f$ admits a \emph{reduced form} defined as follows:\index[IdxNotion]{wedge operation!reduced form}\index[IdxNotion]{generators!wedge generator}
$f\equiv\fctwedge(f_0,\ldots, f_k \mid F_{k+1})$ for $k\in\N$ with $f_i\in\sC$ for $i\leq k$ and $F_{k+1}\subseteq\centered{}$ satisfying:
\begin{enumerate}
\item $(f_i)_{i\leq k}$ is an antichain with respect to continuous reduction.
\item For all $i\leq k$ and for all $h\in F_{k+1}$ we have $h\not\leq f_i$, and $F_{k+1}$ is an antichain. 
\item Either $f\equiv \omega F_{k+1}$, or for all $i\leq k$ we have $\pgl f_i\not\leq F_{k+1}$.
\end{enumerate}
\end{lemma}
\begin{proof}
We start with any $f=\fctwedge(f_0,\ldots, f_k \mid F_{k+1})$ for $k\in\N$ with $f_i\in\sC$ for $i\leq k$ and $F_{k+1}\subseteq\centered{}$,
and we give three finite procedures to get a reduced form for $f$.

First, we can pick an $\leq$-antichain of $\leq$-maximal elements in $\set{f_0,\ldots, f_k}$ 
and still obtain by \cref{cor:wedgeSets} a function equivalent to the original one, getting the first item.

For the second procedure:
suppose that some $h\in F_{k+1}$ is reduced by $f_i$ for some $i\leq k$ (resp. by some $h'\in F_{k+1}$).
By definition of the wedge operation there is $n<|F_{k+1}|$ such that
for all $m\in \N$ we have $f\restr{N_{(k+1+m,n)}}\equiv h$. We want to apply \cref{Wedgeasupperbound} on a `subfunction' of $f$, so we need to produce
relevant sets $A_m$ for all $m\in \N$. Define $A_m=N_{(m)}$ if $i\neq m\leq k$ (resp. for all $i\leq k$), $A_i=N_{(i)}\cup\bigcup_{m\in\N}N_{(k+1+m,n)}$,
and $A_{k+1+m}=N_{(k+1+m)}\setminus N_{(k+1+m,n)}$ for all $m\in\N$. Since $h\leq f_i$ the partition $(A_m)_{m\in\N}$ of $\dom(f)$ along with the point $\iw{0}$
satisfy the hypotheses of \cref{Wedgeasupperbound}, which gives $f\leq\fctwedge(f_0,\ldots,f_k\mid F_{k+1}\setminus\set{h})\leq f$.
Iterating this second procedure we get the second item.

For the third item, observe first that $\omega F_{k+1}\leq f$ by \cref{Gluingaslowerbound}, and consider the set $J=\set{i\leq k}[\pgl f_i\leq F_{k+1}]$.
If $J=k+1$ then $f\leq\omega F_{k+1}$ by \cref{Gluingasupperbound}, and otherwise $f\leq\fctwedge((f_j)_{j\in k+1\setminus J}\mid F_{k+1})\leq f$, concluding the proof.
\end{proof}

%

\begin{lemma}\label{ReducedFormIsUseful}
Let $f=\fctwedge(f_0,\ldots, f_k \mid F_{k+1})$ be a wedge generator in reduced form and call $F=\set{f_0,\ldots,f_k}\cup F_{k+1}$.
Suppose that $(A,A')$ is a clopen partition of $\dom(f)$ and $f\restr{A'}\leq \FinGl{F}$.
Then either $f\equiv \omega F_{k+1}$, or $f\leq f\restr{A}$.
\end{lemma}
\begin{proof}
We show that either $f\equiv  \omega F_{k+1}$, or $A$ contains $(i)\conc\iw{0}$ for all $i\leq k$ and infinitely many centers for every function in $F_{k+1}$.
Since $A$ is a clopen set of $\dom f$ the latter allows us to use \cref{Wedgeasupperbound} inside $A$ and conclude.

Suppose now that $x=(i)\conc\iw{0}\in A'$ for some $i\leq k$. Since $x$ is a center for $\pgl f_i$ by \cref{Pgluingofregulariscentered,Centerinvariance} we have $\pgl f_i\leq f\restr{A'}\leq \FinGl{F}$.
Since $\pgl f_i$ is centered, by \cref{Centerinvariance} we have $\pgl f_i\leq F$. But $\CB(\pgl f_i)=\CB(f_i)+1$ by \cref{BasicsOnPointedGluing} since $f_i\in \sC$, so $\pgl f_i\not\leq f_i$.
Therefore $\pgl f_i\leq F\setminus\set{f_i}$, and since $f$ is in reduced form we cannot have $f_i\leq f_j$, so even less $\pgl f_i\leq f_j$, for some $j\neq i\leq k$.
We must have then $\pgl f_i\leq F_{k+1}$, so \cref{WedgeReducedForm} again tells us that $f\equiv\omega F_{k+1}$.

We can thus suppose that $(i)\conc\iw{0}\in A$ for all $i\leq k$. If $F_{k+1}=\emptyset$ we are done, so suppose it is not, and
some function $h\in F_{k+1}$ has only finitely many centers in $A$. Then $h$ has infinitely many centers in $A'$ by the pigeonhole
principle, so by centeredness and \cref{Gluingaslowerbound} we get $\omega h\leq f\restr{A'}\leq\FinGl{F}$, and by \cref{Intertwinereductionsforomegacentered} we obtain $\omega h\leq F$.
But since $h$ is centered, we cannot have $\omega h\leq h$ by \cref{Rigidityofthecocenter} and in turn $\omega h\leq F\setminus\set{h}$, which contradicts our assumption of normal form.
\end{proof}

We are ready to show the optimality of our generating set.

\index[IdxNotion]{generators}
\begin{theorem}\label{OptimalityOfGenerators}
Assume that $G\subseteq\sC$ is generating $\sC$, that is, any function in $\sC$ is continuously equivalent to one in $\FinGl{G}$.
Then $G$ contains $\generator{}$ up to continuous equivalence, that is for all $f\in \generator{}$ there is $g\in G$ such that $f\equiv g$.
\end{theorem}

\begin{proof}
Take $f \in\generator{}$ and $g\in\FinGl{G}$ such that $f\equiv g$. Write $g=\gl_{i\in I}g_i$ for some finite $I\subseteq\N$ and $g_i\in G$ for all $i\in I$.
Since $g\leq f$ we have $g_i\leq f$ for all $i\leq n$, it is enough to find $i\in I$ such that $f\leq g_i$.
\ynote*{the three clauses of the definition of $\generator{}$}{We prove this by case analysis on
the three clauses of the definition of $\generator{}$.

If $f\in\centered{}$, then $f\leq \gl_i{g_i}$ implies $f\leq g_i$ for some $i\in I$ by
\cref{Centerinvariance}.

If $f\in\omegacentered{}$, then by \cref{lem:omegagenerators} either $f=\Maximalfct{\lambda}$ for
some limit $\lambda$, in which case $\CB(g_i)=\lambda$ for some $i\in I$ by
\cref{CBrankofclopenunion} and so $g_i\equiv f$ by the General Structure
\cref{JSLgeneralstructure}, or $f=\omega h$ for some $h\in\centered{}$, in which case
$f\leq g_i$ for some $i\in I$ by \cref{Intertwinereductionsforomegacentered}.}

Suppose now that $f=\fctwedge(f_0,\ldots, f_k \mid F_{k+1})$ is a wedge generator in reduced form and call $F=\set{f_0,\ldots,f_k}\cup F_{k+1}$.
Note that if $B$ is any clopen set of $\im f$ that does not contain $\iw{0}$ then by definition of the wedge operation we have $f\corestr{B}\leq\FinGl{F}$.

Since $g\leq f$ by \cref{Gluingaslowerbound} there is a relative clopen partition $(B_i)_{i\in I}$ in $\im f$ such that $g_i\leq f\corestr{B_i}$.
If $\iw{0}\notin B=\bigcup_{i\in I}B_i$ then $f\leq g \leq f\corestr{B}\leq\FinGl{F}$, so by \cref{ReducedFormIsUseful} applied to $(\emptyset,\dom(f))$
we obtain $f\equiv\omega F_{k+1}$ and we are done. 

Suppose now that $\iw{0}\in B$. There is a unique $n\in I$ such that $\iw{0}\in B_n$, and for some clopen partition $(C,C')$ of $\im f$ we have $B_n\subseteq C$ and $\bigcup_{i\neq n} B_i \subseteq C'$. Note that $\gl_{i\neq n}g_i\leq f\corestr{C'}\leq\FinGl{F}$. 
Now apply \cref{ReducedFormIsUseful} to the clopen partition $(A,A')=(f^{-1}(C), f^{-1}(C'))$ of $\dom f$. 
If $f\equiv\omega F_{k+1}$ once again we are done, otherwise $f\leq f\restr{A}\leq g_n$ so we are also done.
\end{proof}

Observe that if \cref{OptimalityOfGenerators} shows that one cannot omit our generators,
the sets $\generator{\alpha}$ that we have identified are far from being minimal, for at least two reasons.
First because some pairs of distinct generators can be continuously equivalent, as evidenced (for instance) by the Reduced Form \cref{WedgeReducedForm}.
Second because some generators could be dispensable, besides for each $\alpha=\lambda+n$ with $\lambda$ limit or null and $n\in\N$,
our set $\generator{\alpha}$ generates not only $\sC_{\alpha}$ but actually the whole (bigger) class $\sC_{[\lambda,\alpha]}$.
This suggest the following question.

\begin{question}\label{qu:numberofgenerators}
Given $\alpha=\lambda+n$ with $\lambda<\omega_1$ limit and $n\in\N$, what is the smallest number of functions needed to finitely generate $\sC_{[\lambda,\alpha]}$?
\end{question}

In case $\alpha=\lambda$ is limit, the answer to \cref{qu:numberofgenerators} is trivially 1, since in that case $\generator{\lambda}=\omegacentered{\lambda}=\set{\Maximalfct{\lambda}}$.
We now study \cref{qu:numberofgenerators} at $\lambda+1$ for $\lambda$ limit or 1, proving first the optimality of \cref{FGatsuccessoroflimit}.

\begin{theorem}\label{OptimalityatSuccessorofLimit}
Let $\lambda<\omega_1$ be limit or 1.
Up to continuous equivalence, the smallest subset that generates  $\sC_{[\lambda,\lambda+1]}$ is given by

\[\set{\Maximalfct{\lambda},\Minimalfct{\lambda+1},\pgl\Maximalfct{\lambda},\omega\Minimalfct{\lambda+1},
\fctwedge(\Maximalfct{\lambda}\mid\Minimalfct{\lambda+1}),\Maximalfct{\lambda+1}}.
\]
\end{theorem}
\begin{figure}
\centering
\tikzset{every picture/.style={line width=0.75pt}} 

\begin{tikzpicture}[x=0.75pt,y=0.75pt,yscale=-1,xscale=1,
 every node/.append style={anchor=center, draw, inner sep=3pt, rounded corners=3pt}]

\node (1)   at (120,140)   {$ \Maximalfct{\lambda }$};
\node (2)  at (120,105)   {$\Minimalfct{\lambda +1}$};
\node (3a) at  (60,70)  {$\omega \Minimalfct{\lambda+1}$};
\node (3b)  at (180,70)  {$\pgl \Maximalfct{\lambda}$};

\node   (4) at (120,35)   {$\fctwedge(\Maximalfct{\lambda } \mid \Minimalfct{\lambda +1})$};
\node (5) at (120,0)   {$\Maximalfct{\lambda  +1}$};

\draw  (1) -- (2) ;
\draw  (2) -- (3a) ;
\draw  (2) -- (3b) ;
\draw  (3a) -- (4) ;
\draw  (3b) -- (4) ;
\draw  (4) -- (5) ;
\end{tikzpicture}
\index[IdxFig]{Hasse diagram of our minimal set of generators at successors of limit}
\caption{Hasse diagram of our minimal set of generators of $\sC_{[\lambda,\lambda+1]}$ for $\lambda$ limit or $1$.}
\end{figure}

\begin{proof}
\ynote*{$G$ is the set $\generator{\lambda+1}$}{For $\lambda$ limit or $1$, let $G$ be the set from
the statement. By \cref{AlreadyKnownGenerators}, $G$ is the set $\generator{\lambda+1}$ up to
elementwise continuous equivalence; the same computation applies when $\lambda=1$, using
$\Maximalfct{1}\equiv\omega\Minimalfct{1}\in\omegacentered{1}$. In particular $G$ generates
$\sC_{[\lambda,\lambda+1]}$ by \cref{JSLgeneralstructure,FGatsuccessoroflimit}, and all functions
of $G$ are in $\generator{}$, so that by \cref{OptimalityOfGenerators} it is enough to show that no
two of these functions are continuously equivalent.}
By \cref{cor:CenteredSucessor} we know that $\Minimalfct{\lambda+1}<\pgl\Maximalfct{\lambda}$,
and by \cref{Centerinvariance} we have $\pgl\Maximalfct{\lambda}\nleq \omega \Minimalfct{\lambda+1}$.
By considering their $\CB$-type, it follows from \cref{CBbasicsfromJSL-tp} that $\Minimalfct{\lambda+1}\nleq\Maximalfct{\lambda}$
and $\pgl\Maximalfct{\lambda}\ngeq \omega \Minimalfct{\lambda+1}$ so these two generators are incomparable. Now note that the following reduction holds
\[
\Maximalfct{\lambda}< \Minimalfct{\lambda+1} < 
\Bigl\{\begin{array}{c} \pgl\Maximalfct{\lambda} \\ \omega \Minimalfct{\lambda+1}\end{array}\Bigr\}
\leq \fctwedge(\Maximalfct{\lambda}\mid\Minimalfct{\lambda+1}) \leq \Maximalfct{\lambda+1}
\]
and we show that 
$\omega\Minimalfct{\lambda+1} \ngeq \fctwedge(\Maximalfct{\lambda}\mid\Minimalfct{\lambda+1})$ and
$\fctwedge(\Maximalfct{\lambda}\mid\Minimalfct{\lambda+1})\ngeq \Maximalfct{\lambda+1}$.

Suppose that $(\sigma, \tau)$ reduces $\fctwedge(\Maximalfct{\lambda}\mid\Minimalfct{\lambda+1})$ to $\omega\Minimalfct{\lambda+1}$. In particular we would get
\[\pgl\Maximalfct{\lambda}\equiv\fctwedge(\Maximalfct{\lambda}\mid\Minimalfct{\lambda+1})\restr{N_{(0)}}\leq\omega\Minimalfct{\lambda+1},\]
and by \cref{Centerinvariance} since $\pgl\Maximalfct{\lambda}$ is centered we would have $\pgl\Maximalfct{\lambda}\leq\Minimalfct{\lambda+1}$, a contradiction.

Suppose that $(\sigma, \tau)$ reduces $\Maximalfct{\lambda+1}$ to $\fctwedge(\Maximalfct{\lambda}\mid\Minimalfct{\lambda+1})$, then by \cref{CBbasicsfromJSL} the function $\sigma$ maps injectively the points of $\CB$-rank $\lambda$ for the former function into those of the latter function. Therefore there exists $i\in\N$ and $j\neq 0$ such that $\sigma((i)\conc \iw{0})=(j)\conc \iw{0}$, so by centeredness of $\pgl\Maximalfct{\lambda}$ we would get $\pgl\Maximalfct{\lambda}\leq\Minimalfct{\lambda+1}$, a contradiction again.
\end{proof}

We obtain as a corollary the answer to \cref{qu:numberofgenerators} in case $\alpha=\lambda+1$.

\begin{corollary}
The smallest sets generating $\sC_{[\lambda,\lambda+1]}$ are all order-isomorphic, for $\lambda<\omega_1$ limit and for $\lambda=1$.

In particular they all have the same size of 6 generators.
\end{corollary}

More generally the partial orders induced by the quasi-ordered classes $(\sC_{\lambda+1},{\leq})$ are all isomorphic, for $\lambda$ limit or 1.

\begin{question}\label{qu:isomorphismofthelevels}
Given $\lambda<\omega_1$ limit or $1$, are the partial orders induced by the classes $(\sC_{\lambda+n+1},\leq)$ isomorphic for all $n\in\N$?
\end{question}

We can now give long-awaited examples.

\begin{proposition}\label{SomeCounterExamples}
In $\sC$ some functions are not equivalent to any injective function, there are injective functions not equivalent to any embedding, and there are $\leq$-antichains of any finite size.
More specifically, for $\lambda<\omega_1$ limit or 1:
\begin{enumerate}
\item The injection $\fctwedge(\Maximalfct{\lambda}\mid\Minimalfct{\lambda+1})$, a wedge generator in $\generator{\lambda+1}$ with vertical $\set{\Maximalfct{\lambda}}\subseteq\generator{\lambda}$ and diagonal $\set{\Minimalfct{\lambda+1}}\subseteq\centered{\lambda+1}$, is not continuously equivalent to an embedding.

\item For all $n\in\N$ non null $(k\pgl \pgl\Maximalfct{\lambda}\gl(n-k)\pgl \omega\Minimalfct{\lambda+1})_{ k\leq n}$ is a $\leq$-antichain, built from the two generators $\pgl\pgl\Maximalfct{\lambda},\pgl\omega\Minimalfct{\lambda+1}\in\centered{\lambda+2}$ --- each a pointed gluing of a singleton in $\centered{\lambda+1}\cup\omegacentered{\lambda+1}$.\index[IdxNotion]{antichain}

\item The function $\fctwedge(\pgl\Maximalfct{\lambda},\omega\Minimalfct{\lambda+1}\mid\emptyset)$, a wedge generator in $\generator{\lambda+2}$ with verticals $\set{\set{\pgl\Maximalfct{\lambda}},\set{\omega\Minimalfct{\lambda+1}}}$ and empty diagonal, is a simple and not continuously equivalent to the pointed gluing of its rays, nor to an injective function.
\end{enumerate}
\end{proposition}

\begin{proof}[Proof sketch.]
First item: suppose that $(\sigma,\tau)$ is a reduction from $f:=\fctwedge(\Maximalfct{\lambda}\mid\Minimalfct{\lambda+1})$ to an embedding $g\equiv f$.
Observe that $\CB_{\lambda+1}(f)=\set{(i)\conc\iw{0}}[i\in\N]$ so by \cref{CBbasicsfromJSL} we have $\set{\sigma((i)\conc\iw{0})}[i\in\N]\subseteq\CB_{\lambda+1}(g)$.
If $\set{\sigma((i)\conc\iw{0})}[i\in\N]$ is not discrete, any limit point of it should be of CB-rank $\lambda+2$ for the embedding $g$,
which is impossible as $\CB(g)=\lambda+1$.
So $\set{\sigma((i)\conc\iw{0})}[i\in\N]$ is discrete, and since $g$ is an embedding the set $\set{g\sigma((i)\conc\iw{0})}[i\in\N]$ is also discrete in $\im g$.
Pick pairwise disjoint open sets $U_i\ni g\sigma((i)\conc\iw{0})$ for all $i\in\N$ and by centeredness and \cref{Gluingaslowerbound} we obtain
$\pgl\Maximalfct{\lambda}\gl\omega\Minimalfct{\lambda+1}\leq g\leq f\leq\pgl\Maximalfct{\lambda}\gl\omega\Minimalfct{\lambda+1}$.

Therefore if $f$ is equivalent to an embedding, then it must be equivalent to $\pgl\Maximalfct{\lambda}\gl\omega\Minimalfct{\lambda+1}$.
Now since $f\in\generator{}$, by (the proof of) \cref{OptimalityOfGenerators} this implies that either $\pgl\Maximalfct{\lambda}\equiv f$ or $\omega\Minimalfct{\lambda+1}\equiv f$,
and this is impossible by \cref{OptimalityatSuccessorofLimit}.

\medskip

Second item: first show that $f:=\pgl \pgl\Maximalfct{\lambda}$ and $g:=\pgl \omega\Minimalfct{\lambda+1}$ are incomparable.
Note that there are both of $\CB$-rank $\lambda+2$ and $\CB_{\lambda+1}(f)=\CB_{\lambda+1}(g)=\set{\iw{0}}$.  
Suppose that $(\sigma,\tau)$ reduces $g$ to $f$ then by \cref{CBbasicsfromJSL} we have $\sigma(\iw{0})=\iw{0}$,
and as $f$ is injective we furthermore have $\sigma^{-1}(\set{\iw{0}})=\set{\iw{0}}$ and $\tau^{-1}(\set{\iw{0}})=\set{\iw{0}}$ by \cref{ContRedonEmbed}.

Therefore by continuity $\ray{g}{\iw{0},0}\leq \bigsqcup_{i\leq k}\ray{f}{\iw{0},i}$ for some $k\in\N$. 
So we get $\omega\Minimalfct{\lambda+1}\leq k\pgl\Maximalfct{\lambda}$, in contradiction with \cref{CBbasicsfromJSL-tp}.
Similarly, $f\leq g$ implies $\pgl\Maximalfct{\lambda}\leq \omega\Minimalfct{\lambda+1}$, contradicting \cref{OptimalityatSuccessorofLimit}.

Now suppose  that $mf\gl ng\leq pf\gl qg$ holds for $n,m,p,q\in \N$. By \cref{Centerinvariance,CBbasicsfromJSL-tp}, as $f$ is centered and incomparable with $g$, each of the $m$ copies of $f$ must be reduced by distinct copies of $f$ of the right hand side, so $m\leq p$. Similarly, we must have $n\leq q$.

\medskip

Third point: observe that $f:=\fctwedge(\pgl\Maximalfct{\lambda},\omega\Minimalfct{\lambda+1}\mid\emptyset)$
to is simple of $\CB$-type $(\lambda+2,1)$.
If $f$ was centered, any center would be of maximal $\CB$-rank (see proof of \cref{scatteredhavecocenter}) and so one of $\iw{0}$ or $(1)\conc\iw{0}$ would be a center for $f$. But none of these two points are center for $f$, for this would imply either $f\leq \pgl \pgl\Maximalfct{\lambda}$ or $f\leq \pgl \omega\Minimalfct{\lambda+1}$, and since $f$ reduces both of these functions, it would contradict the second point.

Now suppose that $(\sigma,\tau)$ is a continuous reduction from $f$
to an injective function $g\equiv f$.
There are now two possibilities: either $\sigma(\iw{0})\neq\sigma((1)\conc\iw{0})$ or $\sigma(\iw{0})=\sigma((1)\conc\iw{0})$.
In the former case we have $tp(g)=(\lambda+2,2)$ because $g$ is injective, so $g\nleq f$, a contradiction. In the latter case,
using \cref{Centerinvariance,Pgluingaslowerbound2} we would have
$\pgl\set{\pgl\Maximalfct{\lambda},\omega\Minimalfct{\lambda+1}}\leq g\equiv f\leq \pgl \ray{f}{\iw{0},n}\equiv\pgl\set{\pgl\Maximalfct{\lambda},\omega\Minimalfct{\lambda+1}}$,
thus we would have $f\equiv \pgl\set{\pgl\Maximalfct{\lambda},\omega\Minimalfct{\lambda+1}}$ which is impossible by \cref{Centerinvariance}
since $\pgl\set{\pgl\Maximalfct{\lambda},\omega\Minimalfct{\lambda+1}}$ is centered whereas $f$ is not .
\end{proof}

Observe that in the previous proposition we build antichains of unbounded finite size using only $2$ incomparable generators.
In the same line as \cref{qu:numberofgenerators}, one can wonder if $\sC_{n}$ contains more and more generators as $n$ increases. 
Recently in \cite[Theorem 3.3.5]{phdScamperti} Scamperti has built -- using the gluing and pointed gluing operations on spaces --
sequences of cardinality $n$ of closed countable Polish spaces of CB-rank $n$ for all $n\geq 2$ that form an antichain for topological embeddability.
The identity functions on these spaces are actually all among our generators, which implies that the number of generators in $\sC_{n}$ goes to infinity as $n\to\infty$.

\section{Sharpness}

In our \cref{MainTheorem1} there are various hypotheses on the complexity of: the domain, the co-domain, and the function itself.
Trying to see when one of these hypotheses in one these results can be relaxed leads to a large number of open questions.

\cref{scatterediffemptykernel_general} guarantees a decomposition of the domain of any scattered function,
in this respect it does make sense to conjecture that scattered functions (even discontinuous ones) may be constructed from simpler ones,
which might in turn allow for a (very general) \bqo{} result, reminiscent of Fraïssé's Conjecture for scattered linear orders.

\index[IdxNotion]{Functional Fraïssé's Conjecture}\index[IdxNotion]{better-quasi-order (\bqo{})}
\begin{conjecture}[Functional Fraïssé's Conjecture - FFC]
Continuous reducibility is \bqo{} on the class of scattered functions.
\end{conjecture}

Our \cref{Introthm:BQOonScat} proves the fragment of the FFC for continuous functions with zero-dimensional separable metrizable domain and metrizable range.
If the full FFC seems quite ambitious, getting it for functions on separable metrizable domains might be more accessible.
Our proof however makes essential use of continuity (for example in \cref{Maxfunctions,CenteredasPgluing}) and zero-dimensionality (for example in \cref{ultrametricdecomposition}).

We now develop two directions for future research, sources for open questions motivated by some further results.
First, what happens for classes of definable functions on Polish spaces, and second what happens to the Perfect Function Property on larger classes of functions.

\subsection{On functions with Polish domain}\label{DirectionPolish}

Here is the very first conjecture one can think of after \cref{MainTheorem1}.

\index[IdxNotion]{better-quasi-order (\bqo{})}
\begin{conjecture}
Continuous reducibility is a \bqo{} on Borel functions between Polish zero-dimensional spaces.
\end{conjecture}

This was \cite[Question 5.6]{carroy2013quasi}, we upgrade it here to the status of conjecture, and we do the same for the next question.
\cite[Question 5.7]{carroy2013quasi} gives indeed a possibility to attack the above \bqo{} conjecture by splitting the problem in the study of injective functions on one side,
and those with countable image on the other, provided we first prove the following statement, interesting on its own.

\begin{conjecture}
Given any Borel function $f$ between Polish (zero-dimen\-sional) spaces, there is a Borel partition $(A,B)$ of $\dom f$ such that
$f\restr{A}$ is continuously equivalent to an injection, and $f\restr{B}$ has countable image.
\end{conjecture}

If we drop zero-dimensionality, the picture becomes quickly much more complex.
Schlicht has indeed proven \cite[Theorem 1.7]{schlicht_nozerodim} that the Wadge quasi-order on Borel subsets of a Polish space $X$ is a \wqo{} if and only if $X$ is zero-dimensional,
which by \cref{LinkswithWadgeEmbedMeas} translates to continuous reducibility being non \wqo{} in Borel functions from any non zero-dimensional Polish space.
This does not however ruin all hopes for continuous functions.

\medskip

We can start by giving an upper bound for the descriptive complexity of the quasi-order of continuous reducibility on continuous functions with a compact Polish domain.
Given Polish spaces $X$ and $Y$, we endow the space $\C(X,Y)$\index[IdxSymb]{C@{$\C(X,Y)$}} of continuous functions from $X$ to $Y$ with the topology of compact convergence.
This topology coincides in this case with the compact-open topology (see \cite[Theorem 46.8]{munkres2000topology}) which is defined by taking as a subbase sets of the form
\[
S_{X,Y}(C,U)=\{f:X\to Y\mid f(C)\sub U\}\index[IdxSymb]{S@{$S_{X,Y}(C,U)$}}
\]
for some compact subset $C$ of $X$ and some open subset $U$ of $Y$.

When $X$ is locally compact and $Y$ is Polish, then $\C(X,Y)$ is Polish (see \cite[Theorem 1, p.93 and Theorem 3, p. 94]{kuratowski1968topology}),
while in the non-locally compact case this topology is not even metrizable.

When moreover $X$ is compact, then the topology on $\C(X,Y)$ coincides with the uniform convergence topology induced by the uniform metric
\[
d_u(f,g)=\sup_{x\in X} d_Y(f(x),g(x)),\index[IdxSymb]{du@{$d_u(f,g)$}}
\]
where $d_Y$ is some compatible metric for $Y$. In this case, $\C(X,Y)$ is Polish \cite[(4.19)]{kechris}
and a straightforward adaptation of \cite[Theorem 3.2]{embeddabilityCPZ} shows that the quasi-order of continuous reducibility is analytic:

\index[IdxNotion]{continuous reducibility quasi-order}
\begin{theorem}\label{RedContIsAnalytic}
Let $X$ and $Y$ be Polish with $X$ compact. Then continuous reducibility is an analytic quasi-order on $\C(X,Y)$.
\end{theorem}
		
\begin{proof}
We show that the relation $E$ on $\C(X,X)\times \C(X,Y)^2$ given by
\begin{align*}
(\sigma,f,g)\in E \Longleftrightarrow \text{$\sigma$ induces a reduction of $f$ to $g$}
\end{align*}
is Borel. Since $f\leq g$ if and only if $\exists \sigma \in \C(X,X) \ (\sigma,f,g)\in E$, this implies that the relation $f\leq g$ is analytic, as desired.

Fix a countable dense subset $D=\{d_n:n  \in \omega\}$ of $X$ and a compatible complete metric $d_Y$ on $Y$.
Recall from \cite[Theorem 3.2]{embeddabilityCPZ} that for every $x,x' \in X$ and $k>0$ the set $D_{x,x'}^k=\{f \in \C(X,Y):d_Y(f(x),f(x'))\leq \frac{1}{k}\}$ is closed in $\C(X,Y)$. 
We claim that $(\sigma,f,g)\in E$ if and only if
\[
\forall k>0 \ \exists l>0 \ \forall n,m \ \bigl(d_Y(g\sigma(d_n),g\sigma(d_m))\leq \frac{1}{l} \to d_Y(f(d_n),f(d_m))\leq \frac{1}{k}\bigr).
\]
This shows that $E$ is Borel. Suppose first that $(\sigma,f,g)\in E$. Then we have a unique continuous $\tau: \im g \sigma\to\im f $ such that $(\sigma,\tau)$ reduces $f$ to $g$. 
Since $\im g \sigma$ is compact, $\tau$ is uniformly continuous, and as $f(x)=\tau g \sigma(x)$ for all $x\in X$, the above formula is satisfied. 

Conversely suppose that $\sigma:X\to X$ is continuous such that the above formula holds with respect to some $f$ and $g$ in $\C(X,Y)$.
Note that, in particular, if $g \sigma(d_n)=g \sigma(d_m)$ then $f(d_n)=f(d_m)$ for every $m,n\in \omega$.
We therefore have a well defined map $\tilde{\tau}: g\sigma(D)\to f(D)$, given by $\tilde{\tau}(g\sigma(d))=f(d)$. 
Since $g\sigma(D)$ is dense in $\im g\sigma$ and $\tilde{\tau}$ is uniformly continuous, it extends uniquely to a continuous map $\tau: \im g\sigma\to \im f $.
Finally, since $ f$ and $\tau g\sigma$ are both continuous and coincide on $D$, they are equal on $X$.
Therefore $\tau$ is continuous and such that $(\sigma,\tau)$ continuously reduces $f$ to $g$, as desired.
\end{proof}

This proof does not extend in a straightforward way to the case where the domain is not compact. Since $\C(X,Y)$ has a standard Borel structure making the evaluation map Borel if and only if $X$ is $\sigma$-compact~\cite[Theorem 3.9]{embeddabilityCPZ}, we have the following question:

\begin{question}
Suppose that $X$ and $Y$ are Polish and $X$ is $\sigma$-compact, what is the descriptive complexity of continuous reducibility on $\C(X,Y)$?
\end{question}

In some cases we know that continuous reducibility is not \wqo{}:
by \cite[Theorem 4.5]{Louveau2005complete} embeddability between closed subsets of $[0,1]^2$ is analytic complete;
meaning that it is analytic and it (continuously) reduces every analytic quasi-order, in particular it reduces closed ones like the identity quasi-order
or the lexicographic (linear) order on $\cantor$, and embeddability thus contains antichains and strictly descending chains of size continuum.
So it is as far from being a \bqo{} as an analytic quasi-order possibly can.

\cref{LinkswithWadgeEmbedMeas} tells us that the map $X\mapsto\id_{X}$ is a reduction\footnote{But it does not tell us if it is continuous!}
from embeddability on spaces to continuous reduction on functions, so we directly get

\begin{proposition}\label{analyticcomplete}
Let $\mathcal{C}$ be a class of functions containing all functions $\id_{F}$ for $F$ closed in $[0,1]^2$.
Then continuous reducibility is analytic hard on $\mathcal{C}$. In particular, it is not \wqo{}.
\end{proposition}

So continuous reducibility is analytic hard for classes of functions with domain included $[0,1]^2$, while it is \wqo{} on continuous functions with domain included $\cantor$.
What about functions with domains that do not fall in any of these two cases? In particular

\begin{question}
Is continuous reducibility \wqo{} on continuous real functions?
\end{question}

The same question is open for continuous functions on any uncountable Polish space that does not contain $[0,1]^2$. In fact, the following question is already open:

\begin{question}
Let $X$ be a Polish space containing $\cantor$ but not $[0,1]^2$.
What is the complexity of embeddability on closed subsets of $X$?
\end{question}

\subsection{The Perfect Function Property}\label{DirectionPerfect}

Let us call PFP the Perfect Function Property for short.
The general question is

\index[IdxNotion]{Perfect Function Property}
\begin{question}\label{PFPquestion}
What are the classes of functions satisfying the PFP?
\end{question}

We know the answer for a few classes.
If we drop the separability assumption on the co-domain, we have the following counter-example.

\begin{theorem}\label{counterPFP1}
There is a function $f:\cantor\to\omega_1$ without the Perfect Function Property.
\end{theorem}

\begin{proof}
Consider the function $f:\cantor\to\omega_1$, $x\mapsto \omega_1^{\text{CK}-x}$ which maps each $x$ to the set of all countable ordinal recursive in $x$\footnote{See \cite{ChongYuRecursion} for classical results and \cite{BorelGraphableAKL} for more recent results about this notion.}.
It has uncountable image, so suppose towards a contradiction that $(\sigma,\tau)$ reduces $\id_{\cantor}$ to $f$.
In particular, there is $\sigma\in\C(\cantor, \cantor)$ such that for all $x,y\in\cantor$, if  $x\neq y$ then we must have $\omega_1^{\text{CK}-\sigma(x)}\neq\omega_1^{\text{CK}-\sigma(y)}$.
Recall that the relation $a\leq_{\text{CK}}b$ if and only if $\omega_1^{\text{CK}-a}\leq\omega_1^{\text{CK}-b}$ is an analytic pre-well-order on $\cantor$.
Therefore $\sigma^{-1}(\leq_{\text{CK}})$ is an analytic well-order of $\cantor$ which must hence have the Baire Property, thus contradicting \cite[(8.48)]{kechris}.
\end{proof}

There is a more metamathematical proof of \cref{counterPFP1}:
observe first that the statement $\id_{\cantor}\leq f$ is $\mathbf{\Sigma}^1_2$ and it gives an injection from $\cantor$ to $\omega_1$,
so it is generically absolute by Schoenfield's Absoluteness result (see \cite[Theorem 13.15]{kanamori2008higher}), and this is a contradiction as the Continuum Hypothesis is not.

Assuming the Axiom of Choice (AC), one can embed $\omega_1$ in $\cantor$ and get the previous example with a separable co-domain.
Consistently, we can also have a bit more.

\index[IdxNotion]{Baire-measurable}
\begin{theorem}\label{counterPFP2}
Working in $ZFC$, suppose that the additivity of the meager ideal is $\aleph_2$.

Then there is a Baire-measurable $f:\cantor\to\cantor$ without the Perfect Function Property.
\end{theorem}

\begin{proof}
Use AC to get an injection $\iota:\omega_1\to\cantor$, and consider $f:\cantor\to\cantor$, $x\mapsto \iota(\omega_1^{\text{CK}-x})$.
The image of $f$ has size $\aleph_1$, so the preimage of any set by $f$ is the union of at most $\aleph_1$-many preimages of singletons, that are all analytic, so Baire-measurable.
The additivity of the meager ideal being $\aleph_2$ implies in particular that Baire-measurable sets are closed under union of size less or equal than $\aleph_1$.
Therefore $f$ is Baire-measurable, and $f$ does not have the Perfect Function Property by \cref{counterPFP1}.
\end{proof}

\cref{counterPFP1,counterPFP2} suggest to extend the parallel between embeddability on linear orders and continuous reducibility on functions.
We can thus ask about the functions counterparts of Moore Five Elements Basis result and anti-basis result for embeddability on uncountable linear orders \cite{MR2199228,MR2369944}.

\begin{question}
Is there (even consistently) a finite basis with respect to continuous reducibility for functions with uncountable image?

Is there (even consistently) an antibasis result for continuous reducibility on functions with uncountable image?
\end{question}

There are some positive answers to \cref{PFPquestion}:
\cref{uncountablerange} proves that the class of Borel functions from an analytic to a separable metrizable space satisfies the PFP,
 and thanks to their game-theoretical proof, Lutz and Siskind in \cite{lutz20231} have proven that, under the Axiom of Determinacy (AD), all functions from $\cantor$ to itself have the PFP.
 
Using the Open Graph Dichotomy (another consequence of AD) for all open graphs on separable metrizable spaces -- OGD for short (see for instance \cite{carroymillersoukup}) one can prove

\index[IdxNotion]{Open Graph Dichotomy}
\begin{proposition}
Assume OGD.
Then all continuous functions between separable metrizable spaces have the PFP.
\end{proposition}

\begin{proof}
Given $X,Y$ separable metrizable and $f:X\to Y$ continuous, consider the graph $G$ of all $(x,x')$ such that $f(x)\neq f(x')$. The continuity of $f$ ensures that $G$ is open in $X\times X$.
Now if $G$ has countable chromatic number then $f$ has countable image, and otherwise by OGD there is a continuous $\sigma:\cantor\to X$ such that
$x\neq x'$ implies $(\sigma(x),\sigma(x'))\notin G$, meaning in particular $f\sigma(x)\neq f\sigma(x')$. Conclude using \cref{EquivEmbedCantor}.
\end{proof}

We have seen that PFP implies the Perfect Set Property, but the three results we just mentioned suggest that there might be an equivalence:

\begin{question}
Does the Perfect Set Property implies the Perfect Function Property?
\end{question}

Since combining PFP with our \cref{Introthm:BQOonScat} allows for a \bqo{} result following the proof of \cref{FirststepforBQOthm}, we get

\begin{corollary}
Assume OGD. Continuous reducibility is \bqo{} on the class of all continuous functions from a zero-dimensional separable metrizable space to a metrizable space. 
\end{corollary}

\bibliographystyle{alpha}
\bibliography{_biblio.bib}

\printindex[IdxNotion]
\printindex[IdxSymb]
\printindex[IdxFig]


\end{document}